\documentclass[a4paper,12pt,reqno]{amsart}

\usepackage[T1]{fontenc}
\usepackage[utf8]{inputenc}
\usepackage[english]{babel}
\usepackage[left=25mm,right=25mm,top=35mm,bottom=35mm]{geometry}
\usepackage{amsmath,amssymb,amsthm}
\usepackage[margin=0pt]{caption}
\usepackage[exponent-product = \cdot]{siunitx}
\usepackage{pgfplotstable}
\pgfplotsset{
tick label style = {font = \tiny},
legend style = {font = \scriptsize},
xlabel style={yshift=+0.5ex},
ylabel style={yshift=-0.5ex}
}
\usepackage{algorithmic}
\usepackage{subcaption}
\usepackage{float}
\usepackage[hidelinks]{hyperref}
\newcommand\0{\boldsymbol{0}}
\newcommand{\eps}{\varepsilon}

\newcommand{\y}{\mathbf{y}}
\newcommand{\N}{\mathbb{N}}
\renewcommand{\P}{\mathbb{P}}
\newcommand{\R}{\mathbb{R}}

\newcommand{\V}{\boldsymbol{\mathbb{V}}}
\newcommand{\X}{\mathbb{X}}

\newcommand{\MM}{\mathcal{M}}
\newcommand{\NN}{\mathcal{N}}

\newcommand{\RR}{\mathcal{R}}
\renewcommand{\SS}{\mathcal{S}}
\newcommand{\TT}{\mathcal{T}}
\newcommand\III{\mathfrak{I}}
\newcommand\MMM{\mathfrak{M}}
\newcommand\PPP{\mathfrak{P}}
\newcommand\QQQ{\mathfrak{Q}}

\newcommand{\qsat}{q_{\mathrm{sat}}}

\DeclareMathOperator*{\level}{level}
\DeclareMathOperator*{\hull}{span}
\DeclareMathOperator*{\refine}{refine}
\DeclareMathOperator*{\supp}{supp}
\newcommand{\enorm}[3][]{#1|\!#1|\!#1|\,#2\,#1|\!#1|\!#1|_{#3}}
\newcommand{\bnorm}[2][]{#1|\!#1|\!#1|\,#2\,#1|\!#1|\!#1|}

\newcommand{\norm}[3][]{#1\|#2#1\|_{#3}}
\newcommand\Cest{C_{\rm est}}
\renewcommand\aa{\boldsymbol{a}}
\newcommand\ee{\boldsymbol{e}}
\newcommand\ff{\boldsymbol{f}}
\newcommand\uu{\boldsymbol{u}}
\newcommand\vv{\boldsymbol{v}}
\newcommand\ww{\boldsymbol{w}}
\newcommand\coarse{\bullet}
\newcommand\fine{\circ}
\renewcommand\d{\mathrm{d}}

\newcommand\sfA{\boldsymbol{\sf A}}
\newcommand\sfu{\boldsymbol{\sf u}}
\newcommand\sfb{\boldsymbol{\sf b}}
\newcommand\sfx{\boldsymbol{\sf x}}
\newcommand\CY{C_Y}
\newcommand\KY{K}
\newcommand\Cstab{C_{\rm stb}}
\newcommand\Cloc{C_{\rm loc}}
\newcommand{\dual}[3][]{#1\langle#2\,,\,#3#1\rangle_{D}}
\newcommand\G{\mathbb{G}}
\newcommand\deterministic{\star}
\newcommand\est{\tau}
\newcommand{\UU}{\mathcal{U}}
\def\set#1#2{\big\{#1 \,:\, #2\big\}}
\def\reff#1#2{\stackrel{\eqref{#1}}{#2}}
\def\refp#1#2{\stackrel{\phantom{\eqref{#1}}}{#2}}
\newtheorem{theorem}{Theorem}

\newtheorem{lemma}[theorem]{Lemma}

\newtheorem{algorithm}[theorem]{Algorithm}
\newtheorem{remark}[theorem]{Remark}
\newtheorem{marking}{Marking criterion}

\usepackage{fancyhdr}
\advance\footskip0.5cm
\makeatletter
\def\@seccntformat#1{%
  \protect\textup{\protect\@secnumfont
    \hspace*{5mm}\ifnum\pdfstrcmp{subsection}{#1}=0 \bfseries\fi
    \csname the#1\endcsname
    \protect\@secnumpunct
  }%
}
\makeatother
\title{Two-level a~posteriori error estimation for\\ adaptive multilevel stochastic Galerkin FEM}
\author{Alex Bespalov}
\address{School of Mathematics, University of Birmingham, Edgbaston, Birmingham B15 2TT, UK}
\email{a.bespalov@bham.ac.uk}
\author{Dirk Praetorius}
\address{Institute of Analysis and Scientific Computing, TU Wien, Wiedner Hauptstra\ss{}e~8--10, 1040 Vienna, Austria}
\email{dirk.praetorius@asc.tuwien.ac.at}
\author{Michele Ruggeri}
\address{Institute of Analysis and Scientific Computing, TU Wien, Wiedner Hauptstra\ss{}e~8--10, 1040 Vienna, Austria}
\email{michele.ruggeri@asc.tuwien.ac.at}
\date{\today}
\subjclass[2010]{35R60, 65C20, 65N15, 65N30, 65N50}
\keywords{adaptive methods, a~posteriori error analysis, two-level error estimation,
multilevel stochastic Galerkin method, finite element method, parametric PDEs}
\thanks{\emph{Acknowledgments.}
The work of the first author was supported by the EPSRC under grant EP/P013791/1
and by The Alan Turing Institute under the EPSRC grant EP/N510129/1.
The work of the second and third authors was supported by the Austrian Science Fund (FWF) under grants F65 and P33216.
The authors are grateful to David Silvester (University of Manchester) for useful discussions and
advice on the implementation of iterative solvers in multilevel stochastic Galerkin FEM}

\begin{document}

\begin{abstract}
The paper considers a class of parametric elliptic partial differential equations (PDEs),
where the coefficients and the right-hand side function depend on infinitely many (uncertain) parameters.
We introduce a two-level \textsl{a~posteriori} estimator to control the energy error in
multilevel stochastic Galerkin approximations for this class of PDE problems.
We prove that the two-level estimator always provides a lower bound for the unknown approximation error,
while the upper bound is equivalent to a saturation assumption.
We propose and empirically compare three adaptive algorithms,
where the structure of the estimator is exploited to perform spatial refinement as well as parametric enrichment.
The paper also discusses implementation aspects of computing multilevel stochastic Galerkin approximations.
\end{abstract}

\maketitle
\thispagestyle{fancy}

\section{Introduction} \label{sec:intro}

\subsection{Multilevel stochastic Galerkin FEM}

The effective numerical solution of partial differential equations (PDEs) with uncertain or parameter-dependent inputs requires
non-trivial computational methods and efficient algorithms.
Stochastic Galerkin finite element methods (SGFEMs) provide a powerful alternative to traditional sampling techniques
for such problems, in particular, when the inputs and solutions are sufficiently smooth functions of parameters
(for comparison between SGFEM and popular sampling methods,
such as Monte Carlo and stochastic collocation finite element methods,
we refer to~\cite{gwz2014,bntt2011,gllmn2014}).
Appropriate construction of the underlying approximation spaces and adaptivity
are the keys to computationally efficient SGFEM implementations,
particularly in the case of inputs depending on infinitely many uncertain parameters.

Stochastic Galerkin approximations are typically represented in terms of a finite
generalized polynomial chaos (gPC) expansion
with spatial coefficients residing in finite element spaces.
If all spatial coefficients reside in the same finite element space,
the corresponding SGFEM approximation space is termed \emph{single-level} and its dimension
has a multiplicative representation
(i.e., the total number of degrees of freedom is equal to
the number of active terms in the gPC expansion multiplied by the dimension of the finite element space).
An alternative to this is a more flexible \emph{multilevel} construction,
where spatial gPC-coefficients may reside in different finite element spaces.
In this case, the dimension of the SGFEM approximation space admits an additive representation
(i.e., the total number of degrees of freedom is equal to the sum of dimensions of all involved finite element spaces).

\emph{Multilevel} SGFEMs have emerged in~\cite{CohenDeVoreSchwab10,cds11,Gittelson_13_CRM}.
These works have provided a theoretical benchmark for convergence analysis of the SGFEM.
In particular, under some assumptions on parametric inputs,
they have proved the existence of a sequence of multilevel approximation spaces
such that the errors in the associated Galerkin solutions converge to zero with an optimal rate
(i.e., with the rate of the chosen FEM for the corresponding parameter-free problem).
Practical realizations of adaptive algorithms that generate these sequences
of approximation spaces and Galerkin solutions have been developed in~\cite{egsz14}
and more recently in~\cite{cpb18+}.
While the predicted optimal convergence behavior of multilevel SGFEM approximations
has been observed numerically for parametric problems with spatially regular~\cite{cpb18+}
and spatially singular~\cite{egsz14} solutions,
a provable optimality result for the developed adaptive algorithms for multilevel SGFEMs has so far remained an open problem.

Multilevel approaches
based on hierarchies of spatial approximations have been studied
also for sampling methods.
Remaining within the context of the numerical approximation of elliptic PDEs with uncertain data,
we refer, e.g., to~\cite{cgst2011} for multilevel Monte Carlo (MLMC) methods,
to~\cite{kss2015} for multilevel quasi-Monte Carlo methods,
and to~\cite{tjwg2015} for multilevel stochastic collocation (MLSC) methods.
Adaptive strategies for MLMC and MLSC
have been developed
recently in~\cite{ky2018} and~\cite{LangSS20},~respectively.

\subsection{Main contributions and outline of the paper} \label{sec:contribute}

In this paper, we consider the same parametric model problem as studied in the above cited works~\cite{egsz14,cpb18+}
(among very many other works)---the steady-state diffusion equation with a spatially varying coefficient
that has affine dependence on infinitely many parameters.

For the numerical solution of this problem,
we propose an adaptive algorithm
that iterates the following loop of four modules:
\begin{equation*}
{\sf SOLVE} \longrightarrow {\sf ESTIMATE}  \longrightarrow {\sf MARK}  \longrightarrow {\sf REFINE}
\end{equation*}
(see Algorithm~\ref{algorithm} below).
Let us briefly describe each of these modules emphasizing their specific features pertinent
to the multilevel SGFEM.

$\bullet$ \textsf{SOLVE}:
In this module,
the multilevel SGFEM approximation
is computed as a finite gPC expansion with coefficients in the current set of finite element spaces.
One of the challenges in implementing multilevel SGFEMs is the efficient computation
of nonsquare stiffness matrices associated with two different finite element meshes.
The existing implementations either rely on projection techniques to compute these stiffness matrices approximately
(see~\cite{egsz14,alea}) or restrict themselves to spatial discretizations on nested uniform meshes (see~\cite{cpb18+}).
In this paper, we propose an effective procedure for \emph{direct} computation of nonsquare stiffness matrices
for a pair of general, not necessarily nested meshes
obtained from the same coarse mesh by finitely many steps of a fixed mesh refinement rule
(in our case, newest vertex bisection).
SGFEMs give rise to very large linear systems
with block structure.
Solving such linear systems numerically is a non-trivial task that stimulated
the development and analysis of iterative solvers,
preconditioning strategies, and low-rank approximation techniques;
see, e.g., \cite{ghanemkruger96,elmanpowell2009,ullmann2010,sg2014,ehlmw2014,dklm2015,bly2021}.
In order to solve the linear systems arising in the multilevel SGFEM,
we use a bespoke implementation of the Minimum Residual method from~\cite{ss11}
with the mean-based preconditioner from~\cite{ghanemkruger96,elmanpowell2009}.

$\bullet$ \textsf{ESTIMATE}:
In this module, the error between the (unknown) exact solution and the multilevel SGFEM approximation
is estimated by suitable error indicators.
The \textsl{a~posteriori} error estimation in multilevel SGFEMs has been addressed in~\cite{egsz14,cpb18+}.
While explicit residual-based error estimators are employed in~\cite{egsz14},
hierarchical-type error estimators are analyzed in~\cite{cpb18+}.
Building on the ideas in our recent works for a single-level SGFEM~\cite{bprr18+,bprr18++},
in this paper,
we propose a novel \textsl{a~posteriori} error estimation strategy for multilevel SGFEM approximations.
The estimator, which combines a two-level spatial estimator and a hierarchical parametric estimator,
allows to estimate
the error contributions from finite element discretizations in the physical domain
and those from the approximation (obtained via truncation) in the parameter domain independently from each other.
We prove that the combined error estimator is always efficient,
i.e., up to a multiplicative constant, it provides a lower bound for the energy error,
whereas its reliability (i.e., the upper bound for the error) is equivalent to a saturation assumption
(see subsection~\ref{sec:saturation} below).
This choice of the error estimation strategy is motivated by a recent success in proving
optimal convergence rates for adaptive algorithms for deterministic problems;
see~\cite{PRS20}.
Thus, we see our \textsl{a~posteriori} error analysis in this paper as an important step towards proving
the optimality result for adaptive multilevel SGFEM approximations by extending the methodology
developed in~\cite{PRS20} to the parametric~setting.

$\bullet$ \textsf{MARK}:
In this module,
some of the spatial and parametric components of the current multilevel SGFEM approximation
are selected for refinement
by assessing the values of the error indicators computed in the module \textsf{ESTIMATE}.
The application of the module \textsf{MARK} highlights key differences between adaptive multilevel SGFEM and
multilevel sampling methods (such as MLMC and MLSC).
The latter methods typically require the number of active parameters in approximations to be fixed \textsl{a priori},
and the balance between
the spatial error (e.g., due to finite element discretization) and
the parametric error (e.g., due to Monte Carlo sampling or high-dimensional polynomial interpolation)
is achieved by employing \textsl{a priori} bounds for spatial errors and by minimizing the cost functional
(see, e.g.,~\cite[section~3.4]{gwz2014} for MLMC and~\cite{tjwg2015} for MLSC).
Adaptive SGFEM algorithms are fundamentally different.
Firstly, they require no sampling.
Secondly,
the selection of active parameters, the truncation of the gPC expansion,
and the balance between
spatial and parametric components of approximation errors
are performed \emph{automatically}
using \textsl{a~posteriori} error indicators and
the adopted marking criterion.
The choice of the marking criterion is critical.
In this work, we propose three different marking strategies, all based
on the bulk-chasing criterion proposed in the deterministic setting by D\"orfler \cite{doerfler}.
In addition to two standard marking criteria that lead to separate refinement of either spatial or parametric components
at each iteration of the adaptive loop (see, e.g., \cite{egsz14,egsz15,bs16,bprr18++,cpb18+}),
we also exploit the multilevel structure of the approximation space and perform a \emph{combined} refinement at each
iteration by employing D\"orfler marking on the joint set of spatial and parametric error indicators.
While combined refinement is prohibitively expensive for single-level SGFEM
(because of the multiplicative dependence of the dimension of the discrete space on the number of
active terms in the gPC expansion),
we stress that multilevel SGFEM allows for combined refinement and our experiments indicate optimal convergence behavior.

$\bullet$ \textsf{REFINE}:
In this module,
the finite-dimensional space for computing the next multilevel SGFEM approximation is generated by
enriching the current finite-dimensional space with
the spatial and parametric components selected in
the module \textsf{MARK}.
Specifically, (i) the finite element spaces are enriched by refining
all marked elements of the current spatial meshes;
and (ii) new terms are added to the gPC expansion.

The paper is organized as follows.
Section~\ref{sec:problem} introduces the model parametric problem and its weak formulation.
In section~\ref{sec:mlsgfem}, we describe the main ingredients of the multilevel SGFEM discretization,
introduce the multilevel approximation space, and define the corresponding Galerkin solution.
Section~\ref{sec:aposteriori} is focused on the \textsl{a~posteriori} error analysis of multilevel SGFEM approximations
and includes the main theoretical result of this paper (Theorem~\ref{thm:estimator}).
Adaptive algorithms with three different marking criteria are formulated in section~\ref{sec:algorithms},
whereas implementation aspects of computing multilevel SGFEM approximations are discussed in section~\ref{sec:implement}.
The effectiveness of our error estimation strategy and the performance of the proposed adaptive algorithms
are assessed in a series of numerical experiments presented in section~\ref{sec:numerics}.

\section{Problem formulation} \label{sec:problem}

Let $D \subset \R^d$ ($d = 2, 3$) be a bounded Lipschitz domain with polytopal boundary $\partial D$
and let $\Gamma := \prod_{m=1}^\infty [-1,1]$ denote the infinitely-dimensional hypercube.
We consider the elliptic boundary value problem
\begin{equation} \label{eq:strongform}
\begin{aligned}
 -\nabla \cdot (\aa \nabla \uu) &= \ff \quad &&\text{in } D \times \Gamma,\\
 \uu & = 0 \quad &&\text{on } \partial D \times \Gamma,
\end{aligned}
\end{equation}
where the scalar coefficient $\aa$ and the right-hand side function $\ff$ (and, hence, the solution~$\uu$)
depend on a countably infinite number of scalar parameters, i.e.,
$\aa = \aa(x, \y)$, $\ff = \ff(x, \y)$, and $\uu = \uu(x, \y)$ with $x \in D$ and $\y = (y_m)_{m\in\N} \in \Gamma$.
For the coefficient $\aa$, we assume linear dependence on the parameters, i.e.,
\begin{equation} \label{eq1:a}
\aa(x,\y) = a_0(x) + \sum_{m=1}^\infty y_m a_m(x)
\quad \text{for all } x \in D \text{ and } \y \in \Gamma.
\end{equation}
We assume that
$\ff \,{\in}\, {L^2_\pi(\Gamma;H^{-1}(D))}$,
where $\pi \,{=}\, \pi(\y)$ is a measure on $(\Gamma,\mathcal{B}(\Gamma))$ with
$\mathcal{B}(\Gamma)$ being the Borel $\sigma$-algebra on $\Gamma$.
We assume that $\pi(\y)$ is the product of symmetric Borel probability measures $\pi_m$ on~$[-1,1]$,
i.e., $\pi(\y) = \prod_{m=1}^\infty \pi_m(y_m)$.

For each $m \in \N_0$, the scalar functions $a_m \in L^{\infty}(D)$
in~\eqref{eq1:a} are required to satisfy the following inequalities (cf.~\cite[Section~2.3]{sg11}):
\begin{gather}
\label{eq2:a}
 0 < a_0^{\rm min} \le a_0(x) \le a_0^{\rm max} < \infty
 \quad \text{for almost all } x \in D,\\
\label{eq3:a}
 \tau := \frac{1}{a_0^{\rm min}} \, \bigg\| \sum_{m=1}^\infty |a_m| \bigg\|_{L^\infty(D)} < 1
 \text{\ \ and\ \ }
 \sum_{m=1}^\infty \norm{a_m}{L^\infty(D)} < \infty.
\end{gather}
With the Sobolev space $\X := H^1_0(D)$, consider the Bochner space $\V := L^2_\pi(\Gamma;\X)$.
Define the following bilinear forms on $\V$:
\begin{align}\label{def:B0}
 B_0(\uu,\vv) &:= \int_\Gamma \int_D a_0(x) \nabla \uu(x,\y) \cdot \nabla \vv(x,\y) \, \d{x} \, \d{\pi(\y)},
 \\\label{def:B}
 B(\uu,\vv) &:= B_0(\uu,\vv)
 + \sum_{m=1}^\infty \int_\Gamma \int_D y_ma_m(x) \nabla \uu(x,\y) \cdot \nabla \vv(x,\y) \, \d{x} \, \d{\pi(\y)}.
\end{align}
An elementary computation shows that assumptions~\eqref{eq1:a}--\eqref{eq3:a} ensure that the bilinear 
forms $B_0(\cdot,\cdot)$ and $B(\cdot,\cdot)$ are symmetric, continuous, and elliptic on $\V$.
Let $\bnorm{\cdot}$ (resp., $\enorm{\cdot}{0}$) denote the norm induced by $B(\cdot,\cdot)$ (resp., $B_0(\cdot,\cdot)$).
Then, there holds
\begin{equation}\label{eq:lambda}
 \lambda \, \enorm{\vv}{0}^2 \le \bnorm{\vv}^2 \le \Lambda \, \enorm{\vv}{0}^2
 \quad \text{for all } \vv \in \V,
\end{equation}
where
$\lambda := 1-\tau$ and $\Lambda := 1+\tau$. Note that $0 < \lambda < 1 < \Lambda < 2$.

The parametric problem~\eqref{eq:strongform} is understood in the weak sense:
Given $\ff \,{\in}\, L^2_\pi(\Gamma;H^{-1}(D))$, find $\uu \in \V$ such that
\begin{equation} \label{eq:weakform}
B(\uu,\vv) = F(\vv) := \int_\Gamma \int_D \ff(x,\y) \vv(x,\y) \, \d{x} \, \d{\pi(\y)}
\quad \text{for all } \vv \in \V.
\end{equation}
The existence and uniqueness of the solution $\uu \in \V$ to~\eqref{eq:weakform} follow by the Riesz theorem.

\section{Multilevel stochastic Galerkin FEM discretization} \label{sec:mlsgfem}

The weak formulation~\eqref{eq:weakform} is discretized by constructing a finite-dimensional subspace
$\V_\coarse \subset \V$ and using the Galerkin projection onto $\V_\coarse$.
In the spirit of~\cite{CohenDeVoreSchwab10,Gittelson_13_CRM,egsz14,cpb18+},
this work considers approximation spaces $\V_\coarse$ with a \emph{multilevel} structure.
Specifically, these spaces are constructed from tensor products of different finite element subspaces of $H^1_0(D)$
and multivariable polynomial spaces on $\Gamma$.
We describe each of these ingredients in the next two subsections.

\subsection{Finite element spaces and mesh refinement} \label{sec:spatial:approx}

Let $\TT_\coarse$ be a \emph{mesh}, i.e., a conforming triangulation of $D$ into
compact non-degenerate simplices $T \in \TT_\coarse$ (i.e., triangles for $d = 2$)
and denote by $\NN_\coarse$ the set of vertices of $\TT_\coarse$. 

We consider the space of continuous piecewise linear finite elements
\begin{equation*}
 \X_\coarse := \SS^1_0(\TT_\coarse) := \{v_\coarse \in \X : v_\coarse \vert_T \text{ is affine for all } T \in \TT_\coarse \} \subset \X = H^1_0(D).
\end{equation*}
For $z \in \NN_\coarse$, let $\varphi_{\coarse,z}$
be the associated hat function, i.e., $\varphi_{\coarse,z}$ is piecewise affine, globally continuous,
and satisfies the Kronecker property $\varphi_{\coarse,z}(z') = \delta_{zz'}$ for all $z' \in \NN_\coarse$.
Recall that $\{ \varphi_{\coarse,z} : z \in \NN_\coarse \setminus \partial D \}$ is the standard basis of $\X_\coarse$.

For mesh refinement, we employ newest vertex bisection (NVB);
see, e.g., \cite{stevenson,kpp}.
We assume that any mesh $\TT_\coarse$ employed for the spatial discretization
can be obtained by applying NVB refinement(s) to a given initial (coarse) mesh $\TT_0$.
In particular, we denote by $\refine(\TT_0)$ the set of all meshes obtained from $\TT_0$
by finitely many steps of refinement.

For a given mesh $\TT_\coarse \in \refine(\TT_0)$,
let $\widehat\TT_\coarse$ be the coarsest NVB refinement of $\TT_\coarse$ such that:
(i) for $d=2$, all edges of $\TT_\coarse$ have been bisected once
(which corresponds to uniform refinement of all elements by three bisections);
(ii) for $d = 3$, all faces contain an interior node
(we refer to~\cite{egp18+} for further discussion).
Then, $\widehat\NN_\coarse$ denotes the set of vertices of $\widehat\TT_\coarse$,
and $\NN_\coarse^+ := (\widehat\NN_\coarse \setminus \NN_\coarse) \setminus \partial D$
is the set of new interior vertices created by this refinement of $\TT_\coarse$.

For a set of marked vertices $\MM_\coarse \subseteq \NN_\coarse^+$,
let $\TT_\fine := \refine(\TT_\coarse,\MM_\coarse)$ be the coarsest NVB refinement of $\TT_\coarse$ such that
$\MM_\coarse \subset \NN_\fine$, i.e., all marked vertices are vertices of $\TT_\fine$.
Since NVB is a binary refinement rule, this implies that $\NN_\coarse \subseteq \NN_\fine \subseteq \widehat\NN_\coarse$ and $(\NN_\fine \setminus \NN_\coarse) \setminus \partial D = \NN_\coarse^+ \cap \NN_\fine$.
In particular, the choices $\MM_\coarse = \emptyset$ and $\MM_\coarse = \NN_\coarse^+$ lead to the meshes
$\TT_\coarse = \refine(\TT_\coarse,\emptyset)$ and $\widehat\TT_\coarse = \refine(\TT_\coarse,\NN_\coarse^+)$, respectively. 

The finite element space associated with $\widehat\TT_\coarse$ is denoted by
$\widehat\X_\coarse := \SS^1_0(\widehat\TT_\coarse)$,
and
$\{ \widehat\varphi_{\coarse,z} : z \in \widehat\NN_\coarse \setminus \partial D \}$ is the corresponding basis of hat functions.
Later, we will exploit the ($H^1$-stable) two-level decomposition
$\widehat\X_\coarse = \X_\coarse \oplus \hull\{ \widehat\varphi_{\coarse,z} : z \in \NN_\coarse^+ \}$.

We note that there exist two constants $K,K' \ge 1$ depending only on the initial mesh $\TT_0$
such that
\begin{equation} \label{eq:const:K}
   \# \set{z \in \NN_{\coarse}^+ }{|T \cap \supp(\widehat\varphi_{\coarse,z})|>0} \le K < \infty
   \quad \hbox{for all $T \in \TT_{\coarse}$}
\end{equation}
and
\begin{equation} \label{eq:const:K'}
   \# \{T \in \TT_{\coarse} : |T \cap \supp(\widehat\varphi_{\coarse,z})|>0 \} \le K' < \infty 
   \quad \hbox{for all $z \in \NN_{\coarse}^+$},
\end{equation}
with $K = 3$ and $K' = 2$ for $d = 2$.

\subsection{Polynomial spaces on $\Gamma$ and parametric enrichment} \label{sec:param:approx}

First, we introduce the polynomial spaces on~$\Gamma$.
For each $m \in \N$, let $(P_n^m)_{n\in\N_0}$ denote the sequence of univariate polynomials
which are orthogonal with respect to $\pi_m$ such that $P_n^m$ is a polynomial of degree
$n \in \N_0$ with $\norm{P_n^m}{L^2_{\pi_m}(-1,1)}=1$ and $P_0^m \equiv 1$.
For convenience, we also define $P_{-1}^m \equiv 0$ and, for each $n \in \N_0 \cup \{ -1 \}$,
we denote by $c_n^m$ the leading coefficient of $P_n^m$.
It is well-known that $\{P_n^m : n\in\N_0 \}$ is an orthonormal basis of $L^2_{\pi _m}(-1,1)$.
Moreover, there holds the three-term recurrence formula
\begin{equation}\label{eq:recursion}
\beta_n^m \, P_{n+1}^m(y_m)
= y_m \, P_n^m(y_m) - \beta_{n-1}^m \, P_{n-1}^m(y_m)
\quad \text{for all } y_m \in [-1,1] \text{ and } n\in\N_0,
\end{equation}
where $\beta_{n-1}^m = c_{n-1}^m / c_n^m$.
With $\N_0^\N := \{\nu = (\nu_m)_{m\in\N} : \nu_m\in\N_0 \text{ for all } m\in\N \}$
and $\supp(\nu):= \{m\in\N : \nu_m\neq 0 \}$, let $ \III := \{\nu \in \N_0^{\N} : \#\supp(\nu) < \infty \}$
be the set of all finitely supported multi-indices.
Note that $\III$ is countable. With
\begin{equation*}
P_\nu(\y)
:= \prod_{m\in\N} P_{\nu_m}^m(y_m)
= \prod_{m\in\supp(\nu)} P_{\nu_m}^m(y_m)
\quad \text{for all } \nu \in \III \text{ and all } \y \in \Gamma,
\end{equation*}
the set $\{ P_\nu : \nu\in\III \}$ is an orthonormal basis of $\P := L^2_\pi(\Gamma)$; see~\cite[Theorem~2.12]{sg11}.

For any $m \in \N$, let $\eps_m \in \III$ be the $m$-th unit sequence, i.e., $(\eps_m)_i = \delta_{mi}$ for all $i \in \N$.
A consequence of the three-term recurrence formula~\eqref{eq:recursion} is the identity
\begin{equation} \label{eq:3005:three-term}
y_m P_\mu(\y) = \beta_{\mu_m}^m P_{\mu+\eps_m}(\y) + \beta_{\mu_m-1}^m P_{\mu-\eps_m}(\y)
\quad \text{for all } \mu \in \III, \text{ } \y \in \Gamma, \text{ and } m \in \N.
\end{equation}

Note that the Bochner space $\V = L^2_\pi(\Gamma;\X)$ is isometrically isomorphic to $\X \otimes \P$ and
each function $\vv \in \V$ can be represented in the form
\begin{equation}\label{eq:representation}
 \vv(x,\y) = \sum_{\nu \in \III} v_\nu(x) P_\nu(\y)
 \quad\text{with unique coefficients}\
 v_\nu \in \X.
\end{equation}
Moreover, there holds (see, e.g.,~\cite[Lemma~2.1]{bprr18+})
\begin{equation} \label{eq1:lemma:orthogonal}
B_0(\vv, \ww) = \sum_{\nu \in \III} \int_D a_0(x) \, \nabla v_\nu(x) \cdot \nabla w_\nu(x) \, \d{x}
\quad \text{for all } \vv, \ww \in \V
\end{equation}
and, in particular,
\begin{equation}
\label{eq2:lemma:orthogonal}
\enorm{\vv}{0}^2 
= \sum_{\nu \in \III} \enorm{v_\nu P_\nu}{0}^2
= \sum_{\nu \in \III} \norm{a_0^{1/2}\nabla v_\nu}{L^2(D)}^2
\quad \text{for all } \vv \in \V.
\end{equation}

Let $\0 = (0,0,\dots)$ denote the zero index,
and let $\PPP_\coarse \subset \III$ be a finite index set such that $\0 \in \PPP_\coarse$.
We denote by $\supp(\PPP_\coarse) := \bigcup_{\nu \in \PPP_\coarse} \supp(\nu)$ the set of active parameters in~$\PPP_\coarse$.
Turning now to the parametric enrichment,
we follow the same construction as in~\cite{bs16,br18,bprr18+,bprr18++}.
For a fixed $\overline{M} \in \N$, we consider the \emph{detail index~set}
\begin{equation} \label{def:Q}
   \QQQ_\coarse := \{ \mu \in \III \setminus \PPP_\coarse : \mu = \nu \pm \eps_m
                                 \text{ for all } \nu \in \PPP_\coarse \text{ and all } m = 1,\dots, M_{\PPP_\coarse} + \overline{M}
                             \},
\end{equation}
where $M_{\PPP_\coarse} := \#\supp(\PPP_\coarse) \in \N_0$ is the number of active parameters in the index set $\PPP_\coarse$.
Thus, for a given $\PPP_\coarse \subset \III$, the detail index set represents an ``active boundary'' of $\PPP_\coarse$
that contains multi-indices having up to $M_{\PPP_\coarse} + \overline{M}$ active parameters.
Then, a parametric enrichment is obtained by adding some marked indices $\MMM_\coarse \subseteq \QQQ_\coarse$
to the current index set $\PPP_\coarse$, i.e.,
$\PPP_\coarse \subseteq \PPP_\fine := \PPP_\coarse \cup \MMM_\coarse \subseteq \PPP_\coarse \cup \QQQ_\coarse$. 

\subsection{Multilevel approximation spaces}

For each index $\nu \in \PPP_\coarse$, let $\TT_{\coarse\nu} \in \refine(\TT_0)$ be a mesh and $\X_{\coarse\nu} := \SS^1_0(\TT_{\coarse\nu})$ be the corresponding finite element space.
Furthermore, for all indices $\nu \in \III \backslash \PPP_\coarse$, we set $\TT_{\coarse\nu} := \TT_0$.
Following~\cite{egsz14}, our discretization of~\eqref{eq:weakform} is based on the finite-dimensional subspace
\begin{equation}\label{eq:def:V}
 \V_\coarse := \bigoplus_{\nu \in \PPP_\coarse} \V_{\coarse\nu} \subset \V
 \quad \text{with} \quad
 \V_{\coarse\nu} := \X_{\coarse\nu} \otimes {\rm span}\{P_\nu\} = {\rm span}\set{\varphi_{\coarse\nu,z} P_\nu}{z \in \NN_{\coarse\nu}}.
\end{equation}
Note that the sum of the spaces $\V_{\coarse\nu}$ in~\eqref{eq:def:V} is orthogonal and hence direct.
We emphasize that, in contrast to~\cite{egsz15,bprr18+,bprr18++}, where
$\X_{\coarse\nu} = \X_{\coarse\mu} =: \X_\coarse$  for all $\nu, \mu \in \PPP_\coarse$
and, hence, $\V_\coarse = \X_\coarse \otimes {\rm span}\set{P_\nu}{\nu \in \PPP_\coarse}$ has the tensor product structure
(the so-called \emph{single-level} approximation space),
the approximation space $\V_\coarse$ defined in~\eqref{eq:def:V} has a \emph{multilevel} structure that
allows $\X_{\coarse\nu} \neq \X_{\coarse\mu}$ for $\mu \neq \nu$.
Furthermore, while each mesh $\TT_{\coarse\nu}$ ($\nu \in \PPP_\coarse$) 
is obtained by a local refinement of the same coarse mesh~$\TT_0$,
any two meshes $\TT_{\coarse\nu},\, \TT_{\coarse\mu}$ ($\nu,\,\mu \in \PPP_\coarse$) are not necessarily nested.
This is a more general construction than that considered in~\cite{cpb18+},
where the meshes $\TT_{\coarse\nu},\, \TT_{\coarse\mu}$ ($\nu \neq \mu$) were assumed to be nested.

The Galerkin discretization of~\eqref{eq:weakform} reads as follows:
Find $\uu_\coarse \in \V_\coarse$ such that
\begin{equation}\label{eq:discrete_formulation}
 B(\uu_\coarse, \vv_\coarse) = F(\vv_\coarse) 
 \quad \text{for all } \vv_\coarse \in \V_\coarse. 
\end{equation}
Again, the Riesz theorem proves the existence and uniqueness of the solution $\uu_\coarse \in \V_\coarse$.
Moreover, 
the mapping $\V \ni \uu \mapsto \uu_\coarse \in \V_\coarse$ is the orthogonal projection
onto $\V_\coarse$
with respect to the bilinear form $B(\cdot,\cdot)$.
Therefore, there holds the best approximation property
\begin{equation*}
 \bnorm{\uu - \uu_\coarse} = \min_{\vv_\coarse \in \V_\coarse} \bnorm{\uu - \vv_\coarse}.
\end{equation*}

\section{A~posteriori error estimation} \label{sec:aposteriori}

\subsection{Saturation assumption} \label{sec:saturation}

Given a multilevel subspace $\V_\coarse$ from~\eqref{eq:def:V},
we adopt the approach of~\cite[Remark~4.3]{bs16} and consider an enriched subspace
$\widehat\V_\coarse \supseteq \V_\coarse$ defined~as
\begin{equation}\label{eq:def:Vhat}
 \widehat\V_\coarse 
 := \bigoplus_{\nu \in \PPP_\coarse} \big[ \widehat\X_{\coarse\nu} \otimes {\rm span}\{P_\nu\} \big]
 \oplus \bigoplus_{\nu \in \QQQ_\coarse} \big[ \X_0 \otimes {\rm span}\{P_\nu\} \big]
 \subset \V,
\end{equation}
where we recall that $\TT_{\coarse\nu} = \TT_0$ for all $\nu \in \QQQ_\coarse \subset \III \backslash \PPP_\coarse$.
Note that $\V_\coarse \subseteq \V_\fine \subseteq \widehat\V_\coarse$,
where $\V_\fine$ is obtained from $\V_\coarse$ by \emph{one step} of (adaptive) refinement/enrichment, i.e.,
$\V_\fine$ is represented in the form~\eqref{eq:def:V} with
\begin{subequations} \label{eq:V:fine}
\begin{gather}
      \PPP_\circ = \PPP_\coarse \cup \MMM_\coarse \subseteq \PPP_\coarse \cup \QQQ_\coarse,
      \\
      \TT_{\fine\nu} = \refine(\TT_{\coarse\nu},\MM_{\coarse\nu}) \text{\ for all } \nu \in \PPP_\coarse \text{ \ and \ }
      \TT_{\fine\nu} = \TT_0 \text{\ for all } \nu \in \PPP_\fine\backslash \PPP_\coarse.
\end{gather}
\end{subequations}

Let $\widehat \uu_\coarse \in \widehat\V_\coarse$ be the unique Galerkin solution to
\begin{equation}\label{eq:discrete_formulation-hat}
 B(\widehat \uu_\coarse, \widehat \vv_\coarse) = F(\widehat \vv_\coarse) 
 \quad \text{for all } \widehat \vv_\coarse \in \widehat\V_\coarse. 
\end{equation}
Existence and uniqueness of the solution $\widehat \uu_\coarse \in \widehat\V_\coarse$ follow from the Riesz theorem.
We emphasize that $\widehat \uu_\coarse \in \widehat\V_\coarse$
is only needed for analysis and will not be computed throughout.

We suppose that there exists a uniform constant $0 < \qsat < 1$ such that the following saturation assumption holds:
\begin{equation}\label{eq:saturation}
 \bnorm{\uu - \widehat \uu_\coarse} \le \qsat \, \bnorm{\uu - \uu_\coarse}.
\end{equation}
We recall the orthogonal decomposition
\begin{equation*}
 \bnorm{\uu - \widehat \uu_\coarse}^2 + \bnorm{\widehat \uu_\coarse - \uu_\coarse}^2
 = \bnorm{\uu - \uu_\coarse}^2.
\end{equation*}
Elementary calculation thus proves that the saturation assumption~\eqref{eq:saturation} is equivalent to
\begin{equation}\label{eq2:saturation}
 \bnorm{\widehat \uu_\coarse - \uu_\coarse}^2 
 \le \bnorm{\uu - \uu_\coarse}^2 
 \le \frac{1}{1-\qsat^2} \, \bnorm{\widehat \uu_\coarse - \uu_\coarse}^2,
\end{equation}
i.e., the Galerkin error $\bnorm{\uu - \uu_\coarse}$ of the (computed) coarse-space solution $\uu_\coarse \in \V_\coarse$ is equivalent to the error reduction $\bnorm{\widehat \uu_\coarse - \uu_\coarse}$ with respect to the (non-computed) fine-space solution $\widehat \uu_\coarse \in \widehat\V_\coarse$.

\begin{remark}
The saturation assumption~\eqref{eq:saturation} is a strong restriction (which may even fail in general~\cite{bek96}),
if required for all discrete subspaces $\V_\coarse$.
In practice, however, it is only required for the sequence of nested discrete subspaces
generated by an adaptive solution process.
\end{remark}

\subsection{\textsl{A~posteriori} error estimator. Main result}

The error in multilevel stochastic Galerkin approximations has two principal components:
the \emph{parametric} error arising from the choice of the index set~$\PPP_\coarse$
and the \emph{spatial} error due to finite element discretizations for each $\nu \in \PPP_\coarse$.
We estimate the contributions to the error from each of these two components separately.
To abbreviate notation, let $\dual{w}{v} := \int_D a_0\nabla w \cdot \nabla v \d{x}$ be the energy scalar product
on the space $\X = H^1_0(D)$ in the physical domain
and let $\norm{\cdot}{D} := \norm{a_0^{1/2}\nabla(\cdot)}{L^2(D)}$ be the induced energy norm on $\X$.

The \emph{parametric} error is estimated by means of hierarchical error indicators
(cf.~\cite{bps14,bs16})
\begin{subequations} \label{eq1:parametric-error-estimate}
\begin{equation} \label{eq1:parametric-error-estimate:a}
 \est_\coarse(\nu) := \norm{e_{\coarse\nu}}{D}
 \quad \text{for all } \nu \in \QQQ_\coarse, 
\end{equation}
where $e_{\coarse\nu} \in \X_0$ is the unique solution of
\begin{equation} \label{eq1:parametric-error-estimate:b}
 \dual{e_{\coarse\nu}}{v_0}
 = F(v_0 P_\nu) - B(\uu_\coarse, v_0 P_\nu) 
  \quad \text{for all } v_0  \in \X_0.
\end{equation}
\end{subequations}

In order to estimate the errors due to \emph{spatial} discretizations,
we employ the two-level error estimation strategy, which has been analyzed in~\cite{bprr18+} for single-level approximation spaces.
Specifically, we define the two-level error indicators
\begin{equation}\label{eq1:spatial-error-estimate}
 \est_{\coarse}(\nu,z)
 := \frac{|F(\widehat\varphi_{\coarse\nu,z} P_\nu) - B(\uu_\coarse,\widehat\varphi_{\coarse\nu,z} P_\nu)|}{\norm{\widehat\varphi_{\coarse\nu,z}}{D}}
 \quad \text{for all } \nu \in \PPP_\coarse \text{ and all } z \in \NN_{\coarse\nu}^+.
\end{equation}

Overall, we thus consider the computable \textsl{a~posteriori} error estimate
\begin{equation}\label{eq:def:tau}
 \est_\coarse 
 := \bigg( \sum_{\nu \in \PPP_\coarse}\sum_{z \in \NN_{\coarse\nu}^+} \est_{\coarse}(\nu,z)^2 
 + \sum_{\nu \in \QQQ_\coarse} \est_\coarse(\nu)^2  \bigg)^{1/2}.
\end{equation}
The following theorem is the main theoretical result of this work.

\begin{theorem}\label{thm:estimator}
Let $\V_\coarse$ be a given multilevel approximation space~\eqref{eq:def:V},
and let $\widehat\V_\coarse$ be the enriched space as defined in~\eqref{eq:def:Vhat}.
Then,
for two Galerkin approximations $\uu_\coarse \in \V_\coarse$ and $\widehat \uu_\coarse \in \widehat\V_\coarse$
satisfying~\eqref{eq:discrete_formulation} and~\eqref{eq:discrete_formulation-hat}, respectively, there holds
\begin{equation}\label{eq1:thm:estimator}
 \Cest^{-1} \, \bnorm{\widehat \uu_\coarse - \uu_\coarse}
 \le \est_\coarse \reff{eq:def:tau}= \bigg( \sum_{\nu \in \PPP_\coarse}\sum_{z \in \NN_{\coarse\nu}^+} \est_{\coarse}(\nu,z)^2 
 + \sum_{\nu \in \QQQ_\coarse} \est_\coarse(\nu)^2 \bigg)^{1/2}
 \le \Cest \, \bnorm{\widehat \uu_\coarse - \uu_\coarse} . 
\end{equation}
Furthermore, if $\uu \in \V$ is the solution to problem~\eqref{eq:weakform}, then,
under the saturation assumption~\eqref{eq:saturation}, the estimates~\eqref{eq1:thm:estimator} are equivalent to
\begin{equation}\label{eq2:thm:estimator}
 \frac{(1-\qsat^2)^{1/2}}{\Cest} \, \bnorm{\uu - \uu_\coarse}
 \le \est_\coarse 
 \le \Cest \, \bnorm{\uu - \uu_\coarse},
\end{equation}
i.e., the proposed error estimator is reliable (under the saturation assumption) and (always) efficient.
The constant $\Cest \ge 1$ in~\eqref{eq1:thm:estimator}--\eqref{eq2:thm:estimator}
is generic and depends only on uniform shape regularity of the 
refinements of $\TT_0$,
the mean field $a_0$, and the constant $\tau > 0$ from~\eqref{eq3:a}.
\end{theorem}

\subsection{Auxiliary results in deterministic setting}

Throughout this section, we denote by $\TT_\deterministic \in \refine(\TT_0)$
an arbitrary refinement of the initial mesh. Recall that $\X = H^1_0(D)$ and $\X_\deterministic := \SS^1_0(\TT_\deterministic)$.
The proof of Theorem~\ref{thm:estimator} will employ the (spatial) orthogonal projections
\begin{equation*}
 \G_\deterministic : \X \to \X_\deterministic
 \quad \text{and} \quad
 \widehat\G_{\deterministic,z} : \X \to \widehat\X_{\deterministic,z} := {\rm span}\{\widehat\varphi_{\deterministic, z}\}\ \ 
 \text{for}\ z \in \NN_\deterministic^+
\end{equation*}
defined by
\begin{align}
 \dual{\G_\deterministic w}{v_\deterministic} &= \dual{w}{v_\deterministic}
 \quad\quad \text{for all } v_\deterministic \in \X_\deterministic,
 \label{eq:projector:G}
 \\
 \dual{\widehat\G_{\deterministic,z} w}{\widehat v_{\deterministic,z}} &= \dual{w}{\widehat v_{\deterministic,z}}
 \quad\,\; \text{for all } \widehat v_{\deterministic,z} \in \widehat\X_{\deterministic,z}.
 \label{eq:projector:hatG}
\end{align}
First, we recall the norm equivalence from~\cite[Proof of Lemma~3.4, Steps~1--2]{bprr18+}.

\begin{lemma}\label{lemma:localization}
For all $z \in \NN_\deterministic^+$, let $\widehat w_{\deterministic,z} \in {\rm span}\{\widehat\varphi_{\deterministic,z}\}$.
Then, there holds
\begin{equation} \label{eq:localization}
 K^{-1} \, \norm[\Big]{\sum_{z \in \NN_\deterministic^+}\widehat w_{\deterministic,z}}{D}^2
 \le \sum_{z \in \NN_\deterministic^+} \norm{\widehat w_{\deterministic,z}}{D}^2
 \le \Cloc \, \norm[\Big]{\sum_{z \in \NN_\deterministic^+}\widehat w_{\deterministic,z}}{D}^2.
\end{equation}
Here, $\Cloc > 0$ depends only on the shape regularity of $\widehat\TT_\deterministic$
and the mean field $a_0$, whereas $K > 0$ is the constant from~\eqref{eq:const:K}.
\end{lemma}

Second, we recall that nodal interpolation is stable on finite-dimensional subspaces;
see~\cite[Proof of Lemma~3.5, Step~1]{bprr18+}.

\begin{lemma}\label{lemma:interpolation}
For $\widehat v_\deterministic \in \widehat\X_\deterministic$, let $v_\deterministic := \sum_{z \in \NN_\deterministic} \widehat v_\deterministic(z) \varphi_{\deterministic,z}$ be the nodal interpolation onto $\X_\deterministic$. Then
\begin{equation} \label{eq:interpolation:1}
 \widehat v_\deterministic - v_\deterministic = \sum_{z \in \NN_\deterministic^+} \widehat w_{\deterministic,z} \quad \text{with} \quad \widehat w_{\deterministic,z} \in {\rm span}\{\widehat\varphi_{\deterministic,z}\}
\end{equation}
and there holds
\begin{equation} \label{eq:interpolation:2}
 \norm{\widehat v_\deterministic - v_\deterministic}{D}
 \le \Cstab \, \norm{\widehat v_\deterministic}{D},
\end{equation}
where $\Cstab>0$ depends only on the shape regularity of $\widehat\TT_\deterministic$
and the mean field $a_0$.
\end{lemma}

\subsection{Proof of Theorem~\ref{thm:estimator}}

Recall the orthogonal projectors $\G_\deterministic : \X \to \X_\deterministic$
and $\widehat\G_{\deterministic,z} : \X \to \widehat\X_{\deterministic,z}$
defined in~\eqref{eq:projector:G} and~\eqref{eq:projector:hatG}, respectively.
The following lemma provides the key argument for the proof of Theorem~\ref{thm:estimator}.

\begin{lemma}\label{lemma:twolevel}
For any $\widehat\vv_\coarse = \sum_{\nu \in \PPP_\coarse} \widehat v_{\coarse\nu} P_\nu + \sum_{\nu \in \QQQ_\coarse} v_{\coarse\nu} P_\nu \in \widehat\V_\coarse$, where $\widehat v_{\coarse\nu} \in \widehat\X_{\coarse\nu}$ for $\nu \in \PPP_\coarse$ and $v_{\coarse\nu} \in \X_{\coarse\nu} = \X_0$ for $\nu \in \QQQ_\coarse$, the following estimates hold
\begin{equation} \label{eq:lemma:twolevel}
 \CY^{-1} \, \enorm{\widehat\vv_\coarse}{0}^2
 \le \sum_{\nu \in \PPP_\coarse} 
 	\Big( \norm{\G_{\coarse\nu} \widehat v_{\coarse\nu}}{D}^2
	+ \sum_{z \in \NN_{\coarse\nu}^+} \norm{\widehat\G_{\coarse\nu,z} \widehat v_{\coarse\nu}}{D}^2 \Big)
	+ \sum_{\nu \in \QQQ_\coarse} \norm{v_{\coarse\nu}}{D}^2
 \le 2 \KY \, \enorm{\widehat\vv_\coarse}{0}^2.
\end{equation}
Here, $\CY\ge1$ depends only on the shape regularity of $\widehat\TT$
and the mean field $a_0$, whereas $K > 0$ is the constant from~\eqref{eq:const:K}.
Moreover, the upper bound holds with the constant $\KY$ (instead of $2\KY$) if $\G_{\coarse\nu} \widehat v_{\coarse\nu}= 0$
for all $\nu \in \PPP_\coarse$.
\end{lemma}

\begin{proof}
Using~\eqref{eq2:lemma:orthogonal}, we have
\begin{equation}\label{eq:orthogonal2}
 \enorm{\widehat\vv_\coarse}{0}^2
 = \sum_{\nu \in \PPP_\coarse} \norm{\widehat v_{\coarse\nu}}{D}^2
	+ \sum_{\nu \in \QQQ_\coarse} \norm{v_{\coarse\nu}}{D}^2.
\end{equation}
For all $\nu \in \PPP_\coarse$, we
apply Lemma~\ref{lemma:interpolation} to $\widehat v_{\coarse\nu} \in \widehat\X_{\coarse\nu}$
in order to find $v_{\coarse\nu} \in \X_{\coarse\nu}$ and $\widehat w_{\coarse\nu,z} \in {\rm span}\{\widehat\varphi_{\coarse\nu,z}\}$ for all $z \in \NN_{\coarse\nu}^+$
such that~\eqref{eq:interpolation:1}--\eqref{eq:interpolation:2} hold.

{\bf Step~1.}
First, we prove the lower bound in~\eqref{eq:lemma:twolevel}.
The Cauchy inequality yields that
\begin{align*}
 \norm{\widehat v_{\coarse\nu}}{D}^2
 & \stackrel{\eqref{eq:interpolation:1}}{=} \dual[\Big]{\widehat v_{\coarse\nu}}{v_{\coarse\nu} + \sum_{z \in \NN_{\coarse\nu}^+} \widehat w_{\coarse\nu,z}}
 = \dual{\G_{\coarse\nu}\widehat v_{\coarse\nu}}{v_{\coarse\nu}}
 + \sum_{z \in \NN_{\coarse\nu}^+} \dual{\widehat\G_{\coarse\nu,z}\widehat v_{\coarse\nu}}{\widehat w_{\coarse\nu,z}}
 \\
 & \stackrel{\phantom{\eqref{eq:interpolation:1}}}{\le} \Big( \norm{\G_{\coarse\nu}\widehat v_{\coarse\nu}}{D}^2
	 + \sum_{z \in \NN_{\coarse\nu}^+} \norm{\widehat\G_{\coarse\nu,z}\widehat v_{\coarse\nu}}{D}^2 \Big)^{1/2}
	 \Big(\norm{ v_{\coarse\nu}}{D}^2
	 + \sum_{z \in \NN_{\coarse\nu}^+} \norm{ \widehat w_{\coarse\nu,z}}{D}^2 \Big)^{1/2}.
\end{align*}
Stability~\eqref{eq:interpolation:2} shows that
\begin{equation*}
 \norm{v_{\coarse\nu}}{D} 
 \le \norm{\widehat v_{\coarse\nu}}{D} + \norm{\widehat v_{\coarse\nu} - v_{\coarse\nu}}{D}
 \stackrel{\eqref{eq:interpolation:2}}{\le} (1 + \Cstab) \norm{\widehat v_{\coarse\nu}}{D}.
\end{equation*}
The upper bound in~\eqref{eq:localization} proves that
\begin{equation*}
 \sum_{z \in \NN_{\coarse\nu}^+} \norm{\widehat w_{\coarse\nu,z}}{D}^2
 \le \Cloc \, \norm[\Big]{\sum_{z \in \NN_{\coarse\nu}^+}  \widehat w_{\coarse\nu,z}}{D}^2
 = \Cloc \, \norm{\widehat v_{\coarse\nu} - v_{\coarse\nu}}{D}^2
 \reff{eq:interpolation:2}\le \Cloc\Cstab^2 \, \norm{\widehat v_{\coarse\nu}}{D}^2.
\end{equation*}
Combining the latter three estimates, we conclude that
\begin{equation*}
 \norm{\widehat v_{\coarse\nu}}{D}
 \le \big[(1+\Cstab)^2 + \Cloc\Cstab^2 \big]^{1/2} \, \Big( \norm{\G_{\coarse\nu}\widehat v_{\coarse\nu}}{D}^2
	 + \sum_{z \in \NN_{\coarse\nu}^+} \norm{\widehat\G_{\coarse\nu,z}\widehat v_{\coarse\nu}}{D}^2 \Big)^{1/2}.
\end{equation*}
Using this estimate together with~\eqref{eq:orthogonal2},
we prove the lower bound in~\eqref{eq:lemma:twolevel} with $\CY = (1+\Cstab)^2 + \Cloc\Cstab^2 \ge 1$.

{\bf Step~2.}
To prove the upper bound in~\eqref{eq:lemma:twolevel}, we proceed analogously.
For all $\nu \in \PPP_\coarse$, there holds
\begin{align*}
 &\norm{\G_{\coarse\nu}\widehat v_{\coarse\nu}}{D}^2
	 + \sum_{z \in \NN_{\coarse\nu}^+} \norm{\widehat\G_{\coarse\nu,z}\widehat v_{\coarse\nu}}{D}^2
 = \dual{\G_{\coarse\nu}\widehat v_{\coarse\nu}}{\widehat v_{\coarse\nu}}
 + \sum_{z \in \NN_{\coarse\nu}^+} \dual{\widehat\G_{\coarse\nu,z}\widehat v_{\coarse\nu}}{\widehat v_{\coarse\nu}}
 \qquad\qquad\qquad\quad
 \\[4pt]
 & \quad\qquad\qquad\qquad
 = \dual[\Big]{\G_{\coarse\nu}\widehat v_{\coarse\nu} + \sum_{z \in \NN_{\coarse\nu}^+} \widehat\G_{\coarse\nu,z}\widehat v_{\coarse\nu}}{\widehat v_{\coarse\nu}}
 \le \norm[\Big]{\G_{\coarse\nu}\widehat v_{\coarse\nu} + \sum_{z \in \NN_{\coarse\nu}^+} \widehat\G_{\coarse\nu,z}\widehat v_{\coarse\nu}}{D}
 \norm{\widehat v_{\coarse\nu}}{D}.
\end{align*}
Using the lower bound in~\eqref{eq:localization}
and the fact that $K \ge 1$, we prove that
\begin{align*}
 \norm[\Big]{\G_{\coarse\nu}\widehat v_{\coarse\nu} + \sum_{z \in \NN_{\coarse\nu}^+} \widehat\G_{\coarse\nu,z}\widehat v_{\coarse\nu}}{D}^2
 &\le 2 \, \Big( \norm{\G_{\coarse\nu}\widehat v_{\coarse\nu}}{D}^2 + \norm[\Big]{\sum_{z \in \NN_{\coarse\nu}^+} \widehat\G_{\coarse\nu,z}\widehat v_{\coarse\nu}}{D}^2 \Big)
 \\& 
 \le 2K \, \Big( \norm{\G_{\coarse\nu}\widehat v_{\coarse\nu}}{D}^2 + \sum_{z \in \NN_{\coarse\nu}^+} \norm{\widehat\G_{\coarse\nu,z}\widehat v_{\coarse\nu}}{D}^2 \Big).
\end{align*}
The latter two estimates imply that
\begin{equation*}
 \Big( \norm{\G_{\coarse\nu}\widehat v_{\coarse\nu}}{D}^2 + \sum_{z \in \NN_{\coarse\nu}^+} \norm{\widehat\G_{\coarse\nu,z}\widehat v_{\coarse\nu}}{D}^2 \Big)^{1/2}
 \le \sqrt{2K} \, \norm{\widehat v_{\coarse\nu}}{D}.
\end{equation*} 
By substituting this estimate into~\eqref{eq:orthogonal2}, we conclude the proof.
\end{proof}

\begin{proof}[Proof of Theorem~\ref{thm:estimator}]
The equivalence of estimates~\eqref{eq1:thm:estimator} and~\eqref{eq2:thm:estimator} is an immediate consequence of~\eqref{eq2:saturation}. Therefore, it only remains to prove~\eqref{eq1:thm:estimator}.
The proof consists of three~steps.

{\bf Step~1.} 
Define
$\widehat\ee_\coarse :=
  \sum_{\nu \in \PPP_\coarse} \widehat e_{\coarse\nu} P_\nu +
  \sum_{\nu \in \QQQ_\coarse} e_{\coarse\nu} P_\nu \in \widehat\V_\coarse
$,
where $e_{\coarse\nu} \in \X_0 = \X_{\coarse\nu}$ for $\nu \in \QQQ_\coarse$ is given by~\eqref{eq1:parametric-error-estimate}, while $\widehat e_{\coarse\nu} \in \widehat\X_{\coarse\nu}$ for $\nu \in \PPP_\coarse$ is the
unique solution to
\begin{equation}\label{eqX:thm:estimator}
 \dual{\widehat e_{\coarse\nu}}{\widehat v_{\coarse\nu}} 
 = B(\widehat\uu_\coarse - \uu_\coarse, \widehat v_{\coarse\nu} P_\nu)
 \quad \text{for all } \widehat v_{\coarse\nu} \in \widehat\X_{\coarse\nu}.
\end{equation}
For all $\nu \in \PPP_\coarse$,
Galerkin orthogonality implies~that
\begin{equation*}
 \dual{\widehat e_{\coarse\nu}}{v_{\coarse\nu}} 
 = B(\widehat\uu_\coarse - \uu_\coarse, v_{\coarse\nu} P_\nu)
 = 0 
 \quad \text{for all } v_{\coarse\nu} \in \X_{\coarse\nu}.
\end{equation*}
Hence, we see that $\G_{\coarse\nu} \widehat e_{\coarse\nu} = 0$ for all $\nu \in \PPP_\coarse$.
In conclusion, Lemma~\ref{lemma:twolevel} yields that
\begin{equation*}
 \enorm{\widehat\ee_\coarse}{0}^2
 \simeq \sum_{\nu \in \PPP_\coarse} \sum_{z \in \NN_{\coarse\nu}^+} \norm{\widehat\G_{\coarse\nu,z} \widehat e_{\coarse\nu}}{D}^2 
 	+ \sum_{\nu \in \QQQ_\coarse} \norm{e_{\coarse\nu}}{D}^2
 \reff{eq1:parametric-error-estimate:a}=
 \sum_{\nu \in \PPP_\coarse} \sum_{z \in \NN_{\coarse\nu}^+}
 \norm{\widehat\G_{\coarse\nu,z} \widehat e_{\coarse\nu}}{D}^2 
 	+ \sum_{\nu \in \QQQ_\coarse} \est_\coarse(\nu)^2.
\end{equation*}

{\bf Step~2.} The orthogonal projection onto the one-dimensional space ${\rm span}\{ \widehat\varphi_{\coarse\nu,z}\}$ takes the explicit form
\begin{equation*}
 \widehat\G_{\coarse\nu,z} v
 = \frac{\dual{v}{\widehat\varphi_{\coarse\nu,z}}}{\norm{\widehat\varphi_{\coarse\nu,z}}{D}^2} \, \widehat\varphi_{\coarse\nu,z}
 \quad \text{for any } v \in \X.
\end{equation*}
Hence, for all $\nu \in \PPP_\coarse$ and for each $z \in \NN_{\coarse\nu}^+$, there holds
\begin{equation*}
 \norm{\widehat\G_{\coarse\nu,z} \widehat e_{\coarse\nu}}{D} = \frac{|\dual{\widehat e_{\coarse\nu}}{\widehat\varphi_{\coarse\nu,z}}|}{\norm{\widehat\varphi_{\coarse\nu,z}}{D}} 
 = \frac{|B(\widehat \uu_\coarse - \uu_\coarse, \widehat\varphi_{\coarse\nu,z} P_\nu)|}{\norm{\widehat\varphi_{\coarse\nu,z}}{D}}
 \reff{eq1:spatial-error-estimate}= \est_\coarse(\nu,z).
\end{equation*}
This leads to the equivalence
\begin{equation*}
 \enorm{\widehat\ee_\coarse}{0}^2
 \simeq \sum_{\nu \in \PPP_\coarse} \sum_{z \in \NN_{\coarse\nu}^+} \est_\coarse(\nu,z)^2 
 	+ \sum_{\nu \in \QQQ_\coarse} \est_\coarse(\nu)^2
 = \est_\coarse^2,
\end{equation*}
where the hidden constants depend only on uniform shape regularity of the meshes $\TT_\deterministic \in \refine(\TT_0)$, the (local) mesh-refinement rule, and the mean field $a_0$.

{\bf Step~3.}
It remains to prove the equivalence $\enorm{\widehat\ee_\coarse}{0} \simeq \bnorm{\widehat\uu_\coarse - \uu_\coarse}$.
To that end, we note that the variational formulation~\eqref{eqX:thm:estimator} implies that
\begin{equation*}
 B_0(\widehat\ee_\coarse, \widehat\vv_\coarse) = B(\widehat\uu_\coarse - \uu_\coarse, \widehat\vv_\coarse)
 \quad \text{for all } \widehat\vv_\coarse \in \widehat\V_\coarse.
\end{equation*}
Hence, using norm equivalence~\eqref{eq:lambda}, we obtain that
\begin{equation*}
 \enorm{\widehat\ee_\coarse}0^2 = B(\widehat\uu_\coarse - \uu_\coarse, \widehat\ee_\coarse)
 \le \bnorm{\widehat\uu_\coarse - \uu_\coarse} \bnorm{\widehat\ee_\coarse}
 \le \Lambda^{1/2} \, \bnorm{\widehat\uu_\coarse - \uu_\coarse} \enorm{\widehat\ee_\coarse}{0}
\end{equation*}
and
\begin{equation*}
 \bnorm{\widehat\uu_\coarse - \uu_\coarse}^2
 = B_0(\widehat\ee_\coarse, \widehat\uu_\coarse - \uu_\coarse)
 \le \enorm{\widehat\ee_\coarse}{0} \enorm{\widehat\uu_\coarse - \uu_\coarse}{0}
 \le \lambda^{-1/2} \, \enorm{\widehat\ee_\coarse}{0} \bnorm{\widehat\uu_\coarse - \uu_\coarse}.
\end{equation*}
This concludes the proof.
\end{proof}

\begin{remark} \label{rem:reduction}
Let $\V_\fine$ be a multilevel approximation space that is obtained from $\V_\coarse$
by one step of (adaptive) refinement/enrichment (see~\eqref{eq:V:fine}) such that
$\V_\coarse \subseteq \V_\fine \subseteq \widehat\V_\coarse$.
For $d = 2$, newest vertex bisection ensures that $\widehat\varphi_{\coarse\nu,z} = \varphi_{\fine\nu,z}$
for all $\nu \in \PPP_\coarse$ and for each
$z \in \NN_{\coarse\nu}^+ \cap \NN_{\fine\nu}$.
If $\uu_\coarse \in \V_\coarse$ and $\uu_\fine \in \V_\fine$ are two Galerkin approximations, then by arguing
as in the proof of Theorem~\ref{thm:estimator}, we obtain that 
\begin{equation}\label{eq3:thm:estimator}
 \Cest^{-1} \, \bnorm{\uu_\fine - \uu_\coarse}
 \le\! \bigg( \sum_{\nu \in \PPP_\coarse}\sum_{z \in \NN_{\coarse\nu}^+ \cap \NN_{\fine\nu}} \est_{\coarse}(\nu,z)^2 
 + \sum_{\nu \in \QQQ_\coarse \cap \PPP_\fine} \est_\coarse(\nu)^2 \bigg)^{1/2}\!
 \le \Cest \, \bnorm{\uu_\fine - \uu_\coarse}.
\end{equation} 
Therefore, in this setting (at least in 2D),
the two-level estimator allows to control the error reduction
due to adaptive enrichment of the multilevel approximation space~$\V_\coarse$.
\end{remark}

\section{Adaptive algorithms} \label{sec:algorithms}

In this section, we present adaptive algorithms with three different D{\"o}rfler-type marking criteria
(and hence, different refinement strategies).
These algorithms generate sequences of successively enriched multilevel approximation spaces,
as well as the corresponding Ga\-ler\-kin approximations and error estimates.

We consider the following standard adaptive loop
\begin{equation*}
{\sf SOLVE} \longrightarrow {\sf ESTIMATE}  \longrightarrow {\sf MARK}  \longrightarrow {\sf REFINE},
\end{equation*}
where the precise marking strategy is to be specified in the subsections~below.

\begin{algorithm}\label{algorithm}
{\bfseries Input:}
$\PPP_0 = \{ \0 \}$ and $\TT_{0\nu} := \TT_0$ for all $\nu \in \PPP_0 \cup \QQQ_0$;
marking criterion. Set the counter $\ell := 0$.
\begin{itemize}
\item[\rm(i)] 
Compute the discrete solution $\uu_\ell \in \V_\ell$ by solving~\eqref{eq:discrete_formulation}.
\item[\rm(ii)] 
Compute spatial error indicators $\est_{\ell}(\nu,z)$ from~\eqref{eq1:spatial-error-estimate} for all $\nu \in \PPP_\ell$ and all $z \in \NN_{\ell\nu}^+$.
\item[\rm(iii)] 
Compute parametric error indicators $\est_{\ell}(\nu)$ from~\eqref{eq1:parametric-error-estimate} for all $\nu \in \QQQ_\ell$.
\item[\rm(iv)] 
Use marking criterion to determine $\MM_{\ell\nu} \subseteq \NN_{\ell\nu}^+$ for all $\nu \in \PPP_\ell$ and
$\MMM_\ell \subseteq \QQQ_\ell$.
\item[\rm(v)] For all $\nu \in \PPP_\ell$, set $\TT_{(\ell+1)\nu} := \refine(\TT_{\ell\nu},\MM_{\ell\nu})$.
\item[\rm(vi)] Set $\PPP_{\ell+1} := \PPP_\ell \cup \MMM_\ell$ and
$\TT_{(\ell+1)\nu} := \TT_0$ for all $\nu \in \QQQ_{\ell+1}$.
\item[\rm(vii)] Increase the counter $\ell \mapsto \ell+1$ and goto {\rm(i)}.
\end{itemize}
{\bfseries Output:} For all $\ell \in \N_0$,
the algorithm returns the multilevel stochastic Galerkin approximation~$\uu_\ell \in \V_\ell$
as well as the corresponding error estimate $\est_\ell$.
\end{algorithm}

\subsection{Separate spatial and parametric marking/enrichment} \label{sec:alg1}

The two marking criteria presented below
follow the same approach as utilized in~\cite{bps14,bs16,br18,bprr18++} in the case of single-level stochastic Galerkin FEM.
Under this approach, \emph{either} a spatial refinement \emph{or} a parametric enrichment is performed at each iteration.
The choice between the two
is made by comparing the respective contributions to the total error estimate $\est_\coarse$
given by~\eqref{eq:def:tau} (Marking criterion~\ref{marking:A})
or by comparing the associated error reduction indicators (Marking criterion~\ref{marking:B}; cf.\ Remark~\ref{rem:reduction}).

\begin{marking} \label{marking:A}
\textbf{Input:}
error indicators $\{ \est_{\ell}(\nu,z) : \nu \in \PPP_\ell,\ z \in \NN_{\ell\nu}^+ \}$ and
$\{ \est_{\ell}(\nu) : \nu \in \QQQ_\ell \}$;
marking parameters $0 < \theta_\X, \theta_\PPP \le 1$, and $\vartheta > 0$.
\begin{itemize}
\item[$\bullet$]
If \ $\vartheta \sum_{\nu \in \QQQ_\ell} \est_\ell(\nu)^2 \le \sum_{\nu \in \PPP_\ell} \sum_{z \in \NN_{\ell\nu}^+} \est_{\ell}(\nu,z)^2$,
then proceed as follows:
\begin{itemize}
\item[$\circ$]
Set $\MMM_\ell := \emptyset$.
\item[$\circ$]
Determine $\MM_{\ell\nu} \subseteq \NN_{\ell\nu}^+$ for all $\nu \in \PPP_\ell$ such that 
\begin{equation}
 \label{eq:doerfler:separate1}
 \theta_\X \, \sum_{\nu \in \PPP_\ell} \sum_{z \in \NN_{\ell\nu}^+} \est_{\ell}(\nu,z)^2 
 \le \sum_{\nu \in \PPP_\ell} \sum_{z \in \MM_{\ell\nu}} \est_{\ell}(\nu,z)^2,
\end{equation}
where the cumulative cardinality $\sum_{\nu \in \PPP_\ell} \# \MM_{\ell\nu}$ is minimal (amongst all sets which satisfy the marking criterion~\eqref{eq:doerfler:separate1}).
\end{itemize}
\item[$\bullet$]
Otherwise, if \ 
$\vartheta \sum_{\nu \in \QQQ_\ell} \est_\ell(\nu)^2 > \sum_{\nu \in \PPP_\ell} \sum_{z \in \NN_{\ell\nu}^+} \est_{\ell}(\nu,z)^2$,
then proceed as follows:
\begin{itemize}
\item[$\circ$]
Set $\MM_{\ell\nu} := \emptyset$ for all $\nu \in \PPP_\ell$.
\item[$\circ$]
Determine $\MMM_\ell \subseteq \QQQ_\ell$ such that
\begin{equation}
 \label{eq:doerfler:separate2}
 \theta_\PPP \, \sum_{\nu \in \QQQ_\ell} \est_\ell(\nu)^2 \le \sum_{\nu \in \MMM_\ell} \est_\ell(\nu)^2,
\end{equation}
where the cardinality $\#\MMM_\ell$ is minimal (amongst all sets which satisfy the marking criterion~\eqref{eq:doerfler:separate2}).
\end{itemize}
\end{itemize}
\textbf{Output:}
$\MM_{\ell\nu} \subseteq \NN_{\ell\nu}^+$ for all $\nu \in \PPP_\ell$ and
$\MMM_\ell \subseteq \QQQ_\ell$.
\end{marking}

\begin{marking} \label{marking:B}
\textbf{Input:}
error indicators $\{ \est_{\ell}(\nu,z) : \nu \in \PPP_\ell,\ z \in \NN_{\ell\nu}^+ \}$ and
$\{ \est_{\ell}(\nu) : \nu \in \QQQ_\ell \}$;
marking parameters $0 < \theta_\X, \theta_\PPP \le 1$, and $\vartheta > 0$.
\begin{itemize}
\item[$\bullet$]
Determine $\widetilde\MM_{\ell\nu} \subseteq \NN_{\ell\nu}^+$ for all $\nu \in \PPP_\ell$ such that
\begin{equation}
 \label{eq:doerfler:separate1-2}
 \theta_\X \, \sum_{\nu \in \PPP_\ell} \sum_{z \in \NN_{\ell\nu}^+} \est_{\ell}(\nu,z)^2 
 \le \sum_{\nu \in \PPP_\ell} \sum_{z \in \widetilde\MM_{\ell\nu}} \est_{\ell}(\nu,z)^2,
\end{equation}
where the cumulative cardinality $\sum_{\nu \in \PPP_\ell} \# \widetilde\MM_{\ell\nu}$ is minimal
(amongst all sets which satisfy the marking criterion~\eqref{eq:doerfler:separate1-2}).
\item[$\bullet$]
Define $\widetilde\RR_{\ell\nu} := \NN_{\ell\nu}^+ \cap \widetilde\NN_{\ell\nu}$
for all $\nu \in \PPP_\ell$,
where $\widetilde\NN_{\ell\nu}$ is the set of vertices of $\widetilde\TT_{\ell\nu} = \refine(\TT_{\ell\nu},\widetilde\MM_{\ell\nu})$.
\item[$\bullet$]
Determine $\widetilde\MMM_\ell \subseteq \QQQ_\ell$ such that
\begin{equation}
 \label{eq:doerfler:separate2-2}
 \theta_\PPP \, \sum_{\nu \in \QQQ_\ell} \est_\ell(\nu)^2 \le \sum_{\nu \in \widetilde\MMM_\ell} \est_\ell(\nu)^2,
\end{equation}
where the cardinality $\#\widetilde\MMM_\ell$ is minimal
(amongst all sets which satisfy the marking criterion~\eqref{eq:doerfler:separate2-2}).
\item[$\bullet$]
If \ 
$\vartheta \sum_{\nu \in \widetilde\MMM_\ell} \est_\ell(\nu)^2 \le
  \sum_{\nu \in \PPP_\ell} \sum_{z \in \widetilde\RR_{\ell\nu}} \est_{\ell}(\nu,z)^2$,
then proceed as follows:
\begin{itemize}
\item[$\circ$]
set $\MMM_\ell := \emptyset$
and
$\MM_{\ell\nu} := \widetilde\MM_{\ell\nu}$ for all $\nu \in \PPP_\ell$.
\end{itemize}
\item[$\bullet$]
Otherwise, if \
$\vartheta \sum_{\nu \in \widetilde\MMM_\ell} \est_\ell(\nu)^2 >
  \sum_{\nu \in \PPP_\ell} \sum_{z \in \widetilde\RR_{\ell\nu}} \est_{\ell}(\nu,z)^2$,
then proceed as follows:
\begin{itemize}
\item[$\circ$]
set
$\MMM_\ell := \widetilde \MMM_\ell$
and
$\MM_{\ell\nu} := \emptyset$ for all $\nu \in \PPP_\ell$.
\end{itemize}
\end{itemize}
\textbf{Output:}
$\MM_{\ell\nu} \subseteq \NN_{\ell\nu}^+$ for all $\nu \in \PPP_\ell$ and
$\MMM_\ell \subseteq \QQQ_\ell$.
\end{marking}

\subsection{Combined marking/enrichment} \label{sec:alg2}

In the case of single-level approximation spaces
(where $\TT_{\coarse\nu} = \TT_\coarse$ for all $\nu \in \PPP_\coarse \cup \QQQ_\coarse$),
a combined enrichment of spatial and parametric components at each iteration of the adaptive algorithm
is prohibitively expensive due to the multiplicative increase of the total number of degrees of freedom
(i.e., $\dim \V_\coarse = (\#\PPP_\coarse) \cdot \dim \SS^1_0(\TT_\coarse)$).
The situation is considerably different for
multilevel approximation spaces defined by~\eqref{eq:def:V},
for which combined enrichment always
results in \emph{additive} increase in the total number of degrees of freedom, i.e.,
$\dim \V_\coarse = \sum_{\nu \in \PPP_\coarse} \dim \SS^1_0(\TT_{\coarse\nu})$.
In the context of Algorithm~\ref{algorithm},
this enrichment is steered by the D\"orfler marking performed on the joint set of all spatial and parametric error indicators,
as presented in the following marking criterion.

\begin{marking} \label{marking:C}
\textbf{Input:}
error indicators $\{ \est_{\ell}(\nu,z) : \nu \in \PPP_\ell,\ z \in \NN_{\ell\nu}^+ \}$ and
$\{ \est_{\ell}(\nu) : \nu \in \QQQ_\ell \}$;
marking parameter $0 < \theta \le 1$.
\begin{itemize}
\item[$\bullet$]
Determine the sets $\MM_{\ell\nu} \subseteq \NN_{\ell\nu}^+$ for all $\nu \in \PPP_\ell$
and the set $\MMM_\ell \subseteq \QQQ_\ell$ such that
\begin{equation}\label{eq:doerfler:combined}
 \theta \, \bigg( \sum_{\nu \in \PPP_\ell} \sum_{z \in \NN_{\ell\nu}^+} \est_{\ell}(\nu,z)^2 + \sum_{\nu \in \QQQ_\ell} \est_\ell(\nu)^2 \bigg)
 \le \sum_{\nu \in \PPP_\ell} \sum_{z \in \MM_{\ell\nu}}  \est_{\ell}(\nu,z)^2 + \sum_{\nu \in \MMM_\ell} \est_\ell(\nu)^2,
\end{equation}
where the overall cardinality $\#\MMM_\ell + \sum_{\nu \in \PPP_\ell} \# \MM_{\ell\nu}$ is minimal (amongst all sets which satisfy the marking criterion~\eqref{eq:doerfler:combined}).
\end{itemize}
\textbf{Output:}
$\MM_{\ell\nu} \subseteq \NN_{\ell\nu}^+$ for all $\nu \in \PPP_\ell$ and
$\MMM_\ell \subseteq \QQQ_\ell$.
\end{marking}

In what follows, we will write, e.g., Algorithm~\ref{algorithm}.\ref{marking:A} to refer to the algorithm obtained
by employing Marking criterion~\ref{marking:A} in Step~(iv) of Algorithm~\ref{algorithm}.

\section{Computing multilevel stochastic Galerkin approximations: implementation aspects} \label{sec:implement}

The adaptive multilevel strategies outlined in section~\ref{sec:algorithms} are implemented within
the open-source MATLAB toolbox Stochastic T-IFISS~\cite{BespalovR_stoch_tifiss}.
The toolbox has been developed as an extension of the FEM software package T-IFISS~\cite{tifiss}
to compute stochastic Galerkin approximations of PDE problems with parametric or uncertain inputs.
Overall, this software aims at creating an environment
for testing different discretization and error estimation strategies,
exploring new algorithms, as well as for replication, validation and verification of computational results
(see~\cite{brs20} for a recent review).

In this section, we briefly discuss some implementation aspects of the multilevel stochastic Galerkin FEM.
In particular, we focus on assembling components of the Galerkin matrix and solving the resulting linear system.

\subsection{Matrix formulation of the multilevel stochastic Galerkin FEM} \label{sec:matrix:form}

For each $\mu \in \PPP_\coarse$,
we denote by $N_{\coarse\mu}$ the dimension of the finite element space $\X_{\coarse\mu} = \SS^1_0(\TT_{\coarse\mu})$
(i.e., $N_{\coarse\mu} = \# (\NN_{\coarse\mu} \setminus \partial D)$).
Recalling~\eqref{eq:def:V}, the multilevel stochastic Galerkin approximation $\uu_\coarse \in \V_\coarse$
can be represented as follows:
\begin{equation} \label{eq:u:decomp}
   \uu_\coarse(x, \y)
   = \sum_{\mu \in \PPP_\coarse} \sum_{j=1}^{N_{\coarse\mu}} u_{\coarse\mu,z_j} \varphi_{\coarse\mu,z_j}(x) P_{\mu}(\y).
\end{equation}
Hence, by taking test functions $\vv_\coarse = \varphi_{\coarse\nu,z_i} P_{\nu}$ for all $\nu \in \PPP_\coarse$
and all $i = 1,2,\ldots,N_{\coarse\nu}$,
the discrete formulation~\eqref{eq:discrete_formulation} yields a linear system $\sfA \sfu = \sfb$
for finding the unknown coefficients $u_{\coarse\mu,z_j} \in \R$ in~\eqref{eq:u:decomp}.

Since the approximation space $\V_\coarse$ is built from tensor products of different subspaces
of $\X = H^1_0(D)$ and $\P = L^2_\pi(\Gamma)$ (see~\eqref{eq:def:V}),
the matrix $\sfA$ and the vectors $\sfu$ and $\sfb$ have block structure,
with individual blocks indexed by multi-indices of~$\PPP_\coarse$ as follows:
\[
   \R^{N_\coarse \times N_\coarse} \ni \sfA = (\sfA_{\nu \mu})_{\nu,\mu \in \PPP_\coarse},\qquad
   \R^{N_\coarse} \ni \sfb = (\sfb_{\nu})_{\nu \in \PPP_\coarse},\qquad
   \R^{N_\coarse} \ni \sfu = (\sfu_{\mu})_{\mu \in \PPP_\coarse},
\]
where
$N_\coarse := \dim \V_\coarse = \sum_{\nu \in \PPP_\coarse} N_{\coarse\nu}$,
\begin{equation*}
[\sfA_{\nu \mu}]_{i j} = [\sfA_{\mu \nu}]_{j i} =
B(\varphi_{\coarse\mu,z_j}P_{\mu}, \varphi_{\coarse\nu,z_i}P_{\nu}),\qquad 
[\sfb_{\nu}]_i = F(\varphi_{\coarse\nu,z_i} P_{\nu}),\qquad 
[\sfu_{\mu}]_j = u_{\coarse\mu,z_j}
\end{equation*}
for $i = 1, \dots, N_{\coarse\nu}$ and $j = 1, \dots, N_{\coarse\mu}$.
Hence, recalling~\eqref{def:B0},~\eqref{def:B},~\eqref{eq1:lemma:orthogonal} and~\eqref{eq:weakform}, we find
\begin{equation*}
   \begin{split}
       [\sfA_{\nu \mu}]_{i j}
       & =
       \delta_{\nu\mu} \int_D a_0(x) \nabla \varphi_{\coarse\mu,z_j}(x)\cdot\nabla \varphi_{\coarse{\nu},z_i}(x)\,\d{x} \\
       & \quad
       + \sum_{m=0}^{\infty} \int_\Gamma y_m P_{\mu}(\y) P_{\nu}(\y) \,\d\pi(\y)
       \int_D a_m(x) \nabla \varphi_{\coarse\mu,z_j}(x)\cdot\nabla \varphi_{\coarse{\nu},z_i}(x)\, \d{x}
   \end{split}
\end{equation*}
and
\[
   [\sfb_{\nu}]_i =
   \int_\Gamma \int_D \ff(x,\y) \varphi_{\coarse\nu,z_i}(x) P_{\nu}(\y) \, \d{x} \, \d\pi(\y).
\]
Thus, for all $\nu,\mu \in \PPP_\coarse$, the $\nu\mu$-th block in the Galerkin matrix $\sfA$ is given by
\begin{equation} \label{eq:A:block}
   \sfA_{\nu \mu}
   = \sum_{m=0}^{\infty} [G_m]_{\nu\mu} K_m^{\nu\mu}
   = \sum_{m=0}^{M} [G_m]_{\nu\mu} K_m^{\nu\mu},
\end{equation}
where, for $m \in \N_0$,
\begin{equation} \label{eq:G:matrix}
[G_m]_{\nu \mu} =
\begin{cases}
\delta_{\nu\mu} & \text{if } m=0,\\
\int_\Gamma y_m P_{\mu}(\y) P_{\nu}(\y) \, \d\pi(\y)
       \reff{eq:3005:three-term}=
\beta_{\mu_m}^m \delta_{\mu+\eps_m,\nu} + \beta_{\mu_m-1}^m \delta_{\mu-\eps_m,\nu}
& \text{if } m \in \N
\end{cases}
\end{equation}
and
$K_m^{\nu\mu}$ are the finite element (stiffness) matrices defined by
\begin{equation} \label{eq:K:matrix}
   [K_m^{\nu\mu}]_{ij} =
   \int_D a_m(x) \nabla \varphi_{\coarse\mu,z_j}(x)\cdot\nabla \varphi_{\coarse{\nu},z_i}(x) \, \d{x}
\end{equation}
for $i = 1,\ldots,N_{\coarse\nu}$ and $j = 1,\ldots,N_{\coarse\mu}$,
whereas
$M = \# \supp(\PPP_\coarse)$ is the number of active parameters in $\PPP_\coarse$;
here, we used the fact that $G_m = 0$ for all $m \notin \supp(\PPP_\coarse)$
(due to the symmetry of the measure $\pi_m$ on $\Gamma_m = [-1,1]$ for all $m \in \N$)
and implicitly assumed that $\supp(\PPP_\coarse) = \{1,2,\ldots,M\}$.
For a detailed study of the properties of the matrices $\{G_m\}_{m=1}^{M}$, we refer, e.g., to~\cite{eu2010}.

At first glance, there are $(M+1) (\# \PPP_\coarse)^2$
stiffness matrices to compute; see~\eqref{eq:A:block}.
However, as discussed in~\cite[Section~3.1]{cpb18+}, the actual number of matrices that need to be computed
is significantly less.
Indeed, it follows from~\eqref{eq:A:block} that
one only needs to compute the matrix $K_m^{\nu\mu}$ if the corresponding entry $[G_m]_{\nu\mu}$ is nonzero.
The matrices $G_m$ are very sparse:
while $G_0$ is the identity matrix, it follows from~\eqref{eq:G:matrix} that the matrices $\{G_m\}_{m=1}^{M}$
have at most two nonzero entries per row (see also~\cite[Theorem~9.59]{lps2014}).
This reduces the number of stiffness matrices to be computed to $(2M + 1) \# \PPP_\coarse$ at most.
Furthermore, since the measure $\pi_m$ is symmetric on $\Gamma_m = [-1,1]$ for all $m \in \N$,
the matrices $G_m$, $m \in \N$, are also symmetric and have zero diagonal entries.
In addition to the sparsity and symmetry of $G_m$, we observe that
$K_m^{\nu\mu} = (K_m^{\mu\nu})^{\sf T}$ for all $m = 0,1, \ldots, M$ and $\nu,\mu \in \PPP_\coarse$.
Therefore, the number of stiffness matrices one actually needs to compute is at most $(M + 1) \# \PPP_\coarse$.

\subsection{Computation of stiffness matrices} \label{sec:assembly}

Let us now address the computation of the stiffness matrices $K_m^{\nu\mu}$ given by~\eqref{eq:K:matrix}.
To that end, we fix $m \in \{1,2,\ldots,M\}$ (the computation process is the same for each $m$)
and set $\mu = \nu \pm \eps_m \in \PPP_\coarse$ for some $\nu \in \PPP_\coarse$; cf.~\eqref{eq:G:matrix}.
Note that the entries of $K_m^{\nu\mu}$ are the spatial integrals involving finite element basis functions
associated with the meshes $\TT_{\coarse\nu}$ and $\TT_{\coarse\mu}$, which may be different and not necessarily nested.
As a consequence, $K_m^{\nu\mu}$ are in general non-square if $\TT_{\coarse\nu} \neq \TT_{\coarse\mu}$,
and efficient computation of these matrices is the main difficulty in the implementation of
the multilevel stochastic Galerkin FEM.

The assembly of stiffness matrices in the context of the multilevel stochastic Galerkin FEM
has been previously discussed in~\cite{gittelson13,egsz14,cpb18+}.
In~\cite{gittelson13}, the action of any non-square stiffness matrix $K_m^{\nu(\nu\pm\eps_m)}$
(in the context, e.g., of the preconditioned conjugate gradient method) is \emph{approximated}
via a projection $\Pi_{\nu}^{\nu\pm\eps_m}:\, \X_{\coarse (\nu\pm\eps_m)} \to \X_{\coarse\nu}$,
such that only square matrices $K_m^{\nu\nu}$ need to be assembled.
A more elaborate and computationally expensive approach involving the union of meshes
$\TT_{\coarse\nu}$ and $\TT_{\coarse (\nu\pm\eps_m)}$ is proposed in~\cite[Section~10]{egsz14}.
Again, only square stiffness matrices need to be assembled.
On the other hand, assuming that the meshes $\{\TT_{\coarse\nu}:\; \nu \in \PPP_\coarse\}$
(and, hence, the corresponding finite element spaces $\X_{\coarse\nu}$ in~\eqref{eq:def:V}) are \emph{nested},
it is shown in~\cite{cpb18+} that non-square stiffness matrices $K_m^{\nu\mu}$
can be computed quickly and efficiently without resorting to approximations involving square matrices.

In our implementation,
we aim for \emph{direct} computation of non-square stiffness matrices
$K_m^{\nu\mu}$ for a pair of general, not necessarily nested, meshes $\TT_{\coarse\nu} \neq \TT_{\coarse\mu} \in \refine(\TT_0)$
($\nu,\, \mu = \nu\pm\eps_m \in \PPP_\coarse$).

First, exploiting the fact that the finite element basis functions $\varphi_{\coarse\nu,z}$ in our construction of $\V_\coarse$
are piecewise linear, we find
\begin{align}
      [K_m^{\nu\mu}]_{ij}
      & \reff{eq:K:matrix}=
      \int_D a_m(x) \nabla \varphi_{\coarse\mu,z_j} \cdot \nabla \varphi_{\coarse\nu,z_i} \, \d{x}
      \nonumber
      \\
      & \refp{eq:K:matrix}=
      \sum_{T_\nu \in \TT_{\coarse\nu}}
      \sum_{\stackrel{T_\mu \in \TT_{\coarse\mu}}{|T_\mu \cap T_\nu| \neq 0}}
      \big(\nabla \varphi_{\coarse\mu,z_j}|_{T_\mu} \cdot \nabla \varphi_{\coarse\nu,z_i}|_{T_\nu}\big)
      \int_{T_\mu \cap T_\nu} a_m(x) \, \d{x}.
      \label{eq:assembly}
\end{align}
Thus, efficient identification of all intersections $T_\mu \cap T_\nu$ is critical for the whole computation.
The key observation here is that NVB is a binary refinement rule.
Note that every element $T \in \TT_\coarse \in \refine(\TT_0)$ naturally comes with a
\emph{level} that can be defined in the following inductive way:
\begin{itemize}
\item
for all $T \in \TT_0$, define $\level(T) := 0$;
\item
if $T \in \TT_\coarse \in \refine(\TT_0)$ is bisected into two elements $T_1$ and $T_2$,
then define $\level(T_1) := \level(T) + 1 =: \level(T_2)$.
\end{itemize}
Now, for any
$T \in \TT_\coarse \in \refine(\TT_0)$,
we denote by $T_0(T)$ the unique element of the initial mesh $\TT_0$
such that $T \subseteq T_0(T)$.
Then, the above definition implies that
\begin{equation}\label{eq:level}
   \vert T \vert / \vert T_0(T) \vert = 2^{-\level(T)}.
\end{equation}
Furthermore, there holds the following lemma, which, in particular,
proves that the intersection $T_\mu \cap T_\nu$ is either $T_\mu$, or $T_\nu$, or
a set of measure zero.

\begin{lemma}
Let $\TT_\bullet, \TT_\bullet' \in \refine(\TT_0)$.
Let $T \in \TT_\bullet$ and $T' \in \TT_\bullet'$.
Let $s_T \in T$ denote the center of mass of $T$.
Then, there hold the following statements~\rm(i)--(ii):
\begin{itemize}
\item[\rm(i)] If $\level(T) = \level(T')$, then there holds either $T = T'$ or $|T \cap T'| = 0$.
Moreover, $T = T'$ is equivalent to $s_T \in {\rm interior}(T')$.
\item[\rm(ii)] If $\level(T) > \level(T')$, then there holds either $T \subsetneqq T'$ or $|T \cap T'| = 0$.
Moreover, $T \subsetneqq T'$ is equivalent to $s_T \in {\rm interior}(T')$.
\end{itemize}
\end{lemma}

\begin{proof}
Since NVB is a binary refinement rule, the intersection
$T \cap T'$ satisfies one of the following four conditions:
\begin{itemize}
\item $|T \cap T'| = 0$;
\item $T \cap T' = T = T'$;
\item $T \cap T' = T \subsetneqq T'$;
\item $T \cap T' = T' \subsetneqq T$. 
\end{itemize}
Due to~\eqref{eq:level}, knowing the element's level is sufficient for determining its size.
Moreover, the center of mass of an element always lies in the interior of all of its NVB ancestors. 
\end{proof}

Thus, given two meshes $\TT_{\coarse\nu}, \TT_{\coarse\mu} \in \refine(\TT_0)$ for $\mu \neq \nu$,
the computation of the matrix entries $[K_m^{\nu\mu}]_{ij}$ in~\eqref{eq:assembly} essentially boils down to the construction of two sets
$\UU_{\nu\mu},\, \UU_{\mu\nu}^\circ \subset \TT_{\coarse\nu} \times \TT_{\coarse\mu}$
satisfying the following properties (U1)--(U3):
\begin{itemize}
\item[(U1)]
For all $(T_{\nu},T_{\mu}) \in \, \UU_{\nu\mu}$, there holds $T_{\nu} \subseteq T_{\mu}$;
\item[(U2)]
For all $(T_{\nu},T_{\mu}) \in \, \UU_{\mu\nu}^\circ$, there holds $T_{\mu} \subsetneqq T_{\nu}$;
\item[(U3)]
$\TT_{\coarse\nu} \oplus \TT_{\coarse\mu}
:= \{ T_{\nu} : (T_{\nu},T_{\mu}) \in \, \UU_{\nu\mu} \}
\cup \{ T_{\mu} : (T_{\nu},T_{\mu}) \in \, \UU_{\mu\nu}^\circ \}$
is a mesh\footnote{Note that the notation
used in~(U3) is deliberate, in the sense that
$\TT_{\coarse\nu} \oplus \TT_{\coarse\mu}$ is indeed the overlay
of the meshes $\TT_{\coarse\nu}$ and $\TT_{\coarse\mu}$
(i.e., their coarsest common refinement).}
of $D$.
\end{itemize}
Indeed, with the sets $\UU_{\nu\mu},\, \UU_{\mu\nu}^\circ$ at hand, the formula~\eqref{eq:assembly}
for computing $[K_m^{\nu\mu}]_{ij}$ can be written as follows:
\begin{align*}
[K_m^{\nu\mu}]_{ij}
& = 
\sum_{(T_{\nu},T_{\mu}) \in \, \UU_{\nu\mu}}
\big( \nabla \varphi_{\bullet\mu,z_j} \vert_{T_{\mu}} \cdot \nabla \varphi_{\bullet\nu,z_i} \vert_{T_{\nu}} \big)
\int_{T_{\nu}} a_m(x) \, \d{x} \\
& \quad\ + \sum_{(T_{\nu},T_{\mu}) \in \, \UU_{\mu\nu}^\circ}
\big( \nabla \varphi_{\bullet\mu,z_j} \vert_{T_{\mu}} \cdot \nabla \varphi_{\bullet\nu,z_i} \vert_{T_{\nu}} \big)
\int_{T_{\mu}} a_m(x) \, \d{x}.
\end{align*}
The following searching algorithm provides a simple and surprisingly
effective strategy for constructing the sets
$\UU_{\nu\mu}$ and $\UU_{\mu\nu}^\circ$.
In this algorithm, for each simplex $T$, we denote by
$s_{T} \in T$ the center of mass of $T$.
Furthermore,
we denote by $\lambda_{T,1}(x)$, $\lambda_{T,2}(x)$, $\lambda_{T,3}(x)$
the barycentric coordinates of $x \in D$ with respect to $T$,
i.e., $x = \sum_{j=1}^3 \lambda_{T,j}(x) z_{T,j}$
and $\sum_{j=1}^3 \lambda_{T,j}(x) =1$,
where
$z_{T,1}$, $z_{T,2}$, $z_{T,3}$
are the vertices of $T$.
We recall that
$\lambda_{T,j}(x)$ are uniquely defined for given $x$ and $T$, and
$x \in T$ is equivalent to $\lambda_{T,1}(x), \lambda_{T,2}(x), \lambda_{T,3}(x) \ge 0$.

\begin{algorithm}[construction of $\UU_{\nu\mu}$ and $\UU_{\mu\nu}^\circ$] \label{alg:mesh_intersection}
\textbf{Input:} Meshes $\TT_{\coarse\nu}$ and $\TT_{\coarse\mu}$.
\begin{algorithmic}[1]
\FORALL{$T_0 \in \TT_0$}
\STATE Define $\TT_{\coarse\nu} \vert_{T_0} : = \{T_{\nu} \in \TT_{\coarse\nu} :
T_{\nu} \subseteq T_0\} \subseteq \TT_{\coarse\nu}$.
\STATE Define $\TT_{\coarse\mu} \vert_{T_0} : = \{T_{\mu} \in \TT_{\coarse\mu} :
T_{\mu} \subseteq T_0 \} \subseteq \TT_{\coarse\mu}$.
\FORALL{$T_\nu \in \TT_{\coarse\nu} \vert_{T_0}$}
\STATE Define $\mathcal{V}_{\coarse\mu}(T_{\nu}) := \{T_{\mu} \in \TT_{\coarse\mu} \vert_{T_0} :
\level(T_{\mu}) \leq \level(T_{\nu})\} \subseteq \TT_{\coarse\mu} \vert_{T_0}$.
\STATE Compute $\lambda_{T_{\mu},i}(s_{T_{\nu}})$ for all $i=1,2,3$ and $T_{\mu} \in \mathcal{V}_{\coarse\mu}(T_{\nu})$.
\IF{there exists (a unique) $T_{\mu} \in \mathcal{V}_{\coarse\mu}(T_{\nu})$ with $\lambda_{T_{\mu},i}(s_{T_{\nu}})>0$ for all $i=1,2,3$}
\STATE Assign $(T_{\nu},T_{\mu})$ to $\UU_{\nu\mu}$ (because $T_{\nu} \subseteq T_{\mu}$).
\ELSE
\STATE Compute $\lambda_{T_{\nu},i}(s_{T_{\mu}})$ for all $i=1,2,3$ and $T_{\mu} \in \TT_{\coarse\mu} \vert_{T_0} \setminus \mathcal{V}_{\coarse\mu}(T_{\nu})$.
\STATE Define $\mathcal{W}_{\coarse\mu}(T_{\nu}) := \{ T_{\mu} \in \TT_{\coarse\mu} \vert_{T_0} \setminus \mathcal{V}_{\coarse\mu}(T_{\nu}) : \lambda_{T_{\nu},i}(s_{T_{\mu}})>0 \text{ for all } i=1,2,3 \}$.
\STATE Assign $(T_{\nu},T_{\mu})$ to $\UU_{\mu\nu}^\circ$ for all $T_{\mu} \in \mathcal{W}_{\coarse\mu}(T_{\nu})$
(because $T_{\mu} \subsetneqq T_{\nu}$ if $T_{\mu} \in \mathcal{W}_{\coarse\mu}(T_{\nu})$).
\ENDIF
\ENDFOR
\ENDFOR
\end{algorithmic}
\textbf{Output:} Sets $\UU_{\nu\mu}$ and $\UU_{\mu\nu}^\circ$ satisfying \rm(U1)--(U3).
\end{algorithm}

Algorithm~\ref{alg:mesh_intersection}
has a computational complexity of $\mathcal{O}((\#\TT_{\coarse\nu})(\#\TT_{\coarse\mu}))$
in the worst case.
However, its only intention is to show that unlike \cite{egsz14}
it is possible to compute stiffness matrices associated to different meshes exactly (up to quadrature).
We conjecture that
one can build the matrix $K_m^{\nu\mu}$ from \eqref{eq:assembly}
in log-linear complexity
$\mathcal{O}((\#\TT_{\coarse\nu} +\#\TT_{\coarse\mu}) \log (\#\TT_{\coarse\nu} +\#\TT_{\coarse\mu}))$
by
exploiting the binary tree structure of NVB.
This aspect of the implementation will be the subject of future research.

\subsection{Numerical solution of Galerkin system} \label{sec:solver}

Efficient linear solver is an important ingredient of any stochastic Galerkin implementation.
Sparse factorizations of the (full) system matrix $\sfA$ are memory intensive and computationally costly,
therefore, performing those efficiently is not feasible.
In fact, the coefficient matrix $\sfA$ is never explicitly assembled in stochastic Galerkin FEM implementations
(see, e.g.,~\cite{egsz14,cpb18+,brs20}).
Instead, `matrix-free' iterative solvers are employed, where the matrix-vector products with $\sfA$
are computed blockwise from individual matrix components of $\sfA$ as~follows:
\[
   [\sfA \sfx]_\nu = \sum_{\mu \in \PPP_\coarse} \sfA_{\nu\mu} \sfx_\mu
   \reff{eq:A:block}=
    \sum_{\mu \in \PPP_\coarse} \sum_{m=0}^{M} [G_m]_{\nu\mu} K_m^{\nu\mu} \sfx_\mu,
    \quad
    \sfx = (\sfx_\mu)_{\mu \in \PPP_\coarse},\ \
    \nu \in \PPP_\coarse.
\]
The default iterative solver in Stochastic T-IFISS is a bespoke implementation of the Minimum Residual method,
called EST\!$\_$\!MINRES~\cite{ss11}
(an alternative solver based on the conjugate gradient method and
utilizing the built-in MATLAB function {\tt pcg} is included as an option).

For the iterative solver to be fast, it requires a suitably chosen preconditioner.
In the context of stochastic Galerkin FEM, particularly for parametric PDEs with coefficients having linear dependence on the parameters,
the mean-based preconditioner~\cite{ghanemkruger96,elmanpowell2009} is a standard choice
(for alternative approaches, we refer, e.g., to~\cite{ullmann2010,sg2014,bly2021}).
Specifically, we employ a block-diagonal preconditioner with diagonal blocks given by
the stiffness matrices $K_0^{\nu\nu}$, $\nu \in \PPP_\coarse$, defined in~\eqref{eq:K:matrix}.
Thus, the action of the inverse of the preconditioner on residual vectors can be effected blockwise.
For each $\nu \in \PPP_\coarse$, this is done
by computing sparse triangular factorizations of $K_0^{\nu\nu}$, followed by forward and backward substitutions
on the corresponding block of the residual vector.
In agreement with theoretical results in~\cite{elmanpowell2009} for the \emph{single-level} stochastic Galerkin FEM,
our experiments with \emph{multilevel} approximations have shown
that the number of preconditioned EST\!$\_$\!MINRES iterations needed to satisfy
the default tolerance of \num{e-9} is less than~20, independent of $\#\PPP_\coarse$ and
the resolution of finite element meshes in the multilevel construction.

\section{Numerical experiments} \label{sec:numerics}

In this section, we present a collection of numerical results
that illustrate the effectiveness of the error estimation strategy developed in section~\ref{sec:aposteriori} and demonstrate
the performance of the multilevel adaptive algorithms described in section~\ref{sec:algorithms}.
Here, we stay within the context of the two-dimensional diffusion problem~\eqref{eq:strongform}
with the parametric coefficient $\aa = \aa(x,\y)$ in the affine form~\eqref{eq1:a} satisfying assumptions~\eqref{eq2:a}--\eqref{eq3:a}.
In addition, we assume that the parameters $\y = (y_m)_{m \in \N}$
are images of independent uniformly distributed mean-zero random variables on $[-1,1]$,
i.e., $\d\pi_m(y_m) = \d y_m /2$ for all $m \in \N$.
All computations have been performed using the MATLAB toolbox Stochastic T-IFISS;
see section~\ref{sec:implement}.

In our experiments, we use five adaptive algorithms:
two multilevel algorithms with separate spatial and parametric enrichments
(i.e., Algorithms~\ref{algorithm}.\ref{marking:A} and~\ref{algorithm}.\ref{marking:B} from section~\ref{sec:algorithms}),
their single-level precursors (see, e.g., Algorithms~4.A and~4.B in~\cite{bprr18++}, respectively),
and the novel multilevel algorithm with combined enrichment (Algorithm~\ref{algorithm}.\ref{marking:C}).
For the sake of brevity, we will refer to these five algorithms as
\texttt{ML-A}, \texttt{ML-B}, \texttt{SL-A}, \texttt{SL-B}, and \texttt{ML-C}, respectively.
The parameters in these algorithms are selected as follows:
\begin{itemize}
\item[$\bullet$]
We set the marking parameters
$\theta_\X = \theta_\PPP = 0.5$ in \texttt{ML-A}, \texttt{ML-B}, \texttt{SL-A}, \texttt{SL-B}
and $\theta = 0.5$ in~\texttt{ML-C}.
\item[$\bullet$]
For the parameter $\overline{M}$ in~\eqref{def:Q},
we choose $\overline{M}=1$ in~\S\ref{sec:eigel}
and $\overline{M}=9$ in~\S\ref{sec:cookie}.
\item[$\bullet$]
Except in the last experiment in~\S\ref{sec:cookie},
the parameter $\vartheta$ modulating the choice of the enrichment type
in the algorithms with separate spatial and parametric enrichments
(i.e., \texttt{ML-A}, \texttt{ML-B} and
\texttt{SL-A}, \texttt{SL-B}) is chosen to be $\vartheta = 1$.
\end{itemize}

\subsection{Benchmark problem} \label{sec:eigel}

The following problem has been considered in several works
addressing the numerical approximation of parametric PDEs
(see, e.g., in~\cite{egsz14,egsz15,bs16,em16,br18,cpb18+,bprr18++})
and has thus become a benchmark problem for testing novel discretization strategies.
Let $\ff \equiv 1$ in~\eqref{eq:strongform} and choose the expansion coefficients in~\eqref{eq1:a}
to represent planar Fourier modes of increasing total order;
for $x = (x_1,x_2)$, these coefficients are given by
\begin{equation*}
a_0(x) = 1,
\quad
a_m(x_1,x_2) = A m^{- \sigma} \cos(2\pi\beta_1(m)x_1) \cos(2\pi\beta_2(m)x_2)
\text{ \ for } m \in \N,
\end{equation*}
where $A,\sigma>0$ are constants,
$\beta_1(m) = m - k(m)[k(m)+1]/2$,
$\beta_2(m) = k(m) - \beta_1(m)$,
and $k(m) = \lfloor -1/2 + \sqrt{1/2+2m} \rfloor$.
With this choice, the diffusion coefficient $\aa(x,\y)$ trivially satisfies~\eqref{eq2:a}
with $a_0^{\rm min} = a_0^{\rm max} = 1$.
Furthermore, we set $\sigma = 2$ (yielding a slow decay of the coefficients)
and choose $A = 0.9 / \zeta(\sigma) \approx 0.547$,
so that both inequalities in~\eqref{eq3:a} are satisfied
(here, $\zeta(\cdot)$ denotes the Riemann zeta function).

\subsubsection{Square domain} \label{sec:eigel-square}

Let us numerically solve the benchmark problem on the square domain $D = (0,1)^2$.
For all algorithms, we choose the initial mesh $\TT_0$ to be a uniform mesh of
512 right-angled triangles
and we terminate computations when the error estimate $\est_\ell$ given by~\eqref{eq:def:tau} falls below
the tolerance ${\tt tol} = \num{6e-4}$.

\begin{figure}[t!]
\begin{tikzpicture}
\pgfplotstableread{data/pb2-sl-a.dat}{\one}
\pgfplotstableread{data/pb2-sl-b.dat}{\two}
\begin{semilogxaxis}
[
width = 7.5cm, height=6cm,								
xlabel={number of DOFs, $N_\ell$}, 					
ylabel={effectivity index, $\zeta_\ell$},				
ymajorgrids=true, xmajorgrids=true, grid style=dashed,	
xmin=(0.95)*10^(2), xmax=(1.5)*10^(7),						
ymin = 0.67,	 ymax = 0.88,							
legend style={at={(0.5,0.97)},anchor=north, fill=none, draw=none}
]
\addplot[violet,mark=o,mark size=2.5pt]		table[x=dofs, y=effindices]{\one};
\addplot[orange,,mark=square,mark size=2.5pt]	table[x=dofs, y=effindices]{\two};
\legend{
\texttt{SL-A},
\texttt{SL-B},
}
\end{semilogxaxis}
\end{tikzpicture}
\hfill
\begin{tikzpicture}
\pgfplotstableread{data/pb2-ml-a.dat}{\three}
\pgfplotstableread{data/pb2-ml-b.dat}{\four}
\pgfplotstableread{data/pb2-ml-c.dat}{\five}
\begin{semilogxaxis}
[
width = 7.5cm, height=6cm,								
xlabel={number of DOFs, $N_\ell$}, 					
ylabel={effectivity index, $\zeta_\ell$},		
ymajorgrids=true, xmajorgrids=true, grid style=dashed,	
xmin=(0.95)*10^(2), xmax=(1.05)*10^(6),						
ymin = 0.67,	 ymax = 0.88,							
legend style={at={(0.5,0.97)},anchor=north, fill=none, draw=none}
]
\addplot[red,mark=triangle,mark size=3.5pt]		table[x=dofs, y=effindices]{\three};
\addplot[blue,mark=diamond,mark size=3.5pt]		table[x=dofs, y=effindices]{\four};
\addplot[cyan,mark=square,mark size=3.5pt]		table[x=dofs, y=effindices]{\five};
\legend{
\texttt{ML-A},
\texttt{ML-B},
\texttt{ML-C},
}
\end{semilogxaxis}
\end{tikzpicture}
\caption{
Experiments in section~\ref{sec:eigel-square}:
Effectivity indices $\zeta_\ell$ for the error estimates
$\est_\ell$ in the SGFEM approximations generated by single-level (left) and multilevel (right) adaptive algorithms.
}
\label{fig:pb2:effectivity_indices}
\end{figure}
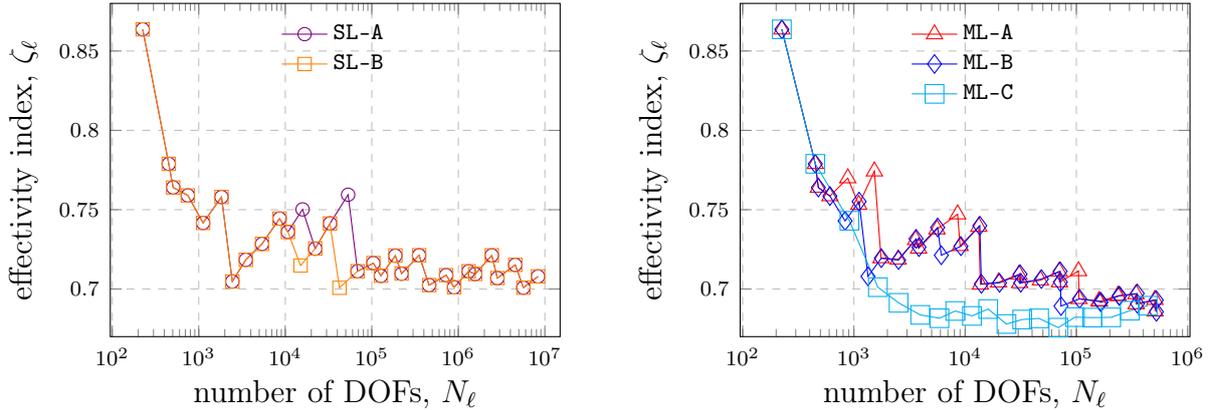

In the first experiment, we assess the effectiveness of our error estimation strategy
by computing the error estimate $\est_\ell$ at each iteration of the adaptive loop and 
comparing $\est_\ell$ with the energy norm of the true error $\uu - \uu_\ell$ approximated by
\begin{equation*}
   \bnorm{\uu - \uu_\ell} =
   \big( \bnorm{\uu}^2 - \bnorm{\uu_\ell}^2 \big)^{1/2} \approx
   \big( \bnorm{\uu_{\mathrm{ref}}}^2 - \bnorm{\uu_\ell}^2 \big)^{1/2}.
\end{equation*}
Here, the equality follows from the Galerkin orthogonality
and the unknown energy $\bnorm{\uu}$ is approximated by the energy of
a sufficiently accurate reference solution $\uu_\mathrm{ref}$
computed with quadratic~(Q2) SGFEM approximations; cf.~\cite[Section~6]{bs16}.
The \emph{effectivity index}
\begin{equation*}
   \zeta_\ell :=
   \frac{\est_\ell}{\big(\bnorm{\uu_\mathrm{ref}}^2 - \bnorm{\uu_\ell}^2 \big)^{1/2}}
\end{equation*}
is then computed at each iteration of the adaptive loop.

In Figure~\ref{fig:pb2:effectivity_indices}, for all adaptive algorithms, we plot
the effectivity indices $\zeta_\ell$ versus the total number of degrees of freedom (DOFs) $N_\ell$
in SGFEM approximations.
For each algorithm, the effectivity indices vary in a range between 0.68 and 0.87 throughout all iterations.
The error is therefore slightly underestimated.
For single-level approximations generated by \texttt{SL-A} and \texttt{SL-B}, this is in agreement with
the results presented in~\cite[Figure~3]{bprr18++}.
Thus, this experiment provides a numerical evidence that
in terms of effectivity, our error estimation strategy for multilevel SGFEM approximations is on a par
with similar strategies for single-level approximations.
The presented results also suggest that by employing the
two-level spatial error estimates we underestimate
the true energy error more than by using hierarchical spatial estimates;
see~\cite{br18} and~\cite{cpb18+} for hierarchical spatial estimates in adaptive single-level and multilevel SGFEMs,
respectively.
However, the better accuracy of hierarchical estimators comes at the price of solving extra linear systems
when computing spatial contributions to the total error estimate at each iteration.

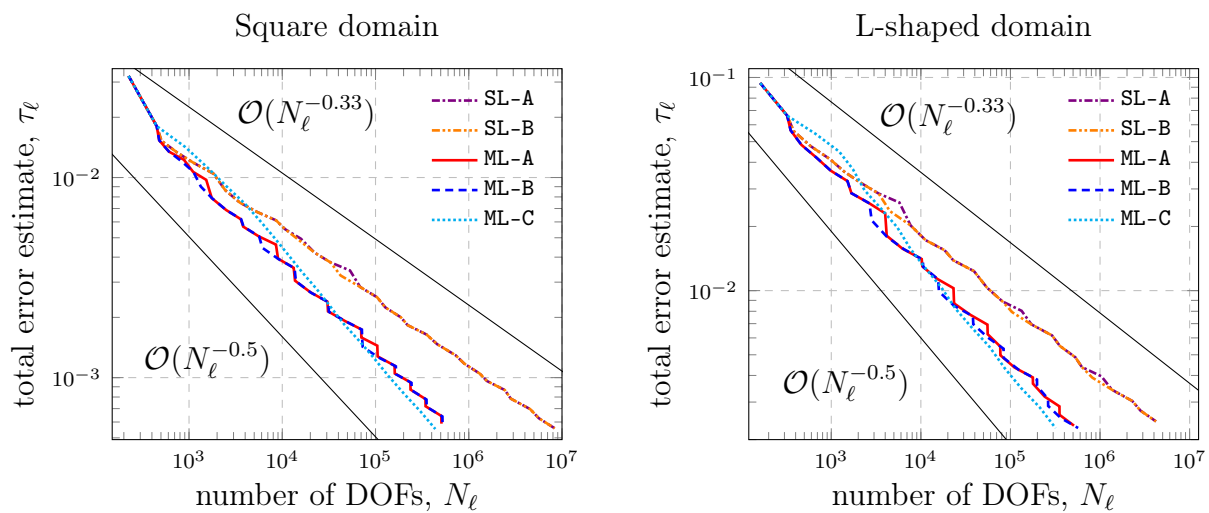
\begin{figure}[b!]
\begin{tikzpicture}
\pgfplotstableread{data/pb2-sl-a.dat}{\one}
\pgfplotstableread{data/pb2-sl-b.dat}{\two}
\pgfplotstableread{data/pb2-ml-a.dat}{\three}
\pgfplotstableread{data/pb2-ml-b.dat}{\four}
\pgfplotstableread{data/pb2-ml-c.dat}{\five}
\begin{loglogaxis}
[
width = 7.5cm, height=6.5cm,						
title={Square domain},							
xlabel={number of DOFs, $N_\ell$}, 					
ylabel={total error estimate, $\est_\ell$},				
ymajorgrids=true, xmajorgrids=true, grid style=dashed,	
xmin = (1.5)*10^(2),
xmax = (1.0)*10^(7),
ymin = (4.9)*10^(-4),
ymax = (3.5)*10^(-2),
legend style={legend pos=north east, legend cell align=left, fill=none, draw=none}
]
\addplot[violet,line width=1.0pt, densely dash dot]		table[x=dofs, y=error]{\one};
\addplot[orange,line width=1.0pt, densely dash dot dot]	table[x=dofs, y=error]{\two};
\addplot[red,line width=1.0pt]						table[x=dofs, y=error]{\three};
\addplot[blue,line width=1.0pt, densely dashed]			table[x=dofs, y=error]{\four};
\addplot[cyan,line width=1.0pt, densely dotted]			table[x=dofs, y=error]{\five};
\addplot[black,solid,domain=10^(2.2):10^(7.2)] { 0.22*x^(-0.33) };
\node at (axis cs:2.5e3,1.5e-2) [anchor=south west] {$\mathcal{O}(N_\ell^{-0.33})$};
\addplot[black,solid,domain=10^(1.2):10^(5.5)] { 0.16*x^(-0.5) };
\node at (axis cs:1e4,1.7e-3) [anchor=north east] {$\mathcal{O}(N_\ell^{-0.5})$};
\legend{
\texttt{SL-A},
\texttt{SL-B},
\texttt{ML-A},
\texttt{ML-B},
\texttt{ML-C},
}
\end{loglogaxis}
\end{tikzpicture}
\hfill
\begin{tikzpicture}
\pgfplotstableread{data/pb5-sl-a.dat}{\one}
\pgfplotstableread{data/pb5-sl-b.dat}{\two}
\pgfplotstableread{data/pb5-ml-a.dat}{\three}
\pgfplotstableread{data/pb5-ml-b.dat}{\four}
\pgfplotstableread{data/pb5-ml-c.dat}{\five}
\begin{loglogaxis}
[
width = 7.5cm, height=6.5cm,								
title={L-shaped domain},	
xlabel={number of DOFs, $N_\ell$}, 						
ylabel={total error estimate, $\est_\ell$},				
ymajorgrids=true, xmajorgrids=true, grid style=dashed,		
xmin = (1.2)*10^(2),
xmax = (1.0)*10^(7.1),
ymin = (2.0)*10^(-3),
ymax = (1.1)*10^(-1),
legend style={legend pos=north east, legend cell align=left, fill=none, draw=none}
]
\addplot[violet,line width=1.0pt, densely dash dot]		table[x=dofs, y=error]{\one};
\addplot[orange,line width=1.0pt, densely dash dot dot]	table[x=dofs, y=error]{\two};
\addplot[red,line width=1.0pt]						table[x=dofs, y=error]{\three};
\addplot[blue,line width=1.0pt, densely dashed]			table[x=dofs, y=error]{\four};
\addplot[cyan,line width=1.0pt, densely dotted]			table[x=dofs, y=error]{\five};
\addplot[black,solid,domain=10^(2.2):10^(7.2)] { 0.75*x^(-0.33) };
\node at (axis cs:2.6e3,4.8e-2) [anchor=south west] {$\mathcal{O}(N_\ell^{-0.33})$};
\addplot[black,solid,domain=10^(2.0):10^(5.2)] { 0.6*x^(-0.5) };
\node at (axis cs:1e4,5e-3) [anchor=north east] {$\mathcal{O}(N_\ell^{-0.5})$};
\legend{
\texttt{SL-A},
\texttt{SL-B},
\texttt{ML-A},
\texttt{ML-B},
\texttt{ML-C},
}
\end{loglogaxis}
\end{tikzpicture}
\caption{
Experiments in sections~\ref{sec:eigel-square} (left) and~\ref{sec:eigel-Lshaped} (right):
Total error estimates $\est_\ell$ versus the number of degrees of freedom $N_\ell$
for all adaptive algorithms.
}
\label{fig:pb2+5:rates}
\end{figure}

Figure~\ref{fig:pb2+5:rates} (left) shows the decay of the error estimates $\est_\ell$
versus the total number of degrees of freedom $N_\ell$ in SGFEM approximations generated by five adaptive algorithms.
For single-level approximations,
the error estimates decay with suboptimal rate $\mathcal{O}(N_\ell^{-0.33})$;
the same rate was observed in~\cite{br18}.
For multilevel approximations, the decay rate is much faster.
In particular, for approximations generated by~\texttt{ML-C}, the error estimates decay with
the optimal rate $\mathcal{O}(N_\ell^{-0.5})$, which is the convergence rate of linear (P1) FEM for the corresponding parameter-free problem.
As a consequence, multilevel SGFEM approximations reach the prescribed accuracy with significantly less degrees of freedom
than their single-level counterparts
(in the asymptotic regime, the number of degrees of freedom in multilevel approximations are less by at least one order of magnitude
compared to the number of degrees of freedom in the single-level approximations having the same accuracy).

\subsubsection{L-shaped domain} \label{sec:eigel-Lshaped}

Let us now consider the benchmark problem on the L-shaped domain $D = (-1,1)^2 \setminus (-1,0]^2$.
In contrast to the problem in~\S\ref{sec:eigel-square}, the exact solution $\uu$ now exhibits a geometric singularity at the reentrant corner.
For this problem, we run all five adaptive algorithms with the same initial mesh $\TT_0$
(a uniform mesh of 384 right-angled triangles)
and the same stopping tolerance ${\tt tol} = \num{2.5e-3}$.

\begin{figure}[b!]
\begin{tikzpicture}
\pgfplotstableread{data/pb5-ml-a.dat}{\one}
\begin{loglogaxis}
[
width = 7.5cm, height=6.5cm,								
title={\texttt{ML-A}},
xlabel={number of DOFs, $N_\ell$},						
ylabel={error estimates},									
ymajorgrids=true, xmajorgrids=true, grid style=dashed,		
xmin = (1.0)*10^(2),
xmax = (7.0)*10^(5),
ymin = (1.2)*10^(-3),
ymax = (1.4)*10^(-1),
legend style={legend pos=south west, legend cell align=left, fill=none, draw=none}
]
\addplot[blue,mark=triangle,mark size=3.0pt,line width=0.5]	table[x=dofs, y=error]{\one};
\addplot[red,mark=square,mark size=2.0pt,line width=0.5]		table[x=dofs, y=yp_one]{\one};
\addplot[cyan,mark=o,mark size=2.0pt,line width=0.5]		table[x=dofs, y=xq_one]{\one};
\addplot[orange,mark=o,mark size=2.0pt,line width=0.5]		table[x=dofs, y=truerr]{\one};
\addplot[black,solid,domain=10^(1.5):10^(7.8)] { 1.7*x^(-0.43) };
\node at (axis cs:3e4,2e-2) [anchor=south west] {$\mathcal{O}(N_\ell^{-0.43})$};
\legend{
$\est_\ell$,
{$\est_{\X_\ell}$},
{$\est_{\PPP_\ell}$},
{$\bnorm{\uu_\mathrm{ref}-\uu_\ell}$},
}
\end{loglogaxis}
\end{tikzpicture}
\hfill
\begin{tikzpicture}
\pgfplotstableread{data/pb5-ml-b.dat}{\one}
\begin{loglogaxis}
[
width = 7.5cm, height=6.5cm,								
title={\texttt{ML-B}},
xlabel={number of DOFs, $N_\ell$},						
ylabel={error estimates},									
ymajorgrids=true, xmajorgrids=true, grid style=dashed,		
xmin = (1.0)*10^(2),
xmax = (7.0)*10^(5),
ymin = (1.2)*10^(-3),
ymax = (1.4)*10^(-1),
legend style={legend pos=south west, legend cell align=left, fill=none, draw=none}
]
\addplot[blue,mark=triangle,mark size=3.0pt,line width=0.5]	table[x=dofs, y=error]{\one};
\addplot[red,mark=square,mark size=2.0pt,line width=0.5]		table[x=dofs, y=yp_one]{\one};
\addplot[cyan,mark=o,mark size=2.0pt,line width=0.5]		table[x=dofs, y=xq_one]{\one};
\addplot[orange,mark=o,mark size=2.0pt,line width=0.5]		table[x=dofs, y=truerr]{\one};
\addplot[black,solid,domain=10^(1.5):10^(7.8)] { 1.7*x^(-0.43) };
\node at (axis cs:3e4,2e-2) [anchor=south west] {$\mathcal{O}(N_\ell^{-0.43})$};
\legend{
$\est_\ell$,
{$\est_{\X_\ell}$},
{$\est_{\PPP_\ell}$},
{$\bnorm{\uu_\mathrm{ref}-\uu_\ell}$},
}
\end{loglogaxis}
\end{tikzpicture}
\hfill
\begin{tikzpicture}
\pgfplotstableread{data/pb5-ml-c.dat}{\one}
\begin{loglogaxis}
[
width = 7.5cm, height=6.5cm,								
title={\texttt{ML-C}},
xlabel={number of DOFs, $N_\ell$},						
ylabel={error estimates},									
ymajorgrids=true, xmajorgrids=true, grid style=dashed,		
xmin = (1.0)*10^(2),
xmax = (5.0)*10^(5),
ymin = (4.0)*10^(-5),
ymax = (1.4)*10^(-1),
legend style={legend pos=south west, legend cell align=left, fill=none, draw=none}
]
\addplot[blue,mark=triangle,mark size=3.0pt,line width=0.5]	table[x=dofs, y=error]{\one};
\addplot[red,mark=square,mark size=2.0pt,line width=0.5]		table[x=dofs, y=yp_one]{\one};
\addplot[cyan,mark=o,mark size=2.0pt,line width=0.5]		table[x=dofs, y=xq_one]{\one};
\addplot[orange,mark=o,mark size=2.0pt,line width=0.5]		table[x=dofs, y=truerr]{\one};
\addplot[black,solid,domain=10^(1.5):10^(7.8)] { 2.9*x^(-0.5) };
\node at (axis cs:3e4,1.5e-2) [anchor=south west] {$\mathcal{O}(N_\ell^{-0.5})$};
\legend{
$\est_\ell$,
{$\est_{\X_\ell}$},
{$\est_{\PPP_\ell}$},
{$\bnorm{\uu_\mathrm{ref}-\uu_\ell}$},
}
\end{loglogaxis}
\end{tikzpicture}
\caption{
Experiments in section~\ref{sec:eigel-Lshaped}:
Decay of the error estimates (total, spatial, and parametric) and the reference errors
computed at each iteration of the adaptive multilevel algorithms.
}
\label{pb5:multilevel_decay}
\end{figure}
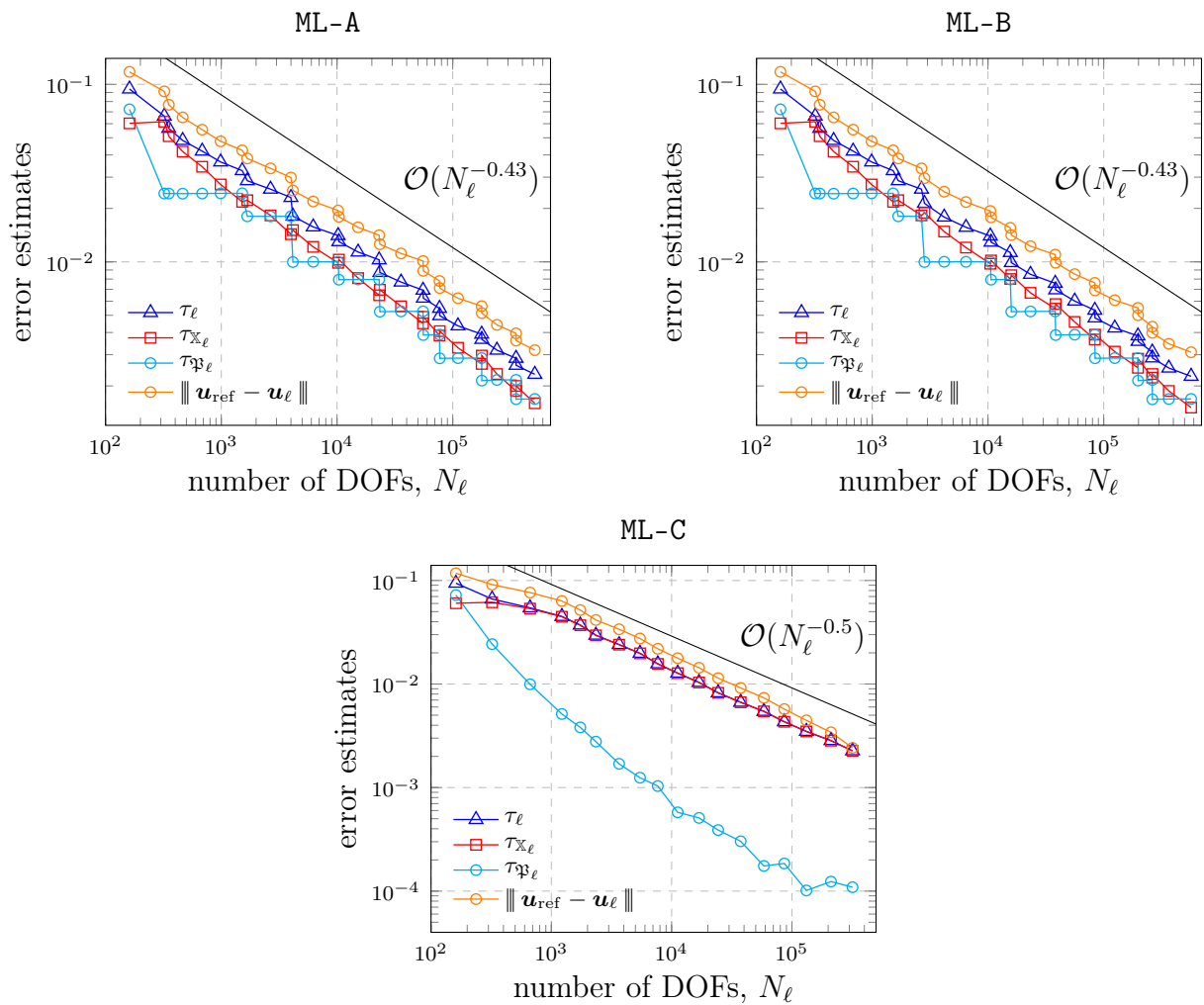

In Figure~\ref{fig:pb2+5:rates} (right), for all adaptive algorithms,
we plot the error estimates $\est_\ell$ against the number of degrees of freedom $N_\ell$.
Despite the singular behavior of the exact solution, we observe the same empirical convergence rates as
in the previous experiment on the square domain.
In particular, the error estimates for all multilevel approximations decay much faster than those for single-level approximations,
while the latter converge with suboptimal rate $\mathcal{O}(N_\ell^{-0.33})$.

\begin{scriptsize}
\begin{table}[t]
\setlength\tabcolsep{4pt} 
\begin{center} 
\renewcommand{\arraystretch}{1.45}
\begin{tabular}{r !{\vrule width 1.0pt} c  c !{\vrule width 1.0pt} c c !{\vrule width 1.0pt} c c} 
\noalign{\hrule height 1.0pt}
&\multicolumn{2}{c!{\vrule width 1.0pt}}{\texttt{ML-A}} 
&\multicolumn{2}{c!{\vrule width 1.0pt}}{\texttt{ML-B}}
&\multicolumn{2}{c}{\texttt{ML-C}} \\	
\noalign{\hrule height 1.0pt}
$L$						&\multicolumn{2}{c!{\vrule width 1.0pt}}{28}		
						&\multicolumn{2}{c!{\vrule width 1.0pt}}{28}
						&\multicolumn{2}{c}{17}\\[-5pt]
$\est_L$					&\multicolumn{2}{c!{\vrule width 1.0pt}}{$2.32526 \cdot 10^{-3}$}
						&\multicolumn{2}{c!{\vrule width 1.0pt}}{$2.26684 \cdot 10^{-3}$}
						&\multicolumn{2}{c}{$2.26429 \cdot 10^{-3}$}\\[-5pt]
$N_L$					&\multicolumn{2}{c!{\vrule width 1.0pt}}{\num{511812}}
						&\multicolumn{2}{c!{\vrule width 1.0pt}}{\num{569321}}
						&\multicolumn{2}{c}{\num{318897}}\\[-5pt]
$\#\PPP_L$				&\multicolumn{2}{c!{\vrule width 1.0pt}}{17}
						&\multicolumn{2}{c!{\vrule width 1.0pt}}{17}
						&\multicolumn{2}{c}{207}\\[-5pt]
$\deg\PPP_L$				&\multicolumn{2}{c!{\vrule width 1.0pt}}{4}
						&\multicolumn{2}{c!{\vrule width 1.0pt}}{4}
						&\multicolumn{2}{c}{7}\\[-5pt]
$M_{\PPP_L}$	&\multicolumn{2}{c!{\vrule width 1.0pt}}{7}
						&\multicolumn{2}{c!{\vrule width 1.0pt}}{7}
						&\multicolumn{2}{c}{17}\\
\hline
$\PPP_\ell$	
&$\ell=0$	  	&$(0\ 0)$				&$\ell=0$		&$(0\ 0)$			&$\ell=0$	&$(0\ 0)$	\\[-1pt]
&$\ell=1$	  	&$(1\ 0)$				&$\ell=1$		&$(1\ 0)$			&$\ell=1$	&$(1\ 0)$	\\[-1pt]
&$\ell=7$		&$(0\ 1)$				&$\ell=7$		&$(0\ 1)$			&$\ell=2$	&$(0\ 1)$\\[-5pt]
&			&					&			&				&		&$(2\ 0)$\\[-1pt]
&$\ell=10$	&$(2\ 0)$				&$\ell=9$		&$(2\ 0)$			&$\ell=3$	&$(0\ 0\ 1)$\\[-5pt]
&			&					&			&				&		&$(1\ 1\ 0)$\\[-5pt]
&			&					&			&				&		&$(3\ 0\ 0)$\\[-1pt]
&$\ell=13$	&$(0\ 0\ 1)$			&$\ell=13$	&$(0\ 0\ 1)$		&$\ell=4$	&$(0\ 0\ 0\ 1)$\\[-5pt]
&			&					&			&				&		&$(1\ 0\ 1\ 0)$\\[-1pt]
&$\ell=16$	&$(1\ 1\ 0)$			&$\ell=15$	&$(1\ 1\ 0)$		&$\ell=5$	&$(0\ 0\ 0\ 0\ 1)$\\[-5pt]
&			&$(3\ 0\ 0)$			&			&$(3\ 0\ 0)$		&		&$(2\ 1\ 0\ 0\ 0)$\\[-1pt]
&$\ell=19$	&$(0\ 0\ 0\ 1)$			&$\ell=18$	&$(0\ 0\ 0\ 1)$		&$\ell=6$	&$(0\ 0\ 0\ 0\ 0\ 1)$\\[-5pt]
&			&$(1\ 0\ 1\ 0)$			&			&$(1\ 0\ 1\ 0)$		&		&$(1\ 0\ 0\ 1\ 0\ 0)$\\[-5pt]
&			&					&			&				&		&$(2\ 0\ 1\ 0\ 0\ 0)$\\[-5pt]
&			&					&			&				&		&$(0\ 2\ 0\ 0\ 0\ 0)$\\[-5pt]
&			&					&			&				&		&$(4\ 0\ 0\ 0\ 0\ 0)$\\[-1pt]
&$\ell=21$	&$(0\ 0\ 0\ 0\ 1)$		&$\ell=21$	&$(0\ 0\ 0\ 0\ 1)$	&$\ell=7$		&5 indices\\[-5pt]
&			&$(2\ 1\ 0\ 0\ 0)$		&			&$(2\ 1\ 0\ 0\ 0)$	&$\ell=8$		&5 indices\\[-1pt]
&$\ell=24$	&$(0\ 0\ 0\ 0\ 0\ 1)$	&$\ell=24$	&$(0\ 0\ 0\ 0\ 0\ 1)$		&$\ell=9$	&7 indices\\[-5pt]
&			&$(2\ 0\ 1\ 0\ 0\ 0)$	&			&$(2\ 0\ 1\ 0\ 0\ 0)$		&$\ell=10$	&8 indices\\[-5pt]
&			&$(1\ 0\ 0\ 1\ 0\ 0)$	&			&$(1\ 0\ 0\ 1\ 0\ 0)$		&$\ell=11$	&9 indices\\[-1pt]
&$\ell=27$	&$(0\ 2\ 0\ 0\ 0\ 0\ 0)$	&$\ell=26$	&$(0\ 2\ 0\ 0\ 0\ 0\ 0)$	&$\ell=12$	&19 indices\\[-5pt]
&			&$(0\ 0\ 0\ 0\ 0\ 0\ 1)$	&			&$(0\ 0\ 0\ 0\ 0\ 0\ 1)$	&$\ell=13$	&16 indices\\[-5pt]
&			&$(4\ 0\ 0\ 0\ 0\ 0\ 0)$	&			&$(4\ 0\ 0\ 0\ 0\ 0\ 0)$	&$\ell=14$	&16 indices\\[-5pt]
&			&					&			&					&$\ell=15$	&33 indices\\[-5pt]
&			&					&			&					&$\ell=16$	&38 indices\\[-5pt]
&			&					&			&					&$\ell=17$	&35 indices\\[-1pt]
\noalign{\hrule height 1.0pt}
\end{tabular}
\vspace{10pt}
\caption{
Experiments in section~\ref{sec:eigel-Lshaped}:
Final outputs and evolution of the index set for adaptive multilevel algorithms.
}
\label{tab:pb5}
\end{center}                                                                   
\end{table}
\end{scriptsize}

\begin{figure}[b]
\captionsetup[subfigure]{labelformat=empty}
\centering
\begin{subfigure}{0.30\textwidth}
\centering
\includegraphics*[width=\textwidth,trim = 105 25 100 25]{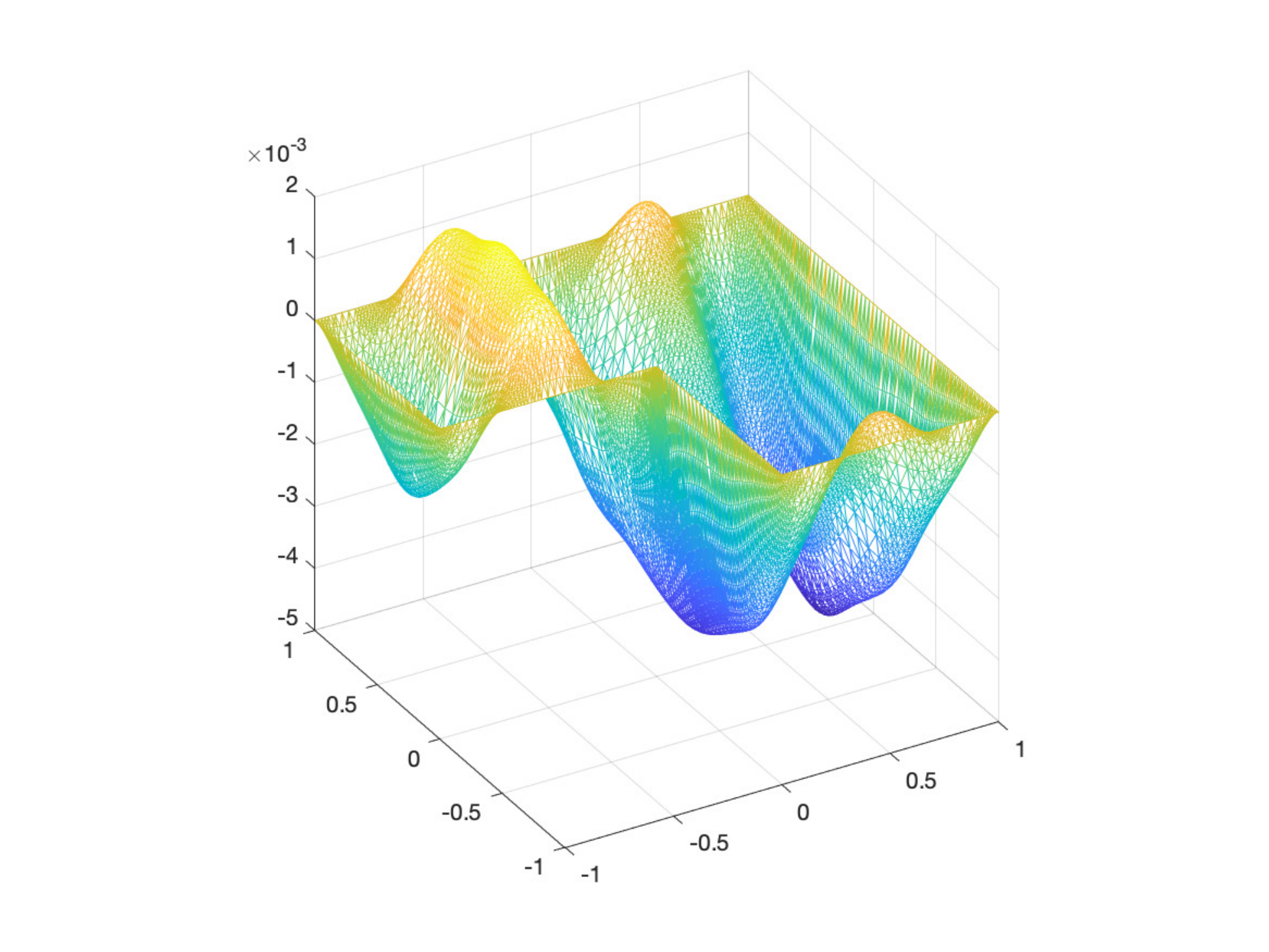}\\
\vspace{1mm}
\includegraphics*[width=\textwidth,trim = 118 45 98 30]{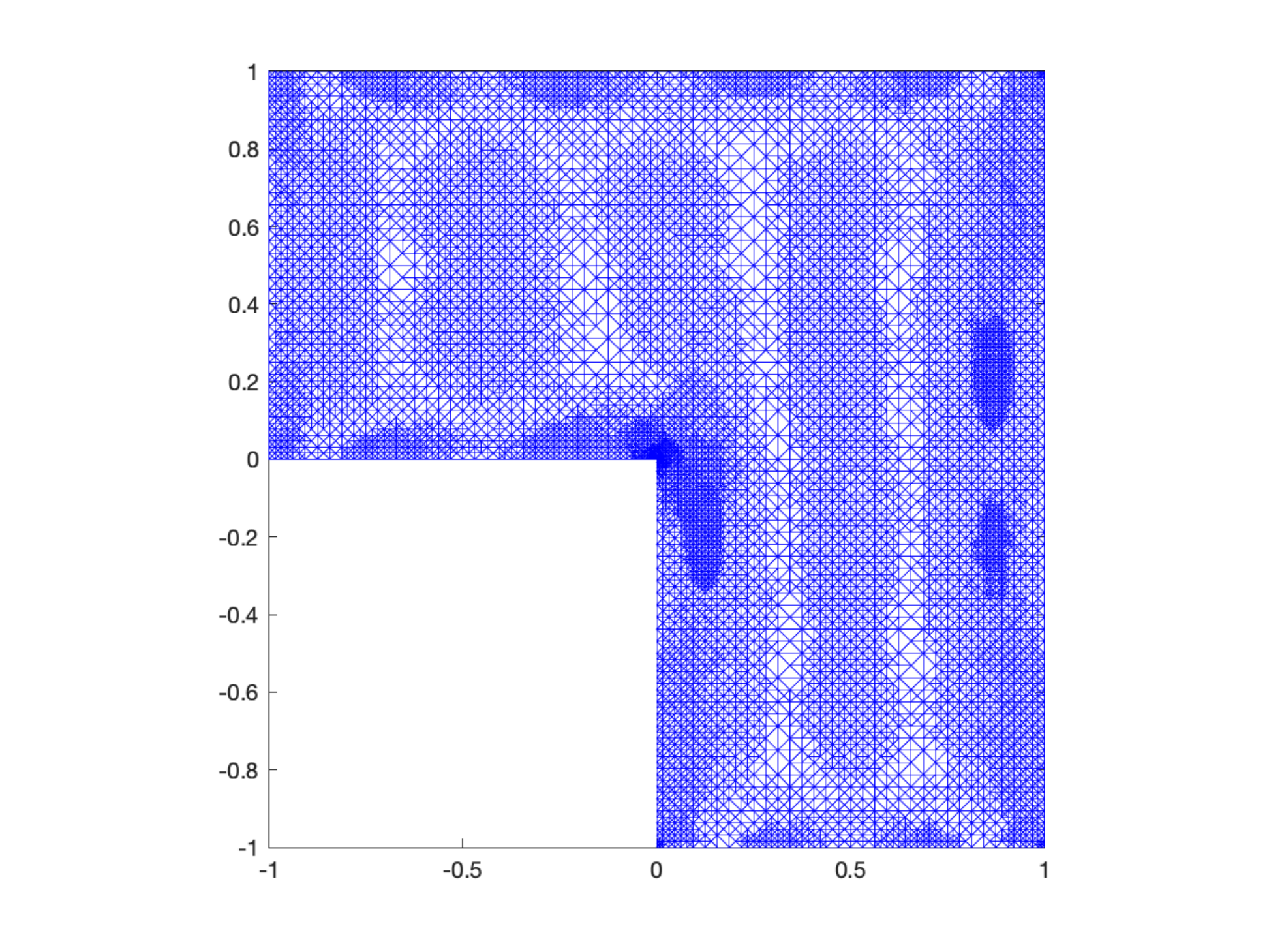}
\caption{$\nu = (0 \ 1)$\\ $\#\TT_{L\nu} = \num{24664}$}
\end{subfigure}
\hfill
\begin{subfigure}{0.30\textwidth}
\centering
\includegraphics*[width=\textwidth,trim = 105 25 100 25]{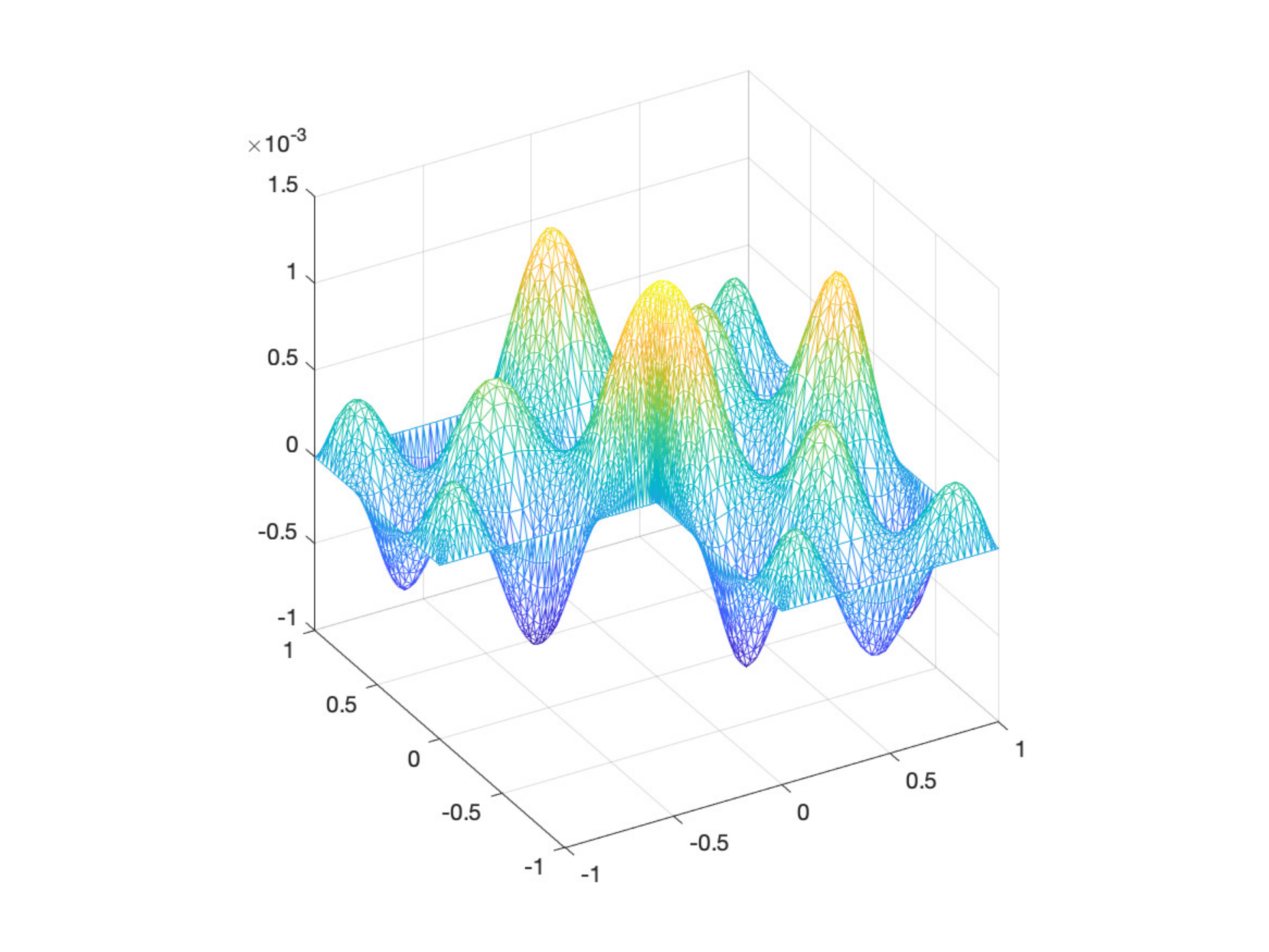}\\
\vspace{1mm}
\includegraphics*[width=\textwidth,trim = 118 45 98 30]{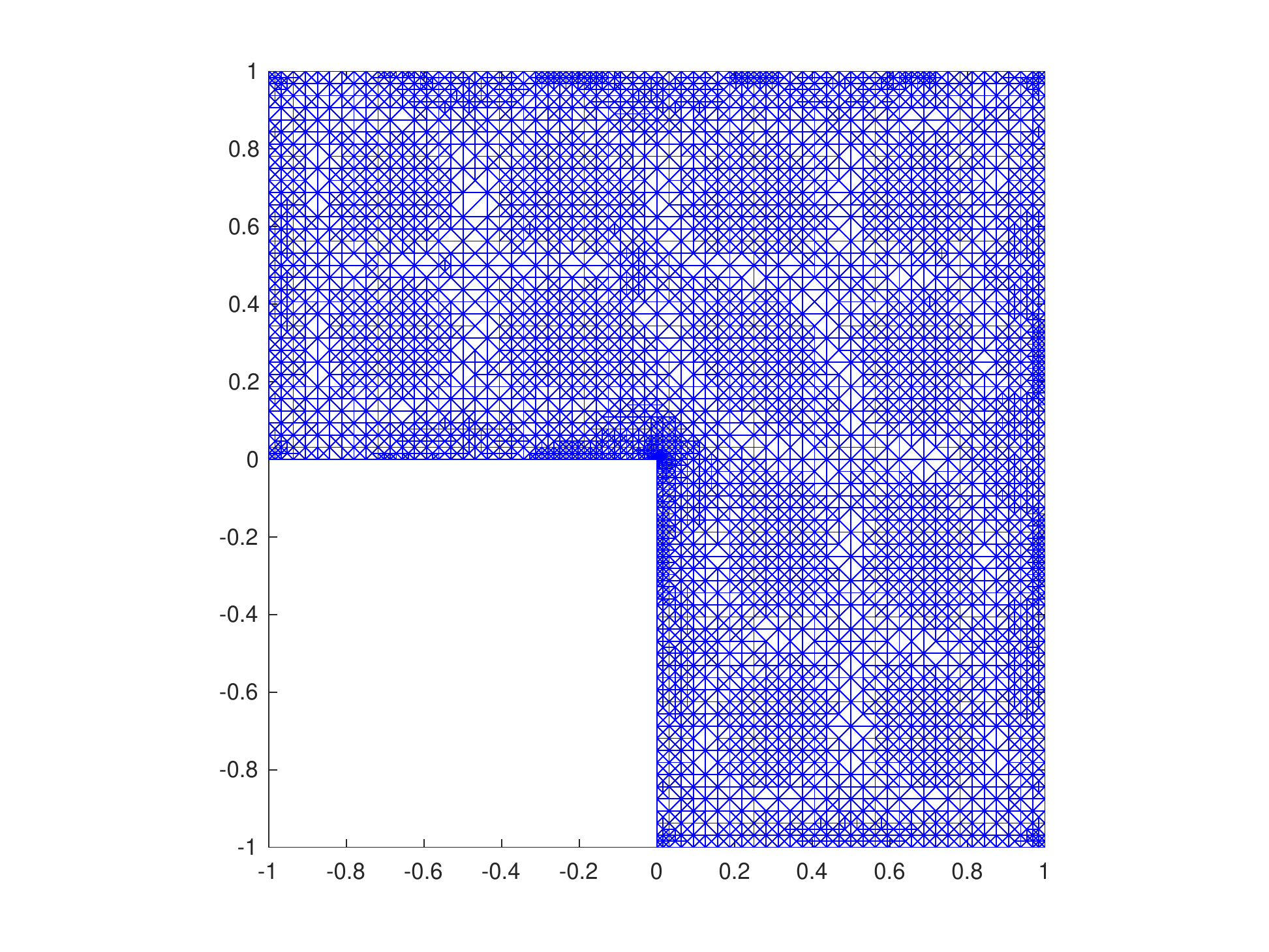}
\caption{$\nu = (1 \ 1)$\\ $\#\TT_{L\nu} = \num{11550}$}
\end{subfigure}
\hfill
\begin{subfigure}{0.30\textwidth}
\centering
\includegraphics*[width=\textwidth,trim = 105 25 100 25]{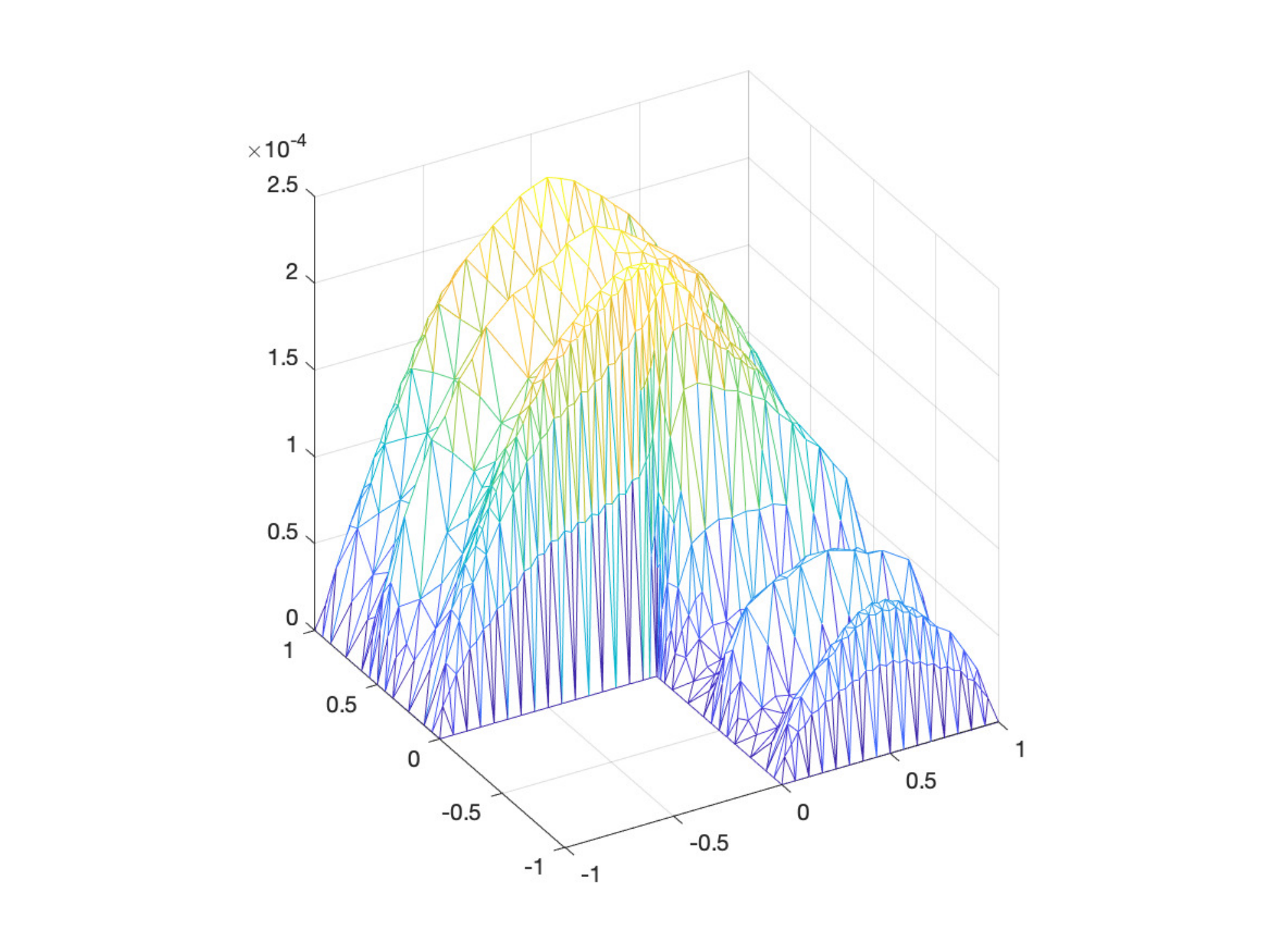}\\
\vspace{1mm}
\includegraphics*[width=\textwidth,trim = 118 45 98 30]{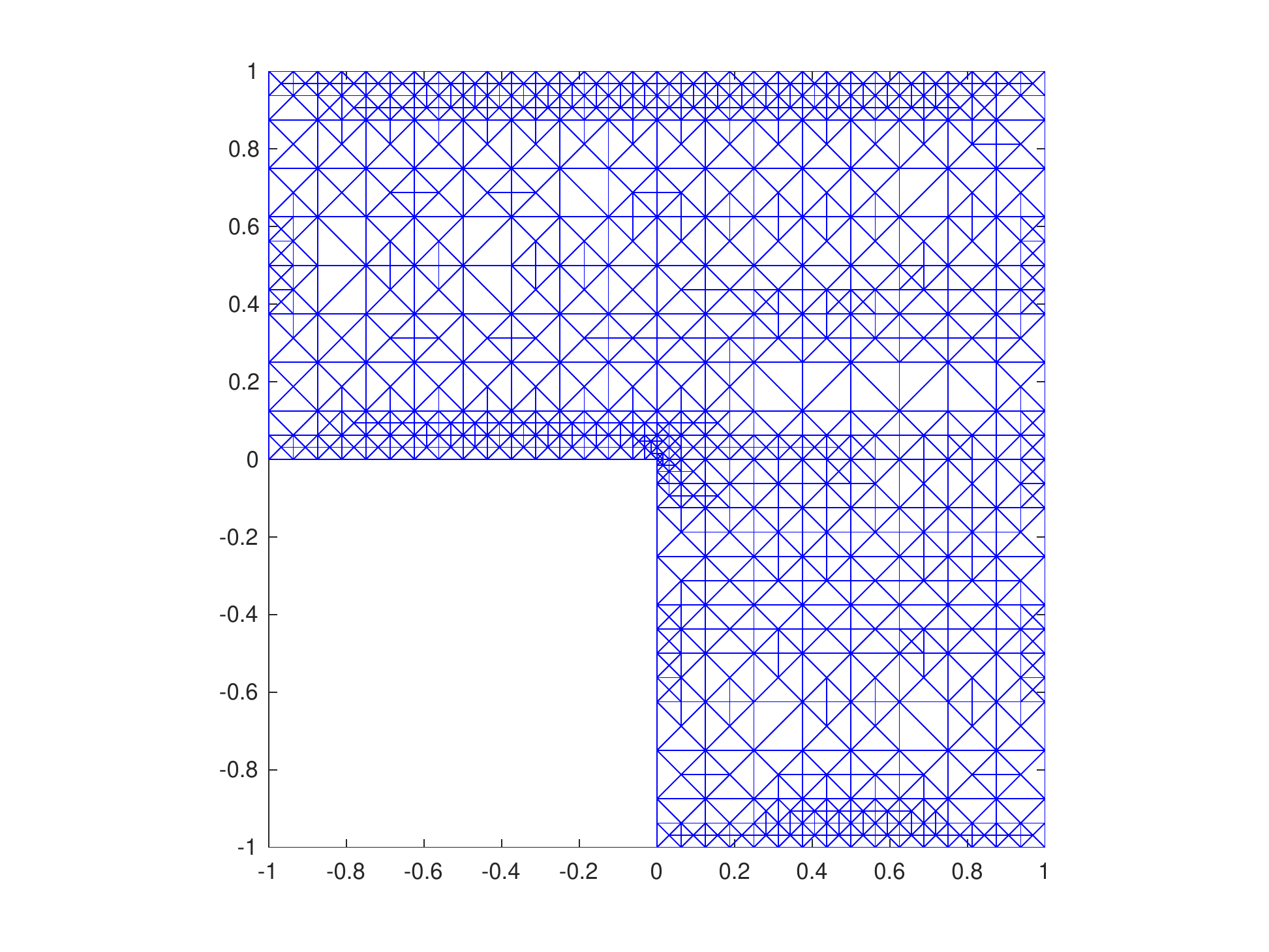}
\caption{$\nu = (4 \ 0)$\\ $\#\TT_{L\nu} = \num{1809}$}
\end{subfigure}
\caption{
Experiments in section~\ref{sec:eigel-Lshaped}:
Coefficients $u_{L\nu} \in \X_{L\nu} = \SS^1(\TT_{L\nu})$
of the final SGFEM approximation generated by \texttt{ML-C} (top plots)
and the associated adaptively refined meshes $\TT_{L\nu}$ (bottom plots) for three indices $\nu \in \PPP_L$.
}
\label{fig:pb5:pictures}
\end{figure}

Let us look in more detail at the performance of multilevel algorithms in this experiment.
In Figure~\ref{pb5:multilevel_decay}, for the algorithms \texttt{ML-A}, \texttt{ML-B} and \texttt{ML-C},
we plot the total error estimates $\est_\ell$ along with their spatial and parametric components given by
\begin{equation*}
\est_{\X_\ell}
:= \bigg( \sum_{\nu \in \PPP_\ell}\sum_{z \in \NN_{\ell\nu}^+} \est_{\ell}(\nu,z)^2\bigg)^{1/2}
\quad
\text{and}
\quad
\est_{\PPP_\ell}
:=
\bigg(\sum_{\nu \in \QQQ_\ell} \est_\ell(\nu)^2  \bigg)^{1/2},
\end{equation*}
respectively,
and the reference energy error $\bnorm{\uu_\mathrm{ref} - \uu_\ell}$,
where $\uu_\mathrm{ref}$ denotes a reference solution computed by running the algorithm \texttt{ML-C} to a lower tolerance.
Note that
$\est_\ell^2 = \est_{\X_\ell}^2 + \est_{\PPP_\ell}^2$; see~\eqref{eq:def:tau}.
For the algorithms with separate spatial and parametric enrichments (i.e., \texttt{ML-A} and \texttt{ML-B}),
the plots in Figure~\ref{pb5:multilevel_decay} look very similar.
For both these algorithms, 
we observe that the parametric error estimates $\est_{\PPP_\ell}$ remain essentially constant during mesh refinement iterations,
whereas the spatial error estimates $\est_{\X_\ell}$ exhibit a noticeable increase at the iteration following each parametric enrichment.
The latter observation is a consequence of assigning the coarse mesh $\TT_0$ to every new index introduced by the parametric enrichment.
As a result, the decay rates of the total error estimates $\est_\ell$ for \texttt{ML-A}, \texttt{ML-B} are still suboptimal.

By looking at the plot for the algorithm with combined enrichment (i.e., \texttt{ML-C})
we see a completely different behavior.
The balanced enrichment of spatial and parametric components of Galerkin approximations that was inherent
to \texttt{ML-A} and \texttt{ML-B} is completely lost in~\texttt{ML-C}.
Instead,~\texttt{ML-C} clearly privileges parametric enrichment by activating significantly more indices than \texttt{ML-A} and \texttt{ML-B}
(see also Table~\ref{tab:pb5}).
This is a consequence of the combined marking strategy~\eqref{eq:doerfler:combined} and
the fact that a small number of parametric error indicators are larger in magnitude than a significant proportion of spatial error indicators.
This results in the parametric error estimates $\est_{\PPP_\ell}$ decaying much faster than their spatial counterparts $\est_{\X_\ell}$.
However, the total error estimate $\est_\ell$ decays with fully optimal rate $\mathcal{O}(N_\ell^{-0.5})$.

In Table~\ref{tab:pb5}, for each multilevel adaptive algorithm,
we show the total number of iterations $L$,
the final value of the total error estimate $\est_L$,
the number of degrees of freedom in the final SGFEM approximation, as well as
the cardinality of the final index set $\PPP_L$,
the (total) degree $\deg\PPP_L := \max_{\nu \in \PPP_L} \sum_{j \geq 1} \nu_j$ of polynomials in the associated polynomial space,
and the number of active parameters $M_{\PPP_L}$ in $\PPP_L$.
We also show the evolution of the index set throughout each computation.
By looking at these results, we observe that in order to reach the prescribed tolerance,
the algorithm with combined enrichment requires significantly
less iterations and generates the final Galerkin approximation with significantly less degrees of freedom
than either of the algorithms with separate enrichments.
In addition to this, these two types of multilevel algorithms generate Galerkin approximations with remarkably different distributions
of spatial and parametric degrees of freedom.
Specifically, while \texttt{ML-A} and \texttt{ML-B} produce relatively small index sets and fine meshes for most of the indices,
\texttt{ML-C} generates a much larger index set but very coarse meshes for the majority of indices.
The latter feature resembles that of multilevel sampling methods, where very few deterministic PDE solves
are performed on fine spatial meshes while the majority of solves use coarse meshes.

In Figure~\ref{fig:pb5:pictures}, for the final SGFEM approximation
\begin{equation*}
\uu_L = \sum_{\nu \in \PPP_L} u_{L\nu} P_\nu \in \V_L
\end{equation*}
generated by \texttt{ML-C},
we plot the coefficients $u_{L\nu} \in \X_{L\nu}$ and the associated meshes $\TT_{L\nu}$
for $\nu \in \big\{ (0\ 1) , (1\ 1) , (4\ 0) \big\} \subset \PPP_L$.
Meshes with {similar patterns} were produced by all other multilevel algorithms.
We observe that adaptively refined meshes
identify the geometric singularity at the reentrant corner (affecting all coefficients in the same way)
and the regions with steep gradient (which are different for each coefficient).
All the identified areas exhibit much stronger mesh refinement than elsewhere in the domain.
More importantly, finer meshes are produced for those coefficients that are more `influential' in the Galerkin solution
(i.e., the coefficients whose indices are activated earlier); cf. the values of $\#\TT_{L\nu}$ in Figure~\ref{fig:pb5:pictures}.
This illustrates how the flexi\-bi\-li\-ty in allocating degrees of freedom ensures greater efficiency of multilevel methods,
compared to the single-level SGFEM.

\subsection{Cookie problem} \label{sec:cookie}

Our second example of parametric problem~\eqref{eq:strongform}--\eqref{eq1:a} is the so-called \emph{cookie problem};
cf.~\cite{BallaniG_15_HTA, ENSW19}.
We consider the square domain $D = (0,1)^2$
that contains nine circular inclusions $D_m \subset D$ ($m = 1,\ldots,9$).
For all $i,j \in \{1,2,3\}$, the subdomain $D_{i+3(j-1)}$ is the disk with center at the point $((2i-1)/6,(2j-1)/6)$ and radius $r = 1/8$.
We set $\ff \equiv 1$ in~\eqref{eq:strongform} and select the expansion coefficients in~\eqref{eq1:a}
as follows:
\begin{equation} \label{eq:coeff_cookie}
a_m(x)
=
\begin{cases}
1 & \text{for } m=0,\\
0.5\,\raisebox{1pt}{$\chi$}_{D_m}(x) & \text{for } m=1,3,7,9,\\
0.7\,\raisebox{1pt}{$\chi$}_{D_m}(x) & \text{for } m=2,4,6,8,\\
0.9\,\raisebox{1pt}{$\chi$}_{D_m}(x) & \text{for } m=5,\\
0 & \text{for } m>9
\end{cases}
\quad \text{for all } x \in D,
\end{equation}
where $\raisebox{1pt}{$\chi$}_{D_m}$ denotes the characteristic function of the subdomain~$D_m$.
Thus, the diffusion coefficient $\aa(x,\y)$ in this example depends on finitely many parameters $y_1,\ldots,y_9 \in [-1,1]$;
furthermore, assumptions~\eqref{eq2:a}--\eqref{eq3:a} are satisfied
(with $a_0^{\rm min} = a_0^{\rm max} = 1$ and $\tau = 0.9$).

We emphasize that, in contrast to the benchmark problem in section~\ref{sec:eigel},
where the amplitude of the coefficient $a_m$
in the expansion~\eqref{eq1:a} decays as $m$ increases,
which induces a hierarchy of the parameters
(with $y_m$ being more `important' than $y_\ell$ if $m<\ell$),
in this example 
the `importance' of the parameters cannot be directly inferred from
the ordering of the terms in expansion~\eqref{eq1:a}.
Hence,
one should not \textsl{a~priori} prescribe any specific order in which the parameters are activated.
That is why, when running adaptive algorithms for the cookie problem, we set $\overline{M}=9$ in~\eqref{def:Q}
(note that in this example $\III := \N_0^9$).
This way, when it comes to the first parametric enrichment, all parameters are available for activation,
and the order in which they are activated is determined by the associated parametric indicators.

In computations with all five adaptive algorithms for this problem,
we set the stopping tolerance {\tt tol} = \num{8.0e-4} and
use the same initial mesh $\TT_0$ as in~\S\ref{sec:eigel-square}.

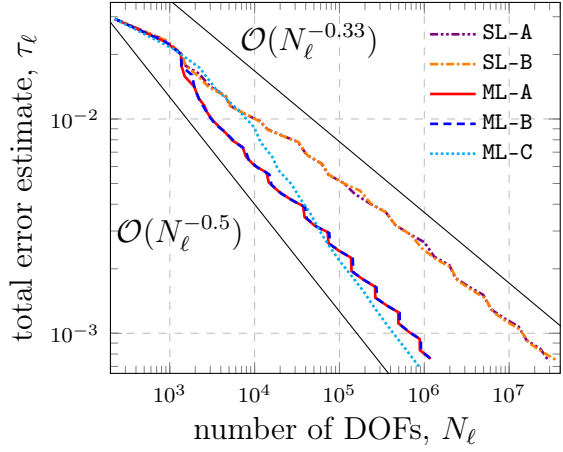
\begin{figure}[ht]
\begin{tikzpicture}
\pgfplotstableread{data/pb9-sl-a.dat}{\one}
\pgfplotstableread{data/pb9-sl-b.dat}{\two}
\pgfplotstableread{data/pb9-ml-a.dat}{\three}
\pgfplotstableread{data/pb9-ml-b.dat}{\four}
\pgfplotstableread{data/pb9-ml-c.dat}{\five}
\begin{loglogaxis}
[
width = 7.5cm, height=6.5cm,						
xlabel={number of DOFs, $N_\ell$}, 					
ylabel={total error estimate, $\est_\ell$},				
ymajorgrids=true, xmajorgrids=true, grid style=dashed,	
xmin = (2.0)*10^(2),
xmax = (4.0)*10^(7),
ymin = (6.5)*10^(-4),
ymax = (3.5)*10^(-2),
legend style={legend pos=north east, legend cell align=left, fill=none, draw=none}
]
\addplot[violet,line width=1.0pt, densely dash dot dot]	table[x=dofs, y=error]{\one};
\addplot[orange,line width=1.0pt, densely dash dot]		table[x=dofs, y=error]{\two};
\addplot[red,line width=1.0pt]						table[x=dofs, y=error]{\three};
\addplot[blue,line width=1.0pt, densely dashed]			table[x=dofs, y=error]{\four};
\addplot[cyan,line width=1.0pt, densely dotted]			table[x=dofs, y=error]{\five};
\addplot[black,solid,domain=10^(2.2):10^(7.8)] { 0.35*x^(-0.33) };
\node at (axis cs:5e3,1.7e-2) [anchor=south west] {$\mathcal{O}(N_\ell^{-0.33})$};
\addplot[black,solid,domain=10^(2.2):10^(6.0)] { 0.4*x^(-0.5) };
\node at (axis cs:1e4,4.0e-3) [anchor=north east] {$\mathcal{O}(N_\ell^{-0.5})$};
\legend{
\texttt{SL-A},
\texttt{SL-B},
\texttt{ML-A},
\texttt{ML-B},
\texttt{ML-C},
}
\end{loglogaxis}
\end{tikzpicture}
\caption{
Experiments in section~\ref{sec:cookie}:
Total error estimates $\est_\ell$ versus the number of degrees of freedom $N_\ell$
for all adaptive algorithms.
}
\label{fig:pb9-rates}
\end{figure}

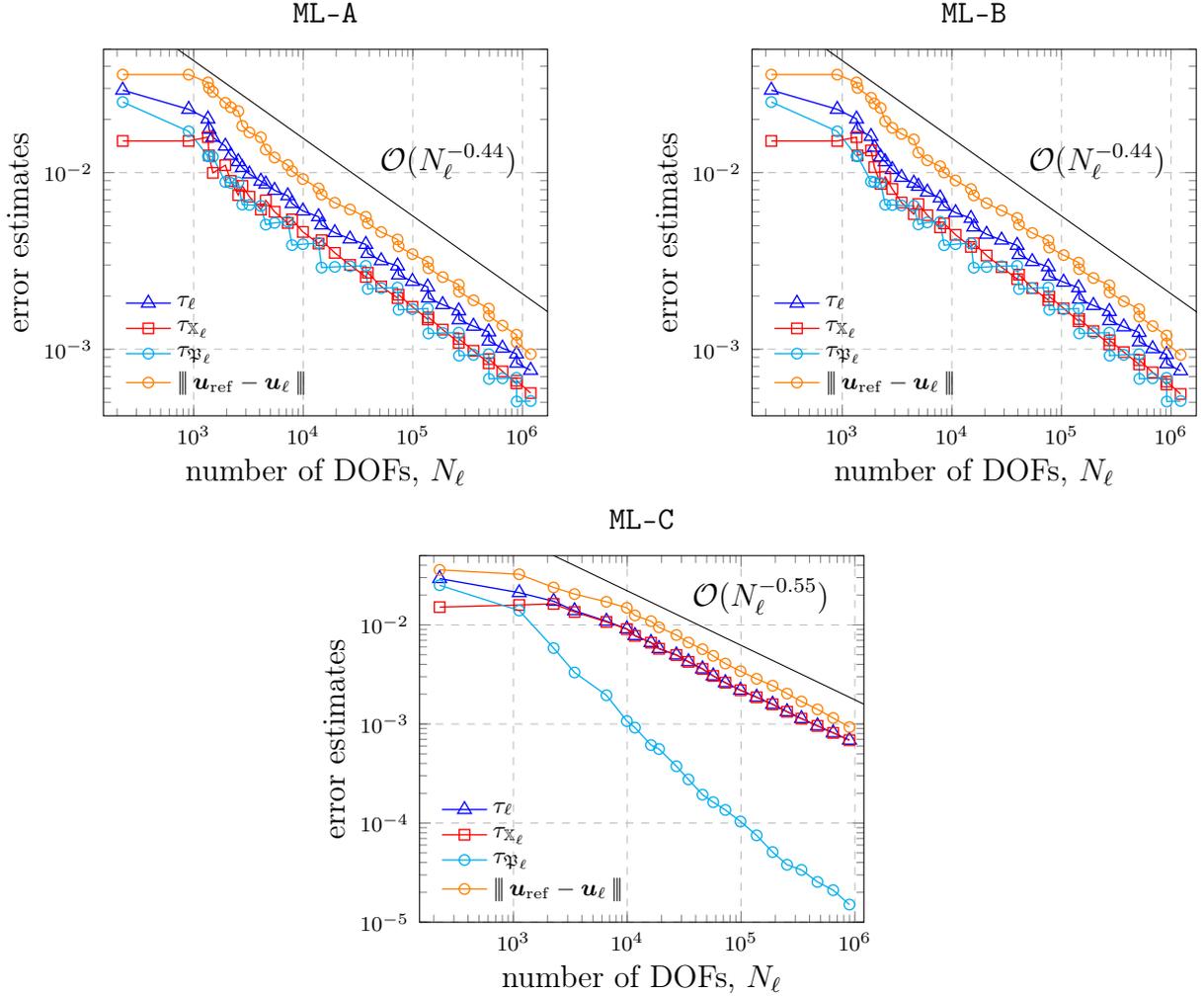
\begin{figure}[H]
\begin{tikzpicture}
\pgfplotstableread{data/pb9-ml-a.dat}{\one}
\begin{loglogaxis}
[
width = 7.5cm, height=6.5cm,								
title={\texttt{ML-A}},
xlabel={number of DOFs, $N_\ell$},						
ylabel={error estimates},									
ymajorgrids=true, xmajorgrids=true, grid style=dashed,		
xmin = (1.5)*10^(2),
xmax = (1.7)*10^(6),
ymin = (4.2)*10^(-4),
ymax = (5.0)*10^(-2),
legend style={legend pos=south west, legend cell align=left, fill=none, draw=none}
]
\addplot[blue,mark=triangle,mark size=3.0pt,line width=0.5]	table[x=dofs, y=error]{\one};
\addplot[red,mark=square,mark size=2.0pt,line width=0.5]		table[x=dofs, y=yp_one]{\one};
\addplot[cyan,mark=o,mark size=2.0pt,line width=0.5]		table[x=dofs, y=xq_one]{\one};
\addplot[orange,mark=o,mark size=2.0pt,line width=0.5]		table[x=dofs, y=truerr]{\one};
\addplot[black,solid,domain=10^(1.5):10^(7.8)] { 0.9*x^(-0.44) };
\node at (axis cs:4e4,8e-3) [anchor=south west] {$\mathcal{O}(N_\ell^{-0.44})$};
\legend{
$\est_\ell$,
{$\est_{\X_\ell}$},
{$\est_{\PPP_\ell}$},
{$\bnorm{\uu_\mathrm{ref}-\uu_\ell}$},
}
\end{loglogaxis}
\end{tikzpicture}
\hfill
\begin{tikzpicture}
\pgfplotstableread{data/pb9-ml-b.dat}{\one}
\begin{loglogaxis}
[
width = 7.5cm, height=6.5cm,								
title={\texttt{ML-B}},
xlabel={number of DOFs, $N_\ell$},						
ylabel={error estimates},									
ymajorgrids=true, xmajorgrids=true, grid style=dashed,		
xmin = (1.5)*10^(2),
xmax = (1.7)*10^(6),
ymin = (4.2)*10^(-4),
ymax = (5.0)*10^(-2),
legend style={legend pos=south west, legend cell align=left, fill=none, draw=none}
]
\addplot[blue,mark=triangle,mark size=3.0pt,line width=0.5]	table[x=dofs, y=error]{\one};
\addplot[red,mark=square,mark size=2.0pt,line width=0.5]		table[x=dofs, y=yp_one]{\one};
\addplot[cyan,mark=o,mark size=2.0pt,line width=0.5]		table[x=dofs, y=xq_one]{\one};
\addplot[orange,mark=o,mark size=2.0pt,line width=0.5]		table[x=dofs, y=truerr]{\one};
\addplot[black,solid,domain=10^(1.5):10^(7.8)] { 0.9*x^(-0.44) };
\node at (axis cs:4e4,8e-3) [anchor=south west] {$\mathcal{O}(N_\ell^{-0.44})$};
\legend{
$\est_\ell$,
{$\est_{\X_\ell}$},
{$\est_{\PPP_\ell}$},
{$\bnorm{\uu_\mathrm{ref}-\uu_\ell}$},
}
\end{loglogaxis}
\end{tikzpicture}
\hfill
\begin{tikzpicture}
\pgfplotstableread{data/pb9-ml-c.dat}{\one}
\begin{loglogaxis}
[
width = 7.5cm, height=6.5cm,								
title={\texttt{ML-C}},
xlabel={number of DOFs, $N_\ell$},						
ylabel={error estimates},									
ymajorgrids=true, xmajorgrids=true, grid style=dashed,		
xmin = (1.5)*10^(2),
xmax = (1.2)*10^(6),
ymin = (1.0)*10^(-5),
ymax = (5.0)*10^(-2),
legend style={legend pos=south west, legend cell align=left, fill=none, draw=none}
]
\addplot[blue,mark=triangle,mark size=3.0pt,line width=0.5]	table[x=dofs, y=error]{\one};
\addplot[red,mark=square,mark size=2.0pt,line width=0.5]		table[x=dofs, y=yp_one]{\one};
\addplot[cyan,mark=o,mark size=2.0pt,line width=0.5]		table[x=dofs, y=xq_one]{\one};
\addplot[orange,mark=o,mark size=2.0pt,line width=0.5]		table[x=dofs, y=truerr]{\one};

\addplot[black,solid,domain=10^(1.5):10^7] {3.5*x^(-0.55) };
\node at (axis cs:3e4,1.0e-2) [anchor=south west] {$\mathcal{O}(N_\ell^{-0.55})$};
\legend{
$\est_\ell$,
{$\est_{\X_\ell}$},
{$\est_{\PPP_\ell}$},
{$\bnorm{\uu_\mathrm{ref}-\uu_\ell}$},
}
\end{loglogaxis}
\end{tikzpicture}
\caption{
Experiments in section~\ref{sec:cookie}:
Decay of the error estimates (total, spatial, parametric)
and the reference errors computed at each iteration of the adaptive multilevel algorithm.
}
\label{fig:pb9:multilevel_decay}
\end{figure}

In Figure~\ref{fig:pb9-rates}, for all adaptive algorithms,
we plot the error estimates $\est_\ell$ against the number of degrees of freedom $N_\ell$.
The results are in agreement with those presented in section~\ref{sec:eigel}:
(i) For single-level approximations, the error estimates decay with suboptimal rate $\mathcal{O}(N_\ell^{-0.33})$;
(ii) The decay rates for the multilevel approximations generated by \texttt{ML-A} and \texttt{ML-B}
are faster than $\mathcal{O}(N_\ell^{-0.33})$ but not optimal;
(iii) For the multilevel approximations generated by \texttt{ML-C},
the error estimates decay with fully optimal rate~$\mathcal{O}(N_\ell^{-0.5})$.

In Figure~\ref{fig:pb9:multilevel_decay}, for algorithms \texttt{ML-A}, \texttt{ML-B} and \texttt{ML-C},
we plot the total error estimates $\est_\ell$ along with their spatial and parametric components
$\est_{\X_\ell}$ and $\est_{\PPP_\ell}$,
as well as the reference energy error $\bnorm{\uu_\mathrm{ref} - \uu_\ell}$,
where $\uu_\mathrm{ref}$ denotes a reference solution computed by running the algorithm \texttt{ML-C} to a lower tolerance
({\tt tol} = \num{2.0e-4}).

\begin{scriptsize}
\begin{table}[ht]
\setlength\tabcolsep{4pt} 
\begin{center} 
\renewcommand{\arraystretch}{1.45}
\begin{tabular}{r !{\vrule width 1.0pt} c  l !{\vrule width 1.0pt} c l !{\vrule width 1.0pt} c l} 
\noalign{\hrule height 1.0pt}
&\multicolumn{2}{c!{\vrule width 1.0pt}}{\texttt{ML-A}} 
&\multicolumn{2}{c!{\vrule width 1.0pt}}{\texttt{ML-B}}
&\multicolumn{2}{c}{\texttt{ML-C}} \\	
\noalign{\hrule height 1.0pt}
$L$					&\multicolumn{2}{c!{\vrule width 1.0pt}}{37}		
					&\multicolumn{2}{c!{\vrule width 1.0pt}}{37}
					&\multicolumn{2}{c}{21}\\[-5pt]
$\est_L$				&\multicolumn{2}{c!{\vrule width 1.0pt}}{$7.60064 \cdot 10^{-4}$}	
					&\multicolumn{2}{c!{\vrule width 1.0pt}}{$7.55016 \cdot 10^{-4}$}
					&\multicolumn{2}{c}{$6.86986 \cdot 10^{-4}$}\\[-5pt]
$N_L$				&\multicolumn{2}{c!{\vrule width 1.0pt}}{\num{1188953}}
					&\multicolumn{2}{c!{\vrule width 1.0pt}}{\num{1223401}}
					&\multicolumn{2}{c}{\num{897023}}\\[-5pt]
$\#\PPP_L$ 			&\multicolumn{2}{c!{\vrule width 1.0pt}}{73}
					&\multicolumn{2}{c!{\vrule width 1.0pt}}{73}
					&\multicolumn{2}{c}{629}\\[-5pt]
$\deg\PPP_L$	 		&\multicolumn{2}{c!{\vrule width 1.0pt}}{8}
					&\multicolumn{2}{c!{\vrule width 1.0pt}}{8}
					&\multicolumn{2}{c}{17}\\[-5pt]
$M_{\PPP_L}$			&\multicolumn{2}{c!{\vrule width 1.0pt}}{9}
					&\multicolumn{2}{c!{\vrule width 1.0pt}}{9}
					&\multicolumn{2}{c}{9}\\
\noalign{\hrule height 1.0pt}
\end{tabular}
\vspace{10pt}
\caption{
Experiments in section~\ref{sec:cookie}:
Final outputs for adaptive multilevel algorithms.
}
\label{tab:pb9}
\end{center}                                                                   
\end{table}
\end{scriptsize}

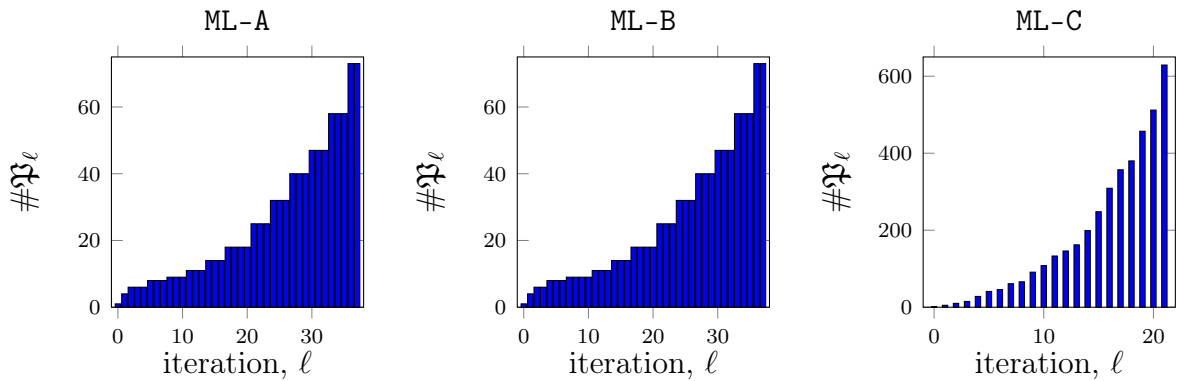
\begin{figure}[ht]
\begin{tikzpicture}
\pgfplotstableread{data/pb9-ml-a.dat}{\one}
\begin{axis}[
width = 4.9cm, height=4.9cm,
ybar,
bar width=2pt,
xmin = -1,
xmax = 38,
ymin=0,
ymax=75,
title=\texttt{ML-A},
xlabel = {iteration, $\ell$},
ylabel = {$\#\PPP_\ell$},
]
\addplot[fill=blue] table[x=iter,y=cardP]{\one};
\end{axis}
\end{tikzpicture}
\quad
\begin{tikzpicture}
\pgfplotstableread{data/pb9-ml-b.dat}{\one}
\begin{axis}[
width = 4.9cm, height=4.9cm,
ybar,
bar width=2pt,
xmin = -1,
xmax = 38,
ymin=0,
ymax=75,
title=\texttt{ML-B},
xlabel = {iteration, $\ell$},
ylabel = {$\#\PPP_\ell$},
]
\addplot[fill=blue] table[x=iter,y=cardP]{\one};
\end{axis}
\end{tikzpicture}
\quad
\begin{tikzpicture}
\pgfplotstableread{data/pb9-ml-c.dat}{\one}
\begin{axis}[
width = 4.9cm, height=4.9cm,
ybar,
bar width=2pt,
xmin = -1,
xmax = 22,
ymin=0,
ymax=650,
title=\texttt{ML-C},
xlabel = {iteration, $\ell$},
ylabel = {$\#\PPP_\ell$},
]
\addplot[fill=blue]	table[x=iter,y=cardP]{\one};
\end{axis}
\end{tikzpicture}
\caption{
Experiments in section~\ref{sec:cookie}:
Evolution of the cardinality of the index set $\PPP_\ell$.
}
\label{fig:pb9:index_set}
\end{figure}

In Table~\ref{tab:pb9}, we show the outputs for the multilevel algorithms.
Each algorithm activates all nine relevant parameters $y_1, \dots, y_9$.
While we do not observe significant differences between \texttt{ML-A} and \texttt{ML-B},
we see that \texttt{ML-C} reaches the prescribed tolerance with
less iterations, a smaller number of degrees of freedom, a richer index set, and a higher polynomial degree
than the two other algorithms
(see also Figure~\ref{fig:pb9:index_set}, where we show the evolution of $\# \PPP_\ell$).
This is again in agreement with the results presented in section~\ref{sec:eigel}.

\begin{figure}[ht]
\centering
\begin{subfigure}{0.19\textwidth}
\centering
\includegraphics*[width=\textwidth,trim = 105 25 100 25]{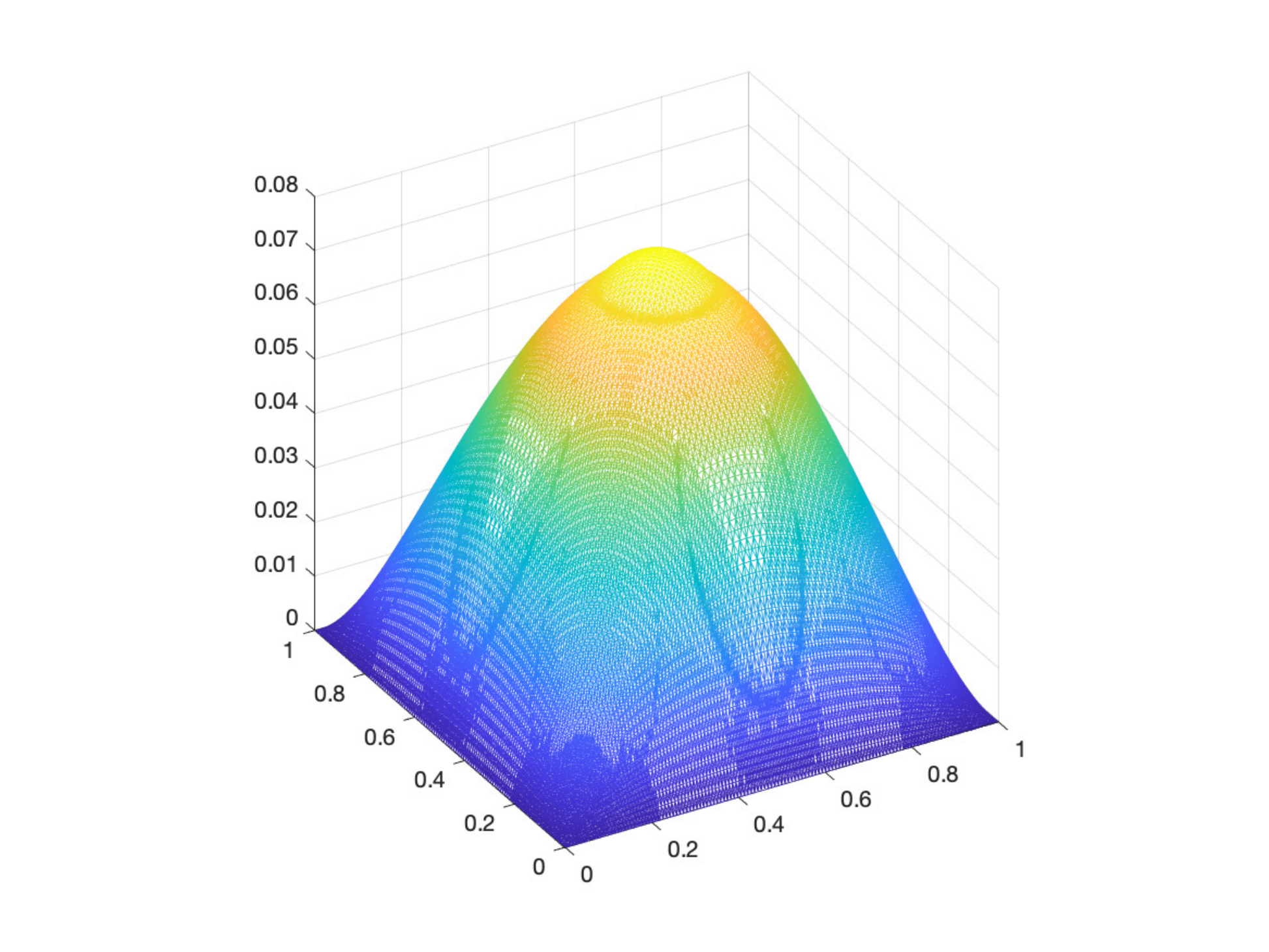}\\
\vspace{1mm}
\includegraphics*[width=\textwidth,trim = 118 45 98 30]{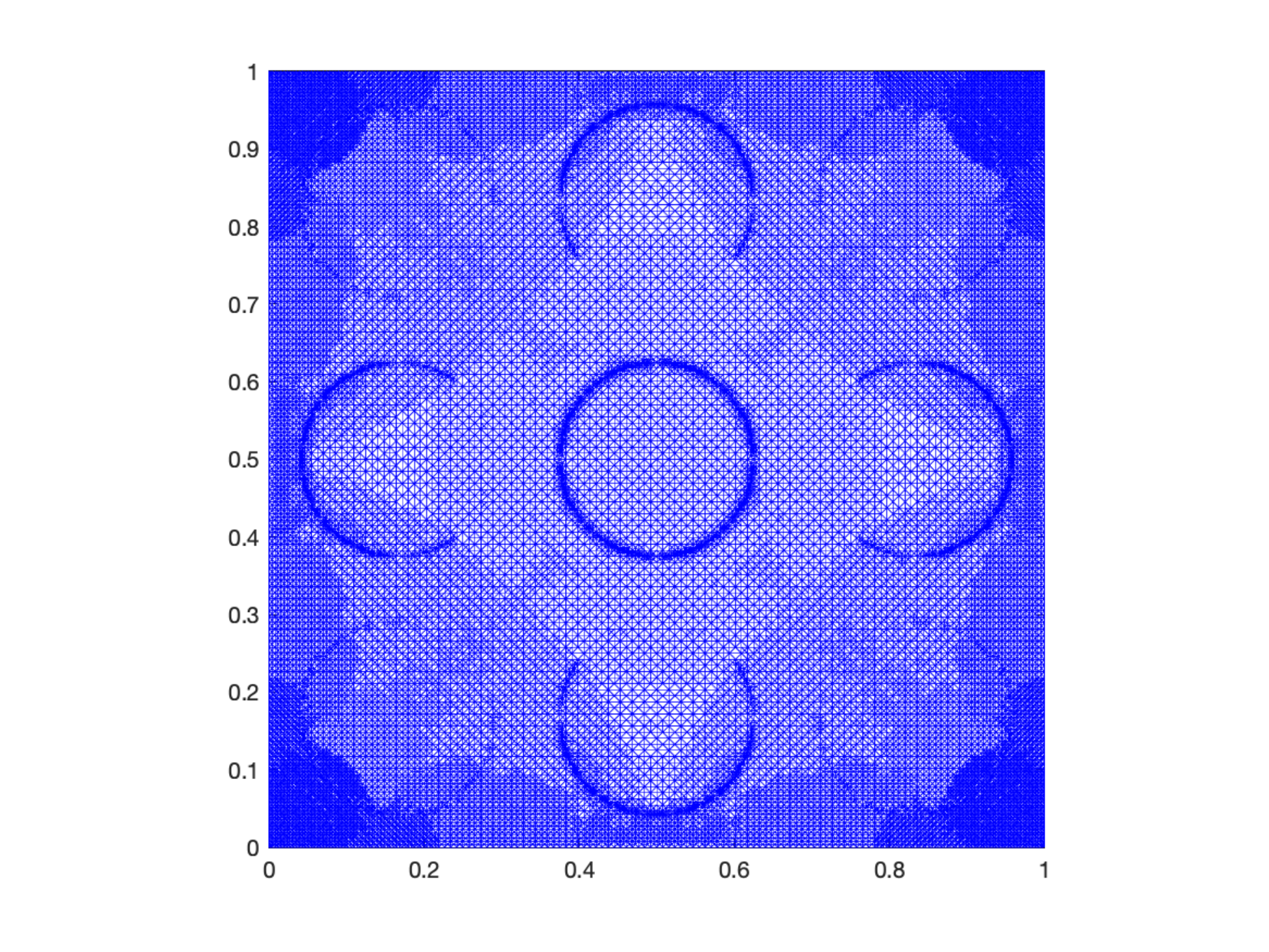}
\caption{$\nu = \0$\\ $\#\TT_{\ell\nu} = \num{84050}$}
\end{subfigure}
\hfill
\begin{subfigure}{0.19\textwidth}
\centering
\includegraphics*[width=\textwidth,trim = 105 25 100 25]{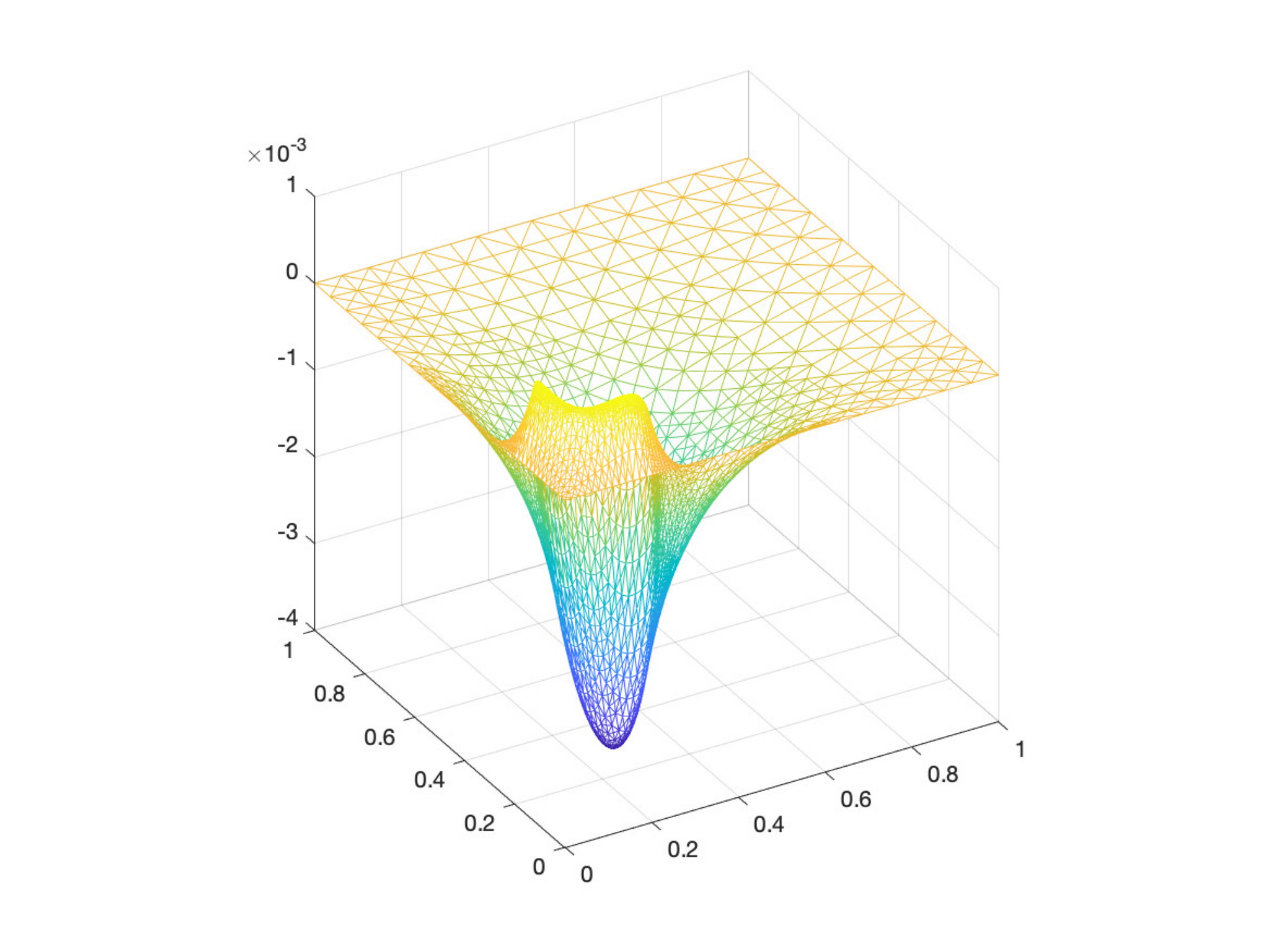}\\
\vspace{1mm}
\includegraphics*[width=\textwidth,trim = 118 45 98 30]{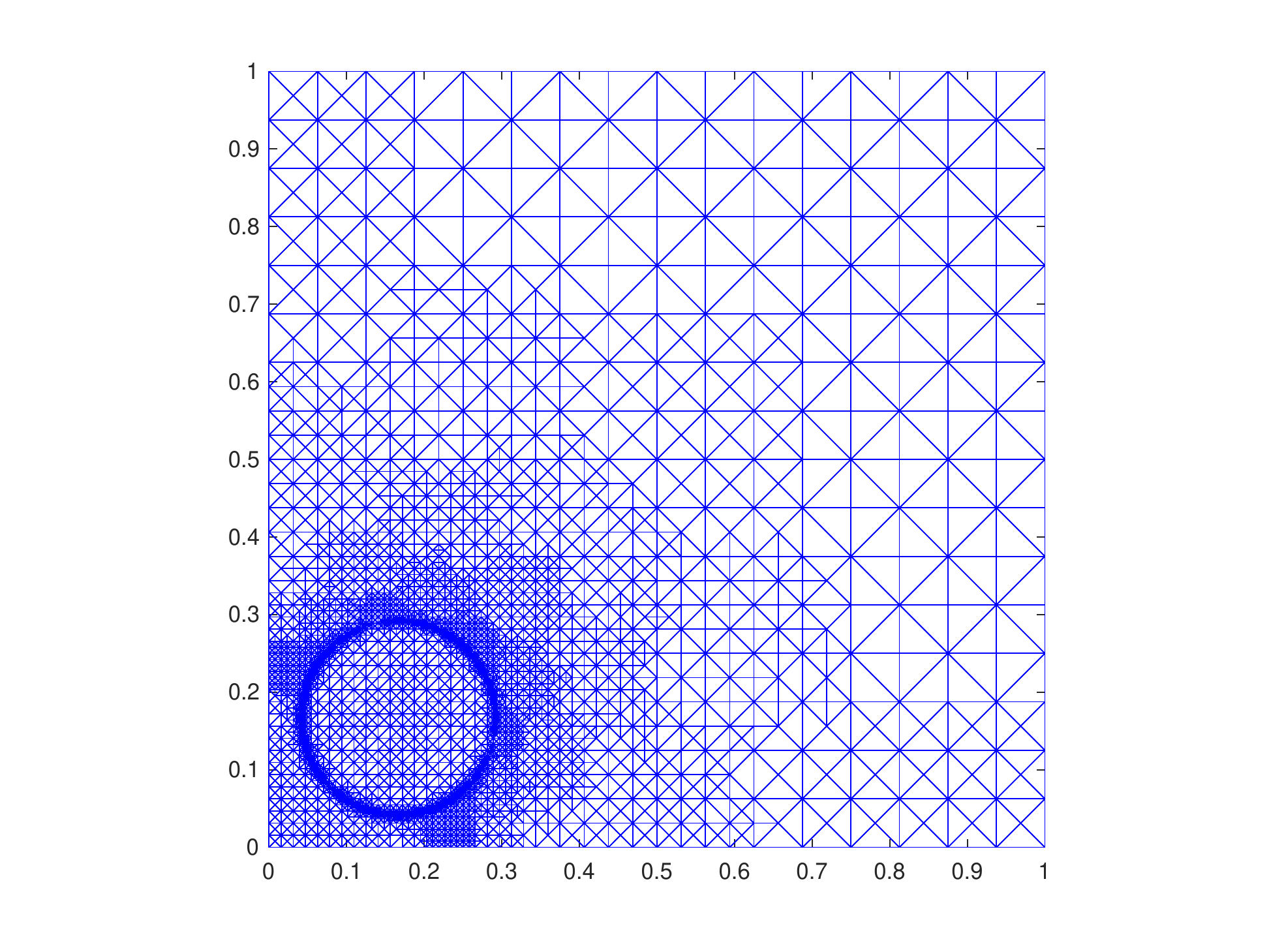}
\caption{$\nu = \eps_1$\\ $\#\TT_{\ell\nu} = \num{10994}$}
\end{subfigure}
\hfill
\begin{subfigure}{0.19\textwidth}
\centering
\includegraphics*[width=\textwidth,trim = 105 25 100 25]{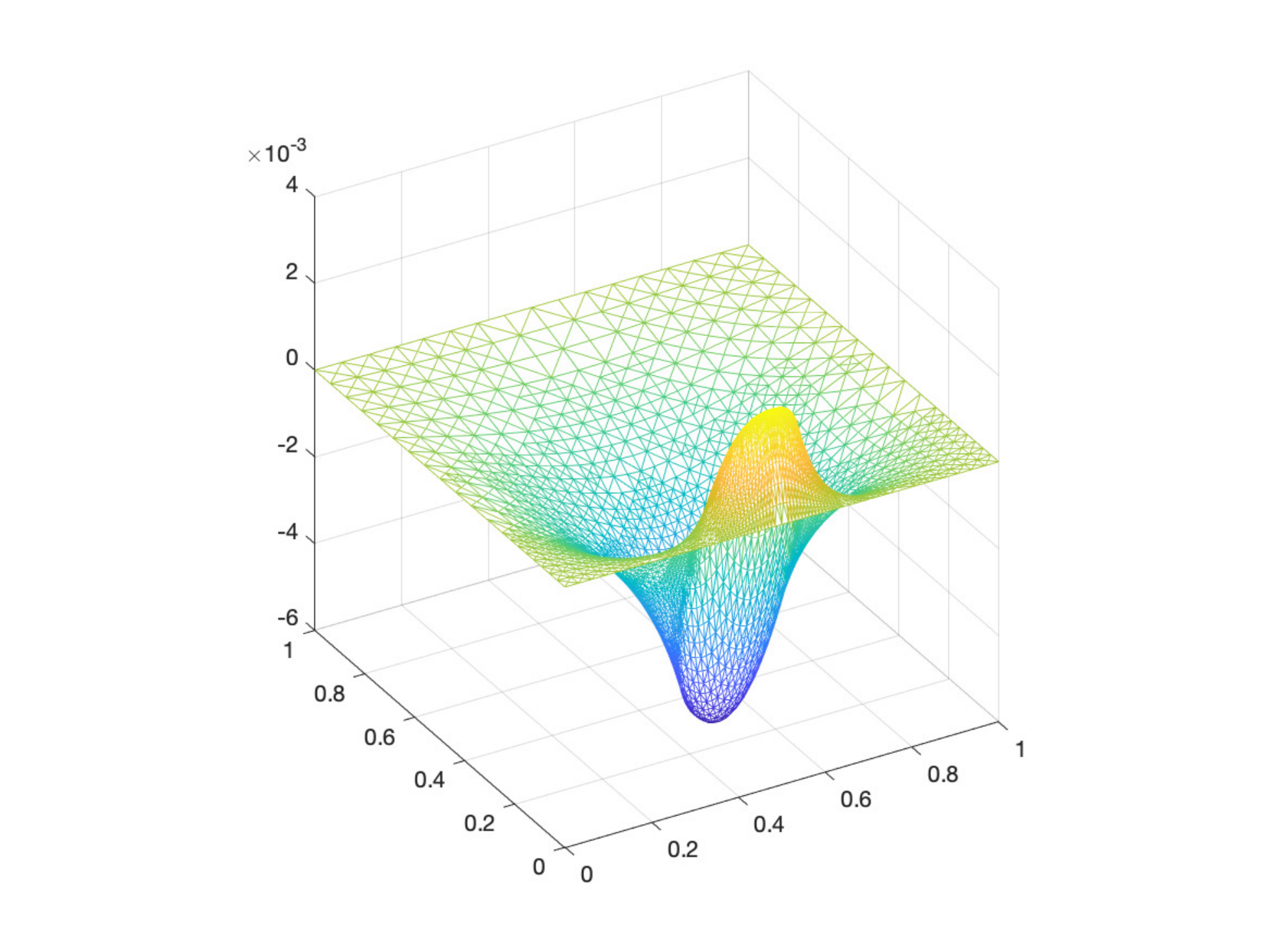}\\
\vspace{1mm}
\includegraphics*[width=\textwidth,trim = 118 45 98 30]{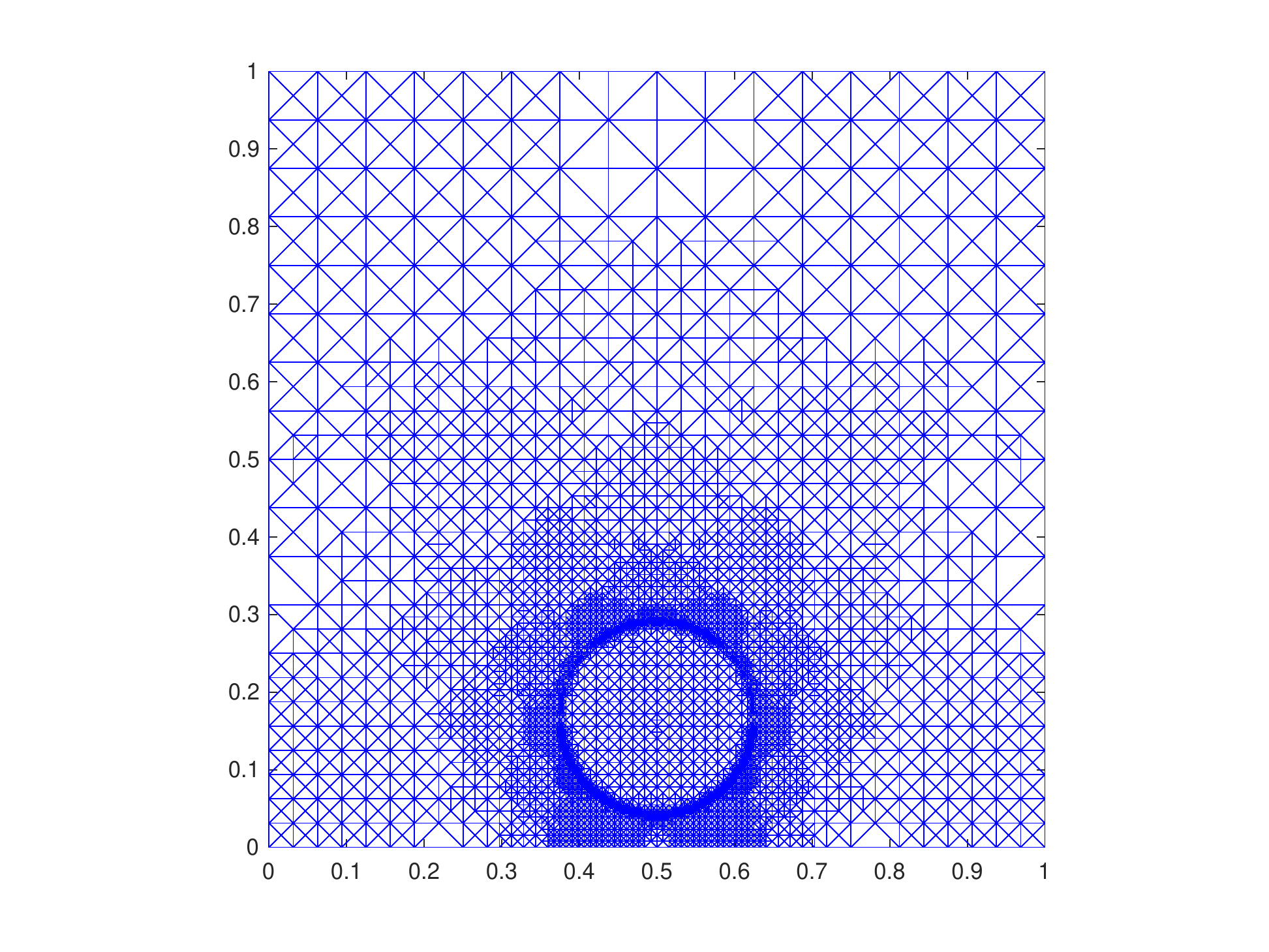}
\caption{$\nu = \eps_2$\\ $\#\TT_{\ell\nu} = \num{16420}$}
\end{subfigure}
\hfill
\begin{subfigure}{0.19\textwidth}
\centering
\includegraphics*[width=\textwidth,trim = 105 25 100 25]{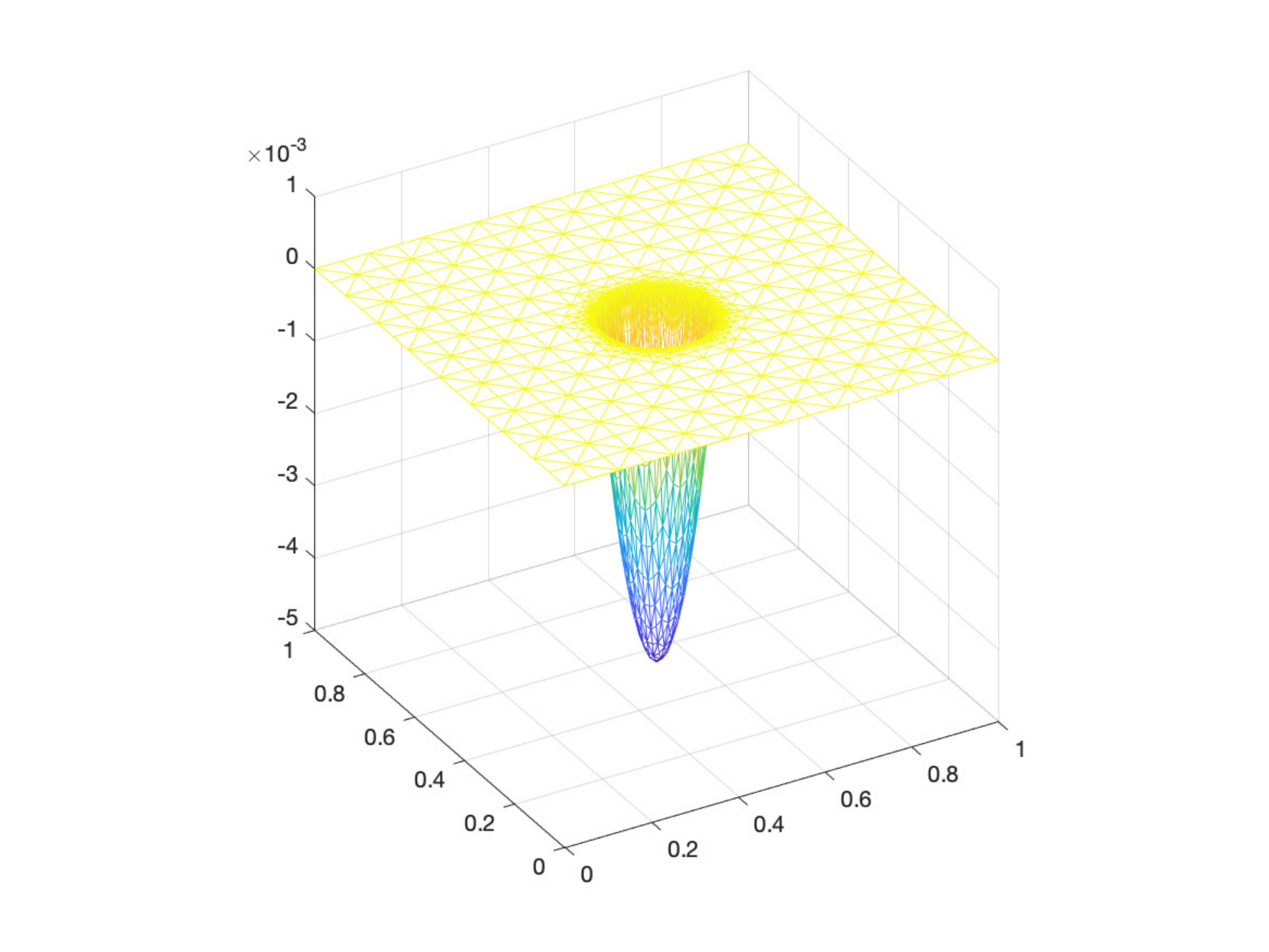}\\
\vspace{1mm}
\includegraphics*[width=\textwidth,trim = 118 45 98 30]{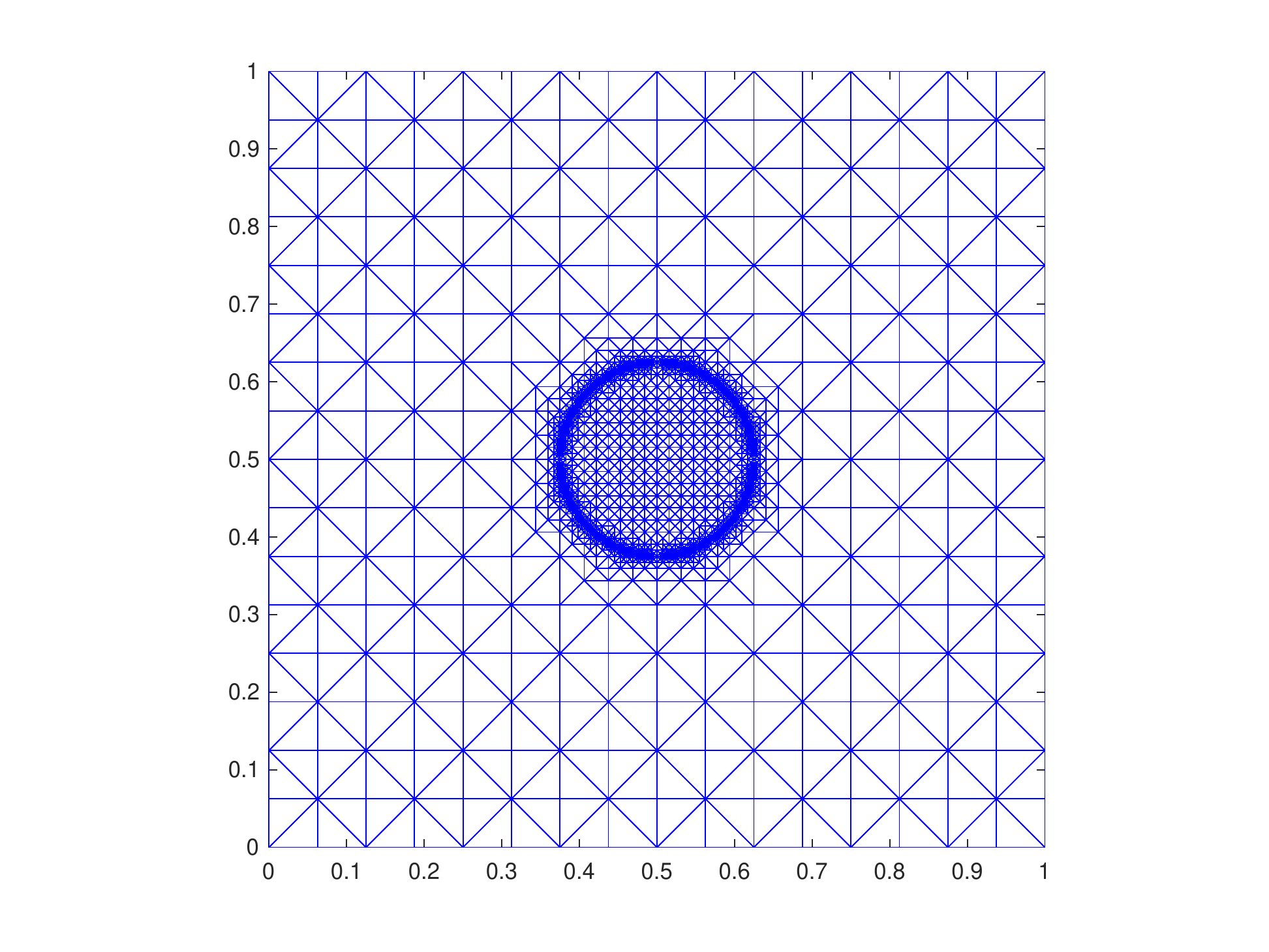}
\caption{$\nu = \eps_5$\\ $\#\TT_{\ell\nu} = \num{9528}$}
\end{subfigure}
\hfill
\begin{subfigure}{0.19\textwidth}
\centering
\includegraphics*[width=\textwidth,trim = 105 25 100 25]{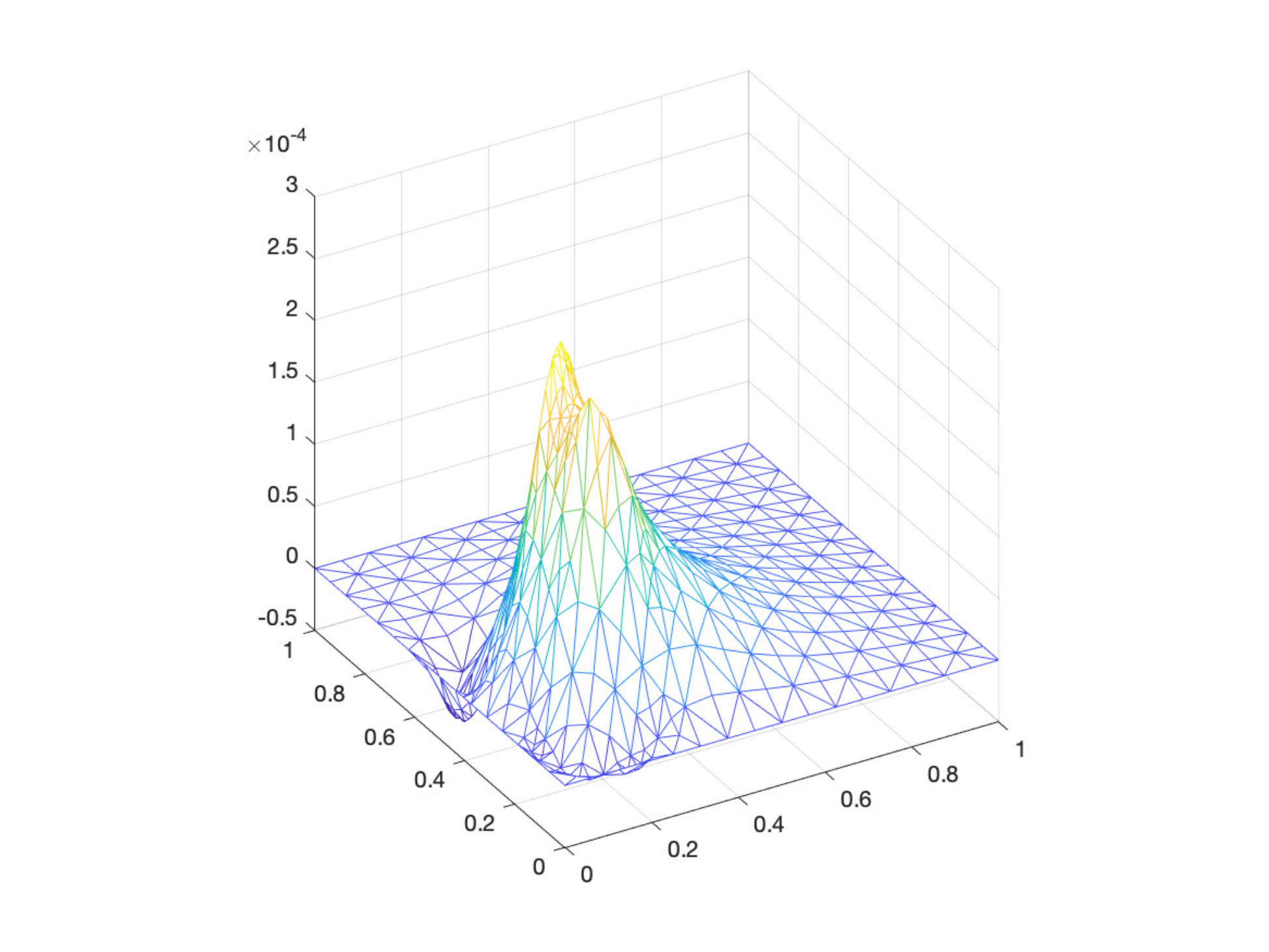}\\
\vspace{1mm}
\includegraphics*[width=\textwidth,trim = 118 45 98 30]{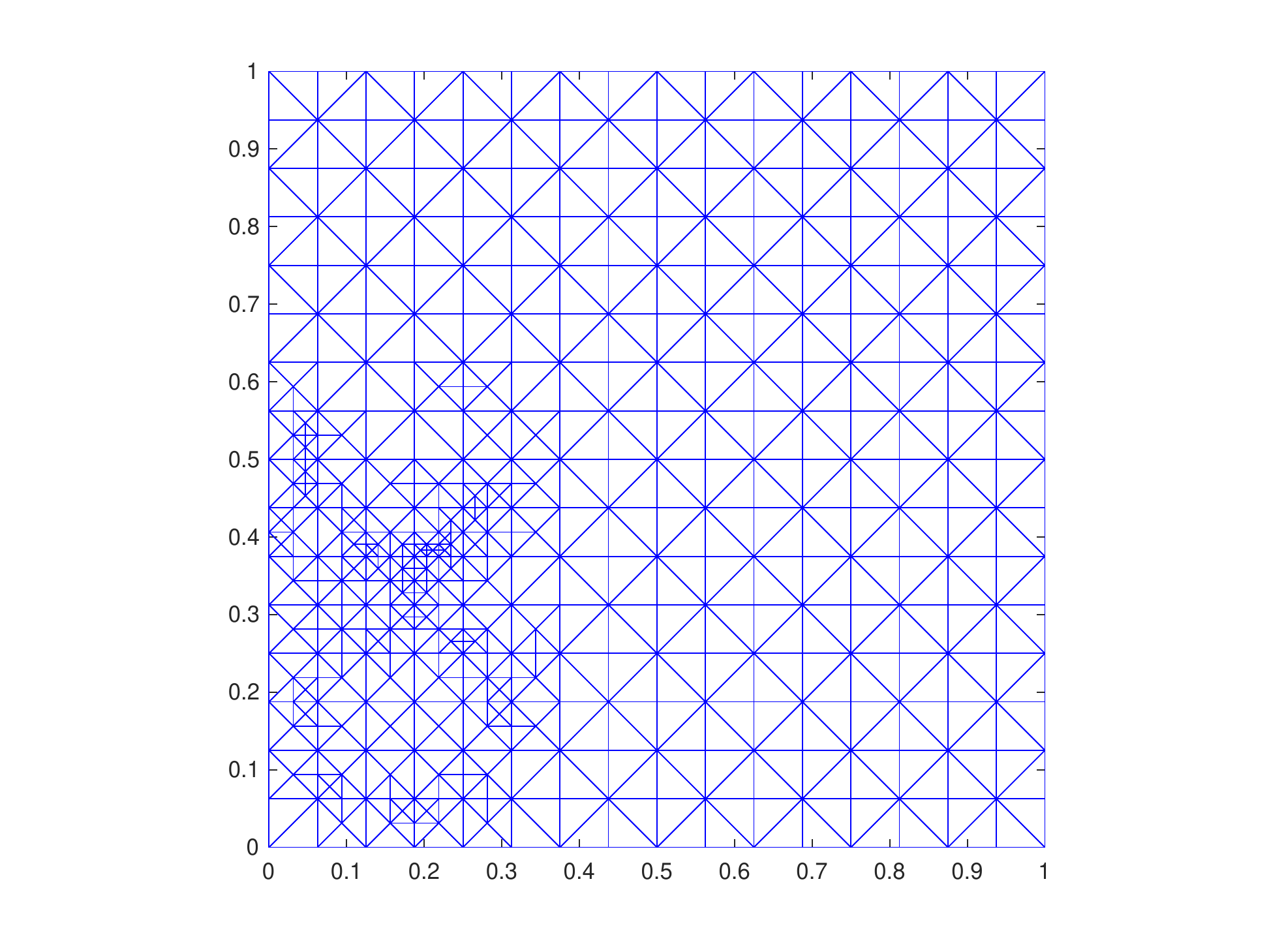}
\caption{$\nu = (1\ 0\ 0\ 1)$\\ $\#\TT_{\ell\nu} = \num{839}$}
\end{subfigure}
\caption{
Experiments in section~\ref{sec:cookie}:
Coefficients $u_{\ell\nu} \in \X_{\ell\nu} = \SS^1(\TT_{\ell\nu})$
of an intermediate SGFEM approximation ($\ell=16$) generated by \texttt{ML-C} (top plots)
and the associated adaptively refined meshes $\TT_{\ell\nu}$ (bottom plots)
for five indices $\nu \in \PPP_\ell$.
}
\label{fig:pb9:pictures}
\end{figure}

In Figure~\ref{fig:pb9:pictures},
we consider an intermediate SGFEM approximation $\uu_\ell  \in \V_\ell$ generated by~\texttt{ML-C}
($\ell= 16$).
For five indices in $\PPP_\ell$,
namely
$\nu = \0$,
three unit indices $\nu = \eps_1, \eps_2, \eps_5$,
and $\nu = (1\; 0\; 0\; 1)$,
we plot the coefficients $u_{\ell\nu} \in \X_{\ell\nu}$
and the associated adaptively refined meshes $\TT_{\ell\nu}$.
Note that
the coefficient associated with $\nu = \0$ rep\-res\-ents the expectation of the SGFEM approximation.
Looking at the mesh associated with $\nu = \0$,
we observe that the intensity of local mesh refinement at the boundary of each subdomain
reflects the `importance' of the corresponding parameter (cf.\ \eqref{eq:coeff_cookie}).
Moreover, we observe that for each $m=1,2,5$, the subdomain $D_m$ is
identified by the mesh associated with the index $\eps_m$.
In the same way, the mesh associated with $\nu = (1\; 0\; 0\; 1)$
identifies the subdomains $D_1$~and~$D_4$.

Next,
we consider the final index set $\PPP_L$ generated by \texttt{ML-C} ($L=21$)
and assess
the maximum polynomial degree
activated for each parameter $y_m$ ($m=1,\dots,9$):
\begin{equation*}
\max_{\nu \in \PPP_L} \nu_m
=
\begin{cases}
6 & \text{for } m=1,3,7,9,\\
9 & \text{for } m=2,4,6,8,\\
17 & \text{for } m=5.\\
\end{cases}
\end{equation*}
We see that the maximum polynomial degrees assigned to the parameters
mirror the hierarchy of the parameters induced by the coefficients (cf.\ \eqref{eq:coeff_cookie}).
This result, together with those reported in Figure~\ref{fig:pb9:pictures},
illustrate
the capability of our multilevel fully adaptive algorithm
to capture the anisotropy of the inclusions 
and allocate degrees of freedom according to the `importance' of
both the individual parameters and the gPC expansion modes.

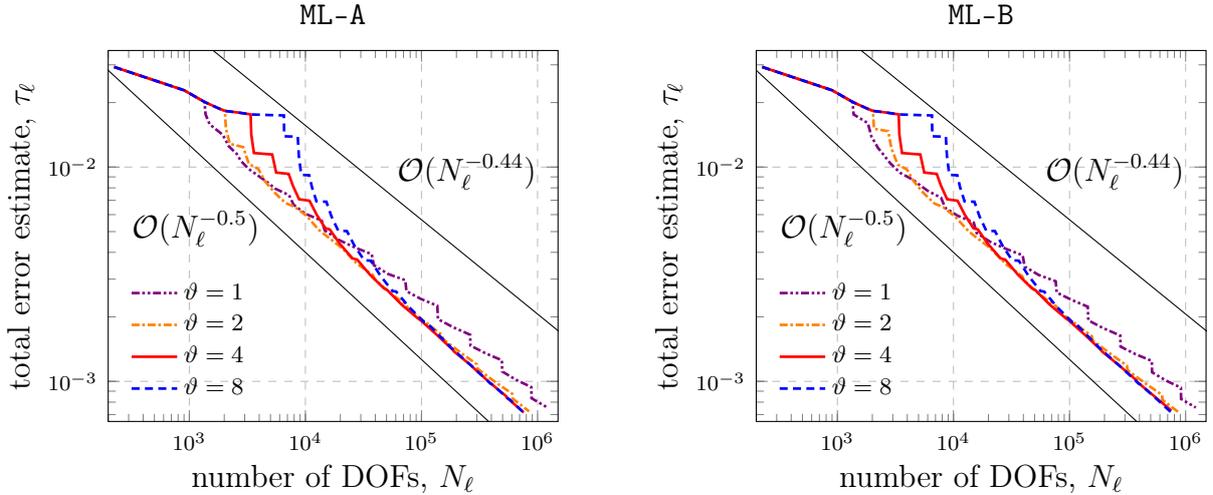
\begin{figure}[ht]
\begin{tikzpicture}
\pgfplotstableread{data/pb9-ml-a.dat}{\one}
\pgfplotstableread{data/pb9-ml-a-vartheta2.dat}{\two}
\pgfplotstableread{data/pb9-ml-a-vartheta4.dat}{\four}
\pgfplotstableread{data/pb9-ml-a-vartheta8.dat}{\eight}
\begin{loglogaxis}
[
title=\texttt{ML-A},
width = 7.5cm, height=6.5cm,						
xlabel={number of DOFs, $N_\ell$}, 					
ylabel={total error estimate, $\est_\ell$},				
ymajorgrids=true, xmajorgrids=true, grid style=dashed,	
xmin = (2.0)*10^(2),
xmax = (1.5)*10^(6),
ymin = (6.5)*10^(-4),
ymax = (3.5)*10^(-2),
legend style={legend pos=south west, legend cell align=left, fill=none, draw=none}
]
\addplot[violet,line width=1.0pt, densely dash dot dot]	table[x=dofs, y=error]{\one};
\addplot[orange,line width=1.0pt, densely dash dot]		table[x=dofs, y=error]{\two};
\addplot[red,line width=1.0pt]						table[x=dofs, y=error]{\four};
\addplot[blue,line width=1.0pt, densely dashed]			table[x=dofs, y=error]{\eight};
\addplot[black,solid,domain=10^(2.2):10^(6.8)] { 0.9*x^(-0.44) };
\node at (axis cs:5e4,7.0e-3) [anchor=south west] {$\mathcal{O}(N_\ell^{-0.44})$};
\addplot[black,solid,domain=10^(2.2):10^(6.0)] { 0.4*x^(-0.5) };
\node at (axis cs:0.5e4,7.0e-3) [anchor=north east] {$\mathcal{O}(N_\ell^{-0.5})$};
\legend{
$\vartheta= 1$,
$\vartheta= 2$,
$\vartheta= 4$,
$\vartheta= 8$,
}
\end{loglogaxis}
\end{tikzpicture}
\hfill
\begin{tikzpicture}
\pgfplotstableread{data/pb9-ml-b.dat}{\one}
\pgfplotstableread{data/pb9-ml-b-vartheta2.dat}{\two}
\pgfplotstableread{data/pb9-ml-b-vartheta4.dat}{\four}
\pgfplotstableread{data/pb9-ml-b-vartheta8.dat}{\eight}
\begin{loglogaxis}
[
title=\texttt{ML-B},
width = 7.5cm, height=6.5cm,						
xlabel={number of DOFs, $N_\ell$}, 					
ylabel={total error estimate, $\est_\ell$},				
ymajorgrids=true, xmajorgrids=true, grid style=dashed,	
xmin = (2.0)*10^(2),
xmax = (1.5)*10^(6),
ymin = (6.5)*10^(-4),
ymax = (3.5)*10^(-2),
legend style={legend pos=south west, legend cell align=left, fill=none, draw=none}
]
\addplot[violet,line width=1.0pt, densely dash dot dot]	table[x=dofs, y=error]{\one};
\addplot[orange,line width=1.0pt, densely dash dot]		table[x=dofs, y=error]{\two};
\addplot[red,line width=1.0pt]						table[x=dofs, y=error]{\four};
\addplot[blue,line width=1.0pt, densely dashed]			table[x=dofs, y=error]{\eight};
\addplot[black,solid,domain=10^(2.2):10^(6.8)] { 0.9*x^(-0.44) };
\node at (axis cs:5e4,7.0e-3) [anchor=south west] {$\mathcal{O}(N_\ell^{-0.44})$};
\addplot[black,solid,domain=10^(2.2):10^(6.0)] { 0.4*x^(-0.5) };
\node at (axis cs:0.5e4,7.0e-3) [anchor=north east] {$\mathcal{O}(N_\ell^{-0.5})$};
\legend{
$\vartheta= 1$,
$\vartheta= 2$,
$\vartheta= 4$,
$\vartheta= 8$,
}
\end{loglogaxis}
\end{tikzpicture}
\caption{
Experiments in section~\ref{sec:cookie}:
Total error estimates $\est_\ell$ versus the number of degrees of freedom $N_\ell$
for all \texttt{ML-A} and \texttt{ML-B} and different values of $\vartheta$.
}
\label{fig:pb9-rates-vartheta}
\end{figure}

In our final experiment,
we investigate whether
appropriately selecting 
the parameter $\vartheta >0$,
which modulates the choice between mesh refinement and parametric enrichment,
can lead to a decay of the error estimate
with fully optimal rate~$\mathcal{O}(N_\ell^{-0.5})$
also for \texttt{ML-A} and \texttt{ML-B}.
In Figure~\ref{fig:pb9-rates-vartheta},
we compare the decay of the error estimates $\est_\ell$
obtained for $\vartheta=1,2,4,8$.
We observe that each choice $\vartheta >1$
leads to a significant improvement of the convergence rate,
which is optimal for $\vartheta=4,8$.
This behavior is in agreement with the results obtained for \texttt{ML-C}
presented in Figure~\ref{pb5:multilevel_decay} and Figure~\ref{fig:pb9:multilevel_decay},
where we see that the combined marking strategy
automatically favors parametric enrichments over spatial refinements.

\subsection{Conclusions on numerical experiments} \label{sec:conclusions}
Overall, the reported results of numerical experiments indicate that:

$\bullet$
the proposed error estimation strategy in the context of the multilevel  SGFEM is as effective as the error estimators
for single-level and multilevel SGFEMs investigated in~\cite{bprr18++} and~\cite{cpb18+}, respectively;

$\bullet$
for the considered test problems, adaptive multilevel SGFEM outperforms its single-level counterpart in terms of
convergence rates and in terms of the number of degrees of freedom required to reach the prescribed tolerance;
this is a consequence of a greater flexibility of the multilevel SGFEM in allocating degrees of freedom compared to the single-level SGFEM;

$\bullet$
the error estimates for multilevel SFGEM approximations generated by the algorithm with combined marking/enrichment
(Algorithm~\ref{algorithm}.\ref{marking:C}) decay with the optimal rate;
on the other hand, the optimal decay rate for approximations generated by the algorithms with separate marking/enrichment
(Algorithms~\ref{algorithm}.\ref{marking:A} and~\ref{algorithm}.\ref{marking:B}) can be ensured by prioritizing
parametric enrichments over spatial refinements
(by setting $\vartheta > 1$ in the associated marking criterion);

$\bullet$
all adaptive algorithms proposed in this paper are effective in identifying the most `important' modes in the gPC expansion
of the solution to the parametric problem,
including the case of infinitely many parameters (as in the test problem in~\S\ref{sec:eigel})
and the case when `importance' of parameters cannot be directly inferred from
the ordering of terms in the coefficient expansion (as in the test problem in~\S\ref{sec:cookie}).

The application of our algorithms to other classes of parametric PDE problems
(e.g., the problems with non-affine coefficient expansions
in terms of a finite number of bounded parameters) is possible
(see, e.g.,~\cite{bx2020} for adaptive single-level SGFEM).
However, for more challenging  problems (e.g., the problems with lognormal parametric coefficients),
the efficiency of the algorithms will significantly benefit
from combining adaptivity with compression techniques (e.g., low-rank tensor methods~\cite{dklm2015}),
as developed recently in~\cite{emps2020} in the context of the single-level SGFEM.
The extension of this methodology to adaptive multilevel SGFEM approximations will be considered in future~research.

\bibliographystyle{alpha}
\bibliography{literature}

\end{document}